\documentclass[onefignum,onetabnum]{siamart190516}

\usepackage{amssymb}

\usepackage{subfig}
\usepackage{amsfonts}
\usepackage{enumerate}

\usepackage{mathtools}
\DeclarePairedDelimiter\N{\|}{\|}
\DeclareMathOperator{\bigO}{\mathcal{O}}
\usepackage{tikz-cd}

\newsiamremark{remark}{Remark}
\newsiamremark{example}{Example}

\DeclareMathOperator{\diag}{diag}
\DeclareMathOperator{\Span}{span}
\DeclareMathOperator{\T}{\sf T}

\allowdisplaybreaks
\headers{Decoupled SDA with Truncation}{Z.-C. Guo,  E.K.-W. Chu, X. Liang and W.-W. Lin}

\title{Decoupled Structure-Preserving Doubling Algorithm with 
Truncation for Large-Scale Algebraic Riccati Equations}

\author{
Zhen-Chen Guo\thanks{Department of Mathematics, Nanjing University,  Nanjing  210093,  China 
(\email{guozhenchen@nju.edu.cn}).} \and 
Eric King-Wah Chu\thanks{School of Mathematics,  Monash University, 9 Rainforest Walk, Victoria 3800, Australia 
(\email{eric.chu@monash.edu}).} \and
Xin Liang\thanks{Yau Mathematical Sciences Center, Tsinghua University, Beijing 10084, China
(\email{liangxinslm@tsinghua.edu.cn}).} \and 
Wen-Wei Lin\thanks{Department of Applied Mathematics, National Chiao Tung University, Hsinchu 300, Taiwan 
(\email{wwlin@math.nctu.edu.tw}).}
}

\begin{document}
\maketitle

\begin{abstract} 
In~\cite{guoCLL2020decoupled} we propose a decoupled form of the structure-preserving doubling algorithm (dSDA). 
The method decouples the original two to four coupled recursions,  
enabling it to solve large-scale algebraic Riccati equations and other related problems. In this paper, 
we   consider the   numerical computations of the novel dSDA for solving  large-scale  
continuous-time  algebraic Riccati equations with low-rank structures 
(thus possessing numerically low-rank solutions). 
With the help of a new truncation strategy, 
the rank of the approximate solution is  controlled.  
Consequently, large-scale problems can be treated efficiently.  
Illustrative numerical examples are presented to demonstrate and confirm our claims. 
\end{abstract}

\begin{keywords}
continuous-time algebraic Riccati equation, decoupled structure-preserving doubling algorithm, 
large-scale problem,  truncation  
\end{keywords}

\begin{AMS}
	15A24, 65F30, 65H10 
\end{AMS}

\section{Introduction} 
A continuous-time algebraic Riccati equation (CARE) has the form:
\begin{equation}\label{eq:care}
	A^{\T} X + XA - XGX + H =0,
\end{equation}
where $A \in \mathbb{R}^{n \times n}$, $G=B R^{-1}B^{\T}$ with  
$B \in \mathbb{R}^{n \times m}$  and $R>0$,  
and  $H=C^{\T}C\ge 0$ with $C \in \mathbb{R}^{l \times n}$. 
Here, a symmetric matrix $M > 0$ ($\geq 0$) when all its eigenvalues are positive (non-negative). 
These algebraic Riccati equations arise  in  many classical applications 
such as model reduction, filtering  and  control  theory; 
please refer to~\cite{biniIM2012numerical, choiL1990efficient, datta1994linear, 
gawronski1998dynamics, lancasterR1995algebraic, mehrmann1991automomous} 
and the references therein.  Generally, the CARE~\eqref{eq:care} admits more than 
one solutions~\cite{lancasterR1995algebraic, mehrmann1991automomous} if exist. 
However, the unique symmetric positive semi-definite   solution ($X\geq 0$) 
is required for applications \cite{lancasterR1995algebraic, mehrmann1991automomous}. 

The research on the numerical solution of CAREs  has been active, due to its practical  importance.  
Many  engineers and applied mathematicians worked on the topic, 
contributed many methods \cite{biniIM2012numerical, lancasterR1995algebraic, mehrmann1991automomous}. 
For CAREs of moderate sizes, classical approaches apply canonical forms, 
determinants and polynomial manipulation while state-of-the-art ones compute in a numerically stable manner; 
see \cite{chuFL2005structurepreserving, laub1979schur, laub1991invariant, moore1985computational}. 
One favourite approach reformulates the CARE as an algebraic eigenvalue problem \cite{laub1979schur} for the associated 
Hamiltonian matrix  $\mathcal{H} \equiv \begin{bmatrix} \hphantom{-}A & -G\hphantom{^{\T}} \\ -H & -A^{\T} \end{bmatrix}$; 
see the command \texttt{care}  in MATLAB \cite{mathworks2010matlab}. 
The other favourite is the structure-preserving doubling algorithm (SDA) \cite{chuFL2005structurepreserving}, 
which approximates the solution via the stable invariant subspace of $\mathcal{H}$.  

As for large-scale CAREs, they have attracted much attention recently 
\cite{amodeiB2010invariant, bansch2012stabilization, 
bennerBKS2018radi, benner2005dimension, bennerS2010newtongalerkinadi, 
bennerS2013numerical, heyouniJ2009extended, jbilou2003block, jbilou2006arnoldi, 
liCLW2013solving, saakHB2010matrix, saakKB2016mmess}. 
Solving CAREs may involve the invariant subspace of the Hamiltonian matrix 
$\mathcal{H}$,  an expensive exercise when computed directly. 
Several authors~\cite{amodeiB2010invariant, bennerBKS2018radi, massoudiOT2016analysis} 
focus on implicitly manipulating the invariant subspace. 
Benner and his collaborators have contributed heavily on the solution 
of large-scale CAREs \cite{bennerHSW2016inexact, bennerLP2008numerical, 
bennerS2013numerical, saakHB2010matrix,  saakKB2016mmess}, 
based on Newton's methods with ADIs for the associated Lyapunov equations. 
One of these efficient methods is  the low-rank Newton-Kleinman ADI method~\cite{saakKB2016mmess}.    
Based on the Cayley iteration,   
the authors in~\cite{bennerBKS2018radi} proposed a RADI method for computing 
the invariant subspaces of the residual equations,   
accumulating  some matrices generated to construct the  approximate solution.  
There are some difficulties in the initial stabilization of the Newton-Kleinman ADI method 
and the choice of parameters for the ADI is mostly by heuristics. 
Another popular approach is the Krylov subspace or projection methods  \cite{heyouniJ2009extended, jbilou2003block, jbilou2006arnoldi, simoncini2016analysis, simonciniSM2014two}. Solvability of the projected equations has to be assumed.  

Although  efficient for  CAREs of moderate sizes, the original SDA 
(which is globally and quadratically convergent except for the critical case~\cite{linX2006convergence}) 
does not work well for large-scale problems.   
The method has three coupled recursions and the corresponding matrix inversions 
lead to a computational complexity of $\bigO(n^3)$. For large-scale problems,  
one of those recursions has to be applied {\it implicitly\/} because of its loss of structures, 
leading to inefficiency.  
In~\cite{guoCLL2020decoupled} we developed  the dSDA, which decouples  the original 
three recursions. The dSDA retains the solid theoretical foundation of the SDA, 
for its global quadratic convergence. %  (linear for the critical case).  

In this paper, we further develop the dSDA in depth,   considering  the practical  
computational issues for large-scale CAREs. To control the rank 
of the approximate solutions, 
a novel truncation strategy is proposed. The practical dSDA$\rm{_t}$ 
(with the subscript indicating truncation) is efficient for large-scale CAREs. 
A detailed analysis verifies  the convergence of the dSDA$\rm{_t}$. 
Illustrated  numerical examples are presented. 

\subsection*{Main Contributions}
\begin{enumerate}[(1)]
\item  We develop  a novel truncation technique in the dSDA$\rm{_t}$, 
preserving  the  simple but elegant form of the dSDA. 
	As a result,  for large-scale CAREs, we need not  
	compute $A_k$ (as in the original SDA) recursively, 
	thus eliminating the $2^k$ factor in the flop count and improving the efficiency. 
	We are only required to compute $H_k$ with a simple formula. 

\item To further improve the  algorithm, 
	we 	combine the doubling and truncation   
	into a nontrivial but more efficient step.

\item  For many other methods for large-scale CAREs,  
	it is assumed that the desired   
	solution is numerically low-rank. 
	From our derivation, 	we explicitly  show that the approximate 
	solutions are low-rank.   Similarly, we do not need to assume 
	the solvability of projected equations, nor we have  any problems 
	in initial stabilization or choosing parameters.

\item 	For numerical stability, much of our effort  
	involves the proof of convergence for the dSDA$\rm{_t}$. 
	We construct some seemingly tedious but concise expressions 
	of the approximate solutions. 

\end{enumerate}

\subsection*{Organization} 
After some preliminaries in Section~\ref{sec:preliminaries}, 
we   construct the truncation strategy for the  dSDA inductively 
in Section~\ref{sec:computaiton-issues}. 
We show the truncation process for the first two steps  in detail.  
Error analysis and convergence proof   are presented in Section~\ref{sec:forward-error-analysis-for-dsda} 
and illustrative numerical examples are presented in Section~\ref{sec:numerical-examples}, 
before we conclude in Section~\ref{sec:conclusions}.  
Appendices~\ref{sec:proof-of-lemma-lm} and~\ref{sec:proof-of-lemma-ref-lm-doubling_k1}
contain two complicated proofs, for the combined  doubling-truncation 
step in Section~\ref{sec:computaiton-issues} 
and the convergence analysis in Section~\ref{sec:forward-error-analysis-for-dsda}, respectively. 

\subsection*{Notations}
By $\mathbb{R}^{n\times n}$ we denote the set of all $n\times n$ real matrices, 
with $\mathbb{R}^n = \mathbb{R}^{n\times 1}$ and $\mathbb{R}=\mathbb{R}^{1}$;  
$\mathbb{S}^n$ denotes the subset symmetric matrices in $\mathbb{R}^{n \times n}$. 
The $n\times n$ identity matrix is $I_n$ and we write $I$ if its dimension is clear.  
The zero matrix is $0$ and  the superscript $(\cdot)^{\T}$ takes the transpose. 
By $M \oplus N$, we denote $\begin{bmatrix}
	M & 0 \\ 0 & N
\end{bmatrix}$, and  $M \otimes N$ is the Kronecker product of 
the  two matrices $M$ and $N$. The inequality $\Phi \le \Psi$ 
holds if and only if $\Psi - \Phi \ge 0$,  
% means $\Psi - \Phi$ is a symmetric positive semi-definite   matrix, 
and similarly for $\Phi < \Psi$, $\Phi \ge \Psi$ and $\Phi > \Psi$. 
The $2$- and  Frobenius norms are  denoted by $\| \cdot\|$ and  $\| \cdot\|_F$, respectively.  

\section{Preliminaries}\label{sec:preliminaries}
The discrete-time algebraic Riccati equation (DARE), analogous to the CARE, is in the form of 
\begin{equation}\label{eq:dare}
-X + A^{\T} X (I + GX)^{-1} A + H = 0.  
\end{equation} 
For solvability, we  assume that both CAREs and DAREs are stabilizable and detectable. 
We shall also assume without loss of generality that $B$ and $C^{\T}$ are of  full column rank 
with $m,l\ll n$  and  $R = I_m$. The DARE  admits many solutions but only the unique 
symmetric positive semi-definite   solution is  of practical interest.

Write $A_c:=(I+GX)^{-1}A$, where $A, G$ and $X$ are specified in \eqref{eq:dare}, 
and define the linear operator 
$\mathcal{L}: \mathbb{S}^n \rightarrow \mathbb{S}^n$ by 
	$\mathcal{L}(\Phi) =\Phi - A_c^{\T} \Phi A_c$, 	
which is invertible when $A_c$ is d-stable (with eigenvalues strictly 
inside the unit circle; see~\cite{lancasterR1995algebraic}).  
Define 
\begin{equation}\label{eq:lxieta}
	\begin{aligned}
		\ell&:=\|\mathcal{L}^{-1}\|^{-1}=\min_{\Phi\in \mathbb{S}^n, \|\Phi\|=1} 
		\| \Phi -  A_c^{\T} \Phi A_c\|, 
		\\
		\xi&:= \max_{\Phi\in \mathbb{R}^{n\times n}, \|\Phi\|=1} 
		\left\|\mathcal{L}^{-1}\left[ A^{\T} (I+XG)^{-1}X \Phi + \Phi^{\T} X(I+GX)^{-1}A\right]\right\|, 
		\\
		\eta&:= \max_{\Phi\in \mathbb{S}^n, \|\Phi\|=1} 
		\left\|\mathcal{L}^{-1}\left[ A^{\T} (I+XG)^{-1}X \Phi X(I+GX)^{-1}A\right]\right\|.
	\end{aligned}
\end{equation}
Let $\widetilde A = A+\Delta A$, $\widetilde G = G+\Delta G$ and $\widetilde  H = H+\Delta H$ 
and consider the perturbed DARE: 
\begin{align}\label{eq:dare_perturbed}
	-\widetilde X + \widetilde{A}^{\T} \widetilde  X (I+ \widetilde G \widetilde X)^{-1} \widetilde{A}  + \widetilde H=0.  
\end{align}
With 
\begin{small}
\begin{align*}
	\delta&:=\frac{\|\Delta A\| + \|X(I+GX)^{-1}A\| \|\Delta G\|}{1-\|X(I+GX)^{-1}\|\|\Delta G\|},
	\quad 
	\alpha:=\frac{\|(I+GX)^{-1}\|(\|A\| + \|\Delta A\|)}{1-\|X(I+GX)^{-1}\| \|\Delta G\|},
	\\
	g&:= \frac{\|(I+GX)^{-1}\|(\|G\| + \|\Delta G\|)}{1-\|X(I+GX)^{-1}\| \|\Delta G\|},
\end{align*}
\end{small}
we have the following result.  
\begin{lemma}\cite[Theorem~4.1]{sun1998perturbation} \label{lm:forward_error_dare}
Let $X$ be the unique symmetric positive semi-definite  solution 	to the DARE \eqref{eq:dare} 
and 
\begin{small}
\begin{gather*}
	\omega:= \frac{\|\Delta H\|}{\ell} + \xi \|\Delta A\| + \eta \|\Delta G\| 
	+ \frac{\delta \|X(I+GX)^{-1}\| }{\ell}(\|\Delta A\| + \|X(I+GX)^{-1}A\| \|\Delta G\|),
	\\
	\zeta:= \delta \|(I+GX)^{-1}\|  \left(2\|(I+GX)^{-1}A\| + \delta \|(I+GX)^{-1}\|\right),
	\\
	\theta:= \frac{2\ell\omega}
	{\ell-\zeta+ \ell g \omega +\sqrt{(\ell-\zeta + \ell g \omega)^2-4\ell g\omega(\ell -\zeta + \alpha^2)}}.
\end{gather*}
\end{small}
If $\widetilde G\geq 0, \widetilde H \geq 0$, 
$\|X(I+GX)^{-1}\|\|\Delta G\|<1$, $g \theta<1$ and  
\begin{small}
\begin{align*}
	\frac{\delta \|(I+GX)^{-1}\| + g \theta \|(I+GX)^{-1}A\|}{1-g\theta} 
	&< 
	\frac{\ell}{\|(I+GX)^{-1}A\| + \sqrt{\ell + \|(I+GX)^{-1}A\|^2}}, 
	\\
	\omega& <\frac{(\ell - \zeta)^2}
	{\ell g \left(\ell -\zeta + 2\alpha + \sqrt{(\ell -\zeta + 2\alpha)^2-(\ell - \zeta)^2}\right)}, 
\end{align*}
\end{small}
then the perturbed DARE \eqref{eq:dare_perturbed} has a unique symmetric 
positive semi-definite solution $\widetilde X$ with the error   
$\|\widetilde X-X\|\leq \theta$. 
\end{lemma}

\begin{remark}\label{rk:forward-error-dare}
Lemma~\ref{lm:forward_error_dare} suggests a first-order perturbation bound for the solution $X$: 
\[
\|\widetilde X- X\|\leq \frac{1}{\ell}\|\Delta H\| + \xi\|\Delta A\| + \eta \|\Delta G\| 
+ \bigO(\|(\Delta H, \Delta A, \Delta G)\|^2)
\]
as $\|(\Delta H, \Delta A, \Delta G)\| \to 0$, leading to 
\[
	\frac{\|\widetilde X-X\|}{\|X\|} 
	\lesssim \frac{1}{\ell}\frac{\|H\|}{\|X\|}\frac{\|\Delta H\|}{\|H\|} 
	+ \xi \frac{\|A\|}{\|X\|} \frac{\|\Delta A\|}{\|A\|} 
	+ \eta \frac{\|G\|}{\|X\|} \frac{\|\Delta G\|}{\|G\|} 
\]
for sufficiently small $\|(\Delta H, \Delta A, \Delta G)\|$, 
with ``$\lesssim$'' denoting ``$\leq$'' while ignoring the $\bigO$-term.
\end{remark}

In the following we sketch the SDA and dSDA for CAREs. 
Define   $A_\gamma := A - \gamma I$ and 
$K_\gamma := A_\gamma^{\T} + H A_\gamma^{-1} G$, 
which are nonsingular for some  parameter $\gamma >0$. Let 
\begin{equation*}
	A_0 = I_n + 2\gamma K_\gamma^{-{\T}}, \ \ 
	G_0 = 2\gamma A_\gamma^{-1} G K_\gamma^{-1}, \ \ 
	H_0 = 2\gamma K_\gamma^{-1} H A_\gamma^{-1}. 
\end{equation*}
Assuming that $I_n + G_k H_k$ are nonsingular for $k=0, 1, \cdots$, 
the SDA  has three iterative recursions: 
\begin{equation}\label{eq:sda}
\begin{aligned}
	A_{k+1} &= A_k (I_n + G_k H_k)^{-1} A_k, \ \ \ 
	G_{k+1} = G_k + A_k (I_n + G_k H_k)^{-1}  G_k A_k^{\T}, \\
	H_{k+1} &= H_k + A_k^{\T} H_k(I_n + G_k H_k)^{-1} A_k.  
\end{aligned}	
\end{equation}
For the SDA~\eqref{eq:sda}, we have $A_k \to 0$, $G_k \to Y$ 
(the solution to the dual CARE: $AY + Y A^{\T} - YHY + G=0$) and 
$H_k \to X$, all quadratically except for the critical case where the convergence is linear. 
It is  worthwhile to point out that the DARE shares the same SDA formulae \eqref{eq:sda}, 
with the alternative starting points   
$A_0:=A,  G_0 := G$, and $H_0 :=  H$. 

It is worth noting that $I + G_k H_k$ are generically nonsingular. 
Several remedies to avoid singularity are available, 
such as adjusting the shift $\gamma$ appropriately, 
or the double-Cayley transform~\cite{guoCL2019doubling}.  
We shall assume this nonsingularity for the rest of the paper and 
leave the research into the remedies to the future. 

Denote  $\widetilde A_{\gamma}:= A_\gamma^{-1} A_{-\gamma}= I+2\gamma A_\gamma^{-1}$, 
then we have the following results for the dSDA.  

\begin{lemma}[dSDA for CAREs]\label{thm:CAREiterationk}
	Let $U_0 = A_\gamma^{-1} B$, $V_0 = A_\gamma^{-{\T}} C^{\T}$. 	
	Denote $U_j := \widetilde A_{\gamma} U_{j-1}$ and 
	$V_j := \widetilde A_{\gamma}^{{\T}} V_{j-1}$ for $j\geq 1$. 
	For all $k\ge 1$, the SDA produces the following decoupled form   
	\begin{equation}\label{eq:care-iteration}
		\begin{aligned}
			A_k
			&=
			\widetilde A_{\gamma}^{2^k}-2\gamma 
			\breve{U}_k
			\left(I_{2^k m}+Y_kY_k^{\T}\right)^{-1}Y_k
			\breve{V}_k^{\T},  
			\\
			G_k
			&=
			2\gamma 
			\breve{U}_k
			\left(I_{2^k m}+Y_kY_k^{\T}\right)^{-1}
			\breve{U}_k^{\T}, 
			\ \ \ 
			H_k
			=
			2\gamma 
			\breve{V}_k
			\left(I_{2^k l}+Y_k^{\T} Y_k\right)^{-1}
			\breve{V}_k^{\T}, 
		\end{aligned}	
	\end{equation}
	where $\breve{U}_k := [U_0,U_1,\cdots, U_{2^k-1}]$, 
	$\breve{V}_k :=[V_0,V_1,\cdots, V_{2^k-1}]$,  $Y_k=\begin{bmatrix}
		0&Y_{k-1}\\Y_{k-1}&2\gamma T_{k-1}
	\end{bmatrix}\in \mathbb{R}^{2^k m \times 2^k l}$  
	and  
	$T_{k}=\breve{U}_k^{\T} \breve{V}_k$, with   $Y_0 = B^{\T} A_\gamma^{-{\T}} C^{\T}$   
	and $T_0 = U_0^{\T} V_0$. 
\end{lemma}

The three formulae in \eqref{eq:care-iteration} are decoupled. To solve CAREs, 
it is sufficient to iterate with $H_k$ and monitor 
$\|H_k-H_{k-1}\|$ or the normalized  residual  for convergence, ignoring $A_k$ and $G_k$. 

From Lemma~\ref{thm:CAREiterationk}, the dSDA is clearly related to the projection method  with the Krylov subspace spanned  by the columns of 
$\breve{V}_k$. As it is well-known that 
Krylov subspaces lose their linear independence as their dimensions grow, 
it is common to truncate their bases, 
or eliminate the insignificant components. This controls any unnecessary growths in the rank of 
the approximate solutions, thus improving the efficiency of  the computation, while sacrificing 
a negligible amount of accuracy. In addition, the kernel $2\gamma (I_{2^k l}+Y_k^{\T} Y_k)^{-1}$ of the 
approximation in~\eqref{eq:care-iteration}, as the solution of the projected CARE, 
will deteriorate in condition as $k$ grows. This condition may be improved by limiting the rank of 
$\breve{V}_k$. The main results of our paper concern the truncation in the dSDA$\rm{_t}$, 
which is described in details in the next section. 

\section{Computational Issues}\label{sec:computaiton-issues} 

This section is dedicated to  the truncation of  
 $H_k$ (or $G_k$, if desired), which will be kept low-rank. 
 We first outline the whole  truncation process  in Figure~\ref{fig:truncation-routine:} 
(for $G$ only and that for $H$ is similar).    
From the initial $G_0$, the dSDA yields $G_1$ which is truncated to $G_1^{(1)}$. 
This in turn is processed by the dSDA to produce $G_2^{(1)}$ which is truncated to $G_2^{(2)}$. 
Recursively, at stage $k$ in the doubling-truncating step, $G_k^{(k)}$, 
the result of the truncation from  $G_{k}^{(k-1)}$, 
produces  $G_{k+1}^{(k)}$ by the dSDA and then we  truncate $G_{k+1}^{(k)}$  to obtain $G_{k+1}^{(k+1)}$. 
In other words, the subscripts are the indices in the dSDA and the superscripts are from the truncation. 

\newcommand\red[1]{{\color{red}#1}}
\def\trunc{\text{truncation}}
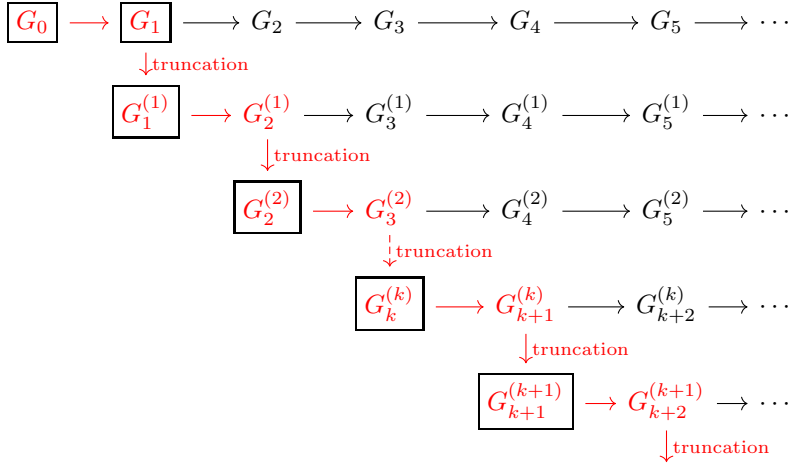
\begin{figure}[ht]
	\centering
	\begin{tikzcd}[sep=small]
		\boxed{\red{G_0}} \arrow[r, red]	&\boxed{\red{G_1}} \arrow[r] \arrow[d, red, "\trunc" red]  & G_2 \arrow[r]                                 & G_3 \arrow[r]                                 & G_4 \arrow[r]                                 & G_5 \arrow[r]             & \cdots 
		\\
		&	\boxed{\red{G_1^{(1)}}} \arrow[r, red]                & \red{G_2^{(1)}} \arrow[r] \arrow[d, red, "\trunc" red] & G_3^{(1)} \arrow[r]                           & G_4^{(1)} \arrow[r]                           & G_5^{(1)} \arrow[r]       & \cdots 
		\\
		&	                                         & \boxed{\red{G_2^{(2)}}} \arrow[r, red]                     & \red{G_3^{(2)}} \arrow[r] \arrow[d, dashed, red, "\trunc" red] & G_4^{(2)} \arrow[r]      & G_5^{(2)} \arrow[r]       & \cdots 
		\\
		&	                                         &                                               & \boxed{\red{G_k^{(k)}}} \arrow[r, red]                     & \red{G_{k+1}^{(k)}} \arrow[r] \arrow[d, red, "\trunc" red] & G_{k+2}^{(k)} \arrow[r]       & \cdots 
		\\
		&	                                         &                                               &                                               & \boxed{\red{G_{k+1}^{(k+1)}}} \arrow[r, red] & \red{G_{k+2}^{(k+1)}} \arrow[r] \arrow[d, red, "\trunc" red] & \cdots \\
		&	                                        	&&&&\quad&                                                                                
	\end{tikzcd}
	\caption{Truncation in dSDA$\rm{_t}$}
	\label{fig:truncation-routine:}
\end{figure}

Occasionally,  we write $\widetilde G_j\equiv G_j^{(j)}$, $j=1, 2, \cdots$, 
the truncated matrices of $G_j^{(j-1)}$, where 
$G_1^{(0)}:=G_1$. It is worthwhile to point out that  
in Figure~\ref{fig:truncation-routine:} only those terms in boxes  are actually computed, 
 and we shall produce a formula for the  short-cut from $G_{k}^{(k)}$ to $G_{k+1}^{(k+1)}$, 
without going through $G_{k+1}^{(k)}$. 
This section also contains the details of the truncation of 
$G_{k}^{(k-1)}$ and $H_{k}^{(k-1)}$ to $G_{k}^{(k)}$ and $H_{k}^{(k)}$ respectively, 
and the general form of $G_{j}^{(k)}$ and $H_{j}^{(k)}$, for $k\geq 1$ and $j>k$. 
These details are difficult to obtain but  indispensable for the understanding and analysis of the dSDA$\rm{_t}$. 

It is worth noting that the truncation technique in the dSDA$\rm{_t}$ is extendable to other 
associated problems solvable by the dSDA, such as the DAREs and the Bethe-Salpeter eigenvalue problems.

\subsection{Truncation}\label{ssec:truncation-process}

\subsubsection{Truncating \texorpdfstring{$G_1$}{G1} and \texorpdfstring{$H_1$}{H1}}\label{sssec:truncating-g_1-and-h_1-}

Let the QR factorizations with column pivoting of $[U_0, U_1]$ and  
$[V_0, V_1]$, respectively,  be 
\begin{align*}
	[U_0, U_1] (P_1^U)^{\T}=Q_1^U R_1^U, \qquad 
[V_0, V_1] (P_1^V)^{\T}=Q_1^V R_1^V, 
\end{align*}
where $P_1^U\in \mathbb{R}^{2m\times 2m}$ and $P_1^V\in \mathbb{R}^{2l\times 2l}$ are permutations,  
$Q_1^U\in \mathbb{R}^{n\times p_1}$, $R_1^U\in \mathbb{R}^{p_1\times 2m}$ with $p_1\le 2m$,  
$Q_1^V\in \mathbb{R}^{n\times q_1}$, $R_1^V\in \mathbb{R}^{q_1\times 2l}$ with $q_1\le 2l$.  
Next  construct the SVD: $Y_1 = U_1^Y\Sigma_1^Y (V_1^Y)^{\T}$, 
where $U_1^Y\in \mathbb{R}^{2m \times 2m}$, $\Sigma_1^Y\in \mathbb{R}^{2m\times 2l}$ and 
$V_1^Y\in \mathbb{R}^{2l\times 2l}$. 
Let $\Upsilon_1^G=I_{2m}+\Sigma_1^Y(\Sigma_1^Y)^{\T} >0$ and 
$\Upsilon_1^H=I_{2l}+(\Sigma_1^Y)^{\T} \Sigma_1^Y >0$, we compute    
the SVDs: 
\begin{equation}\label{eq:svd1}
	\begin{aligned}
		R_1^UP_1^U U_1^Y (\Upsilon_1^G)^{-1/2} 
		= \Theta_1^G \Sigma_1^G (\Phi_1^G)^{\T}, 
		& \quad 
		R_1^VP_1^V V_1^Y (\Upsilon_1^H)^{-1/2}
		= \Theta_1^H \Sigma_1^H (\Phi_1^H)^{\T}, 
	\end{aligned}
\end{equation}
where $\Theta_1^G, \Sigma_{1}^G\in \mathbb{R}^{p_1\times p_1}$;  
$\Theta_1^H, \Sigma_1^H\in \mathbb{R}^{q_1\times q_1}$;  
$\Phi_1^G\in \mathbb{R}^{2m \times p_1}$ and $\Phi_1^H\in \mathbb{R}^{2l \times q_1}$. We then  have 
\begin{align*}
	G_1=2\gamma Q_1^U \Theta_1^G (\Sigma_1^G)^2 (\Theta_1^G)^{\T} (Q_1^U)^{\T}, 
	\qquad 
	H_1=2\gamma Q_1^V \Theta_1^H (\Sigma_1^H)^2 (\Theta_1^H)^{\T} (Q_1^V)^{\T}. 
\end{align*}
Let $\Sigma_1^G = \Sigma_{1,1}^G \oplus \Sigma_{2,1}^G$ and  
$\Sigma_1^H = \Sigma_{1,1}^H \oplus \Sigma_{2,1}^H$,    
where $\Sigma_{1,1}^G \in \mathbb{R}^{r_1^G\times r_1^G}$ and 
$\Sigma_{1,1}^H \in \mathbb{R}^{r_1^H\times r_1^H}$ with  
$\|\Sigma_{2,1}^G\|\leq \varepsilon_1 \|\Sigma_{1,1}^G\|$ and 
$\|\Sigma_{2,1}^H\|\leq \varepsilon_1 \|\Sigma_{1,1}^H\|$ for 
some small tolerance $\varepsilon_1$. Actually, 
the tolerances for $\Sigma_{2,1}^G$ and $\Sigma_{2,1}^H$ can be different 
and for simplicity we use the same.  
Write $\Theta_1^G=[\Theta_{1,1}^G, \Theta_{2,1}^G]$, 
$\Theta_1^H=[\Theta_{1,1}^H, \Theta_{2,1}^H ]$, 
$\Phi_1^G=[\Phi_{1,1}^G, \Phi_{2,1}^G]$ and  
$\Phi_1^H=[\Phi_{1,1}^H, \Phi_{2,1}^H ]$,  
where $\Theta_{1,1}^G \in \mathbb{R}^{p_1\times r_1^G}$, 
$\Theta_{1,1}^H \in \mathbb{R}^{q_1\times r_1^H}$, 
$\Phi_{1,1}^G \in \mathbb{R}^{2m \times r_1^G}$ and 
$\Phi_{1,1}^H \in \mathbb{R}^{2l \times r_1^H}$. 
With respect to the tolerance $\varepsilon_1$,  the truncated matrices of $G_1$ and $H_1$, 
respectively, are 
\begin{equation}\label{eq:widetildeG1-and-widetilde-H1}
	\begin{aligned}
		\widetilde G_1
		&\equiv 
		G_1^{(1)}=2\gamma Q_1^U \Theta_{1,1}^G (\Sigma_{1,1}^G)^2 
		(\Theta_{1,1}^G)^{\T} (Q_1^U)^{\T},   
		\\ 
		\widetilde H_1
		&\equiv 
		H_1^{(1)}=2\gamma Q_1^V \Theta_{1,1}^H (\Sigma_{1,1}^H)^2 
		(\Theta_{1,1}^H)^{\T} (Q_1^V)^{\T}.
	\end{aligned}	
\end{equation}
	%\begin{align}
		%\widetilde G_1\equiv G_1^{(1)}=2\gamma Q_1^U \Theta_{1,1}^G (\Sigma_{1,1}^G)^2 
		%(\Theta_{1,1}^G)^{\T} (Q_1^U)^{\T},   
		%\ 
		%\widetilde H_1\equiv H_1^{(1)}=2\gamma Q_1^V \Theta_{1,1}^H (\Sigma_{1,1}^H)^2 
		%(\Theta_{1,1}^H)^{\T} (Q_1^V)^{\T}. \label{eq:widetildeG1-and-widetilde-H1} 
	%\end{align}

After truncation, we now proceed with the dSDA starting from $G_1^{(1)}$ and $H_1^{(1)}$. Before that we need to reformulate  
$G_1^{(1)}$ and $H_1^{(1)}$ in decoupled forms. 
Noting  that 
\begin{align*}
	\Theta_{1,1}^G(\Sigma_{1,1}^G)^2(\Theta_{1,1}^G)^{\T}
	&=\Theta_1^G  (I_{r_1^G} \oplus 0) (\Sigma_1^G)^2  
	(I_{r_1^G} \oplus 0)  (\Theta_1^G)^{\T},
	\\
	\Theta_{1,1}^H(\Sigma_{1,1}^H)^2(\Theta_{1,1}^H)^{\T}
	&=\Theta_1^H   (I_{r_1^H} \oplus 0)   (\Sigma_1^H)^2   
	(I_{r_1^H} \oplus 0)   (\Theta_1^H)^{\T}, 
\end{align*}
then \eqref{eq:svd1}  and \eqref{eq:widetildeG1-and-widetilde-H1} imply 
\[ 
	G_1^{(1)}\equiv 2\gamma \mathcal{Q}_1^U 
	\left(I_{2m}+Y_1 Y_1^{\T}\right)^{-1} (\mathcal{Q}_1^U)^{\T},
	\ \ \ 
	H_1^{(1)}\equiv  2\gamma \mathcal{Q}_1^V 
	\left(I_{2l}+ Y_1^{\T} Y_1 \right)^{-1} (\mathcal{Q}_1^V)^{\T},
\] 
where $\mathcal{Q}_1^U:=Q_1^U \Theta_1^G    
(I_{r_1^G}\oplus 0)   (\Theta_1^G)^{\T} R_1^UP_1^U$ and 
$\mathcal{Q}_1^V:=Q_1^V \Theta_1^H    
(I_{r_1^H}\oplus 0)   (\Theta_1^H)^{\T} R_1^V P_1^V$. 

Denoting $A_1^{(1)}:=\widetilde A_{\gamma}^{2}-
2\gamma \mathcal{Q}_1^U (I_{2m}+Y_1 Y_1^{\T})^{-1}Y_1 (\mathcal{Q}_1^V)^{\T}$, 
\begin{align*}
	&\widehat{\mathcal{X}}_k^{U, (1)}: = 
	\left[\mathcal{Q}_1^U, \widetilde A_{\gamma}^2 \mathcal{Q}_1^U, 
	\cdots, \widetilde A_{\gamma}^{2^k-2} \mathcal{Q}_1^U\right],
	\, 
	\widehat{\mathcal{X}}_k^{V,(1)} := 
	\left[	\mathcal{Q}_1^V, (\widetilde A_{\gamma}^{\T})^2 \mathcal{Q}_1^V, 
	\cdots, (\widetilde A_{\gamma}^{\T})^{2^k-2}\mathcal{Q}_1^V \right], 
	\\  
	&\mathcal{X}_{k}^{U,(1)}:= 
	\left[ Q_1^U \Theta_{1,1}^G, \widetilde A_{\gamma}^2 Q_1^U \Theta_{1,1}^G, 
	\cdots, \widetilde A_{\gamma}^{2^k-2} Q_1^U \Theta_{1,1}^G\right], 
	\\ 
	&\mathcal{X}_{k}^{V,(1)}:=
	\left[Q_1^V \Theta_{1,1}^H, (\widetilde A_{\gamma}^{\T})^2 Q_1^V \Theta_{1,1}^H, 
	\cdots, (\widetilde A_{\gamma}^{\T})^{2^k-2} Q_1^V \Theta_{1,1}^H\right], 
	\end{align*}
 and $M_0^G:=(\Theta_{1,1}^G)^{\T}R_1^UP_1^U$, 
$M_0^H:=(\Theta_{1,1}^H)^{\T}R_1^VP_1^V$,  
then from $A_1^{(1)}$,  $G_1^{(1)}$ and $H_1^{(1)}$, 
the dSDA in \eqref{eq:care-iteration} produces  ($k\ge 2$):  
\begin{equation}\label{eq:Hk1}
	\begin{aligned}
		G_k^{(1)} 
		&= 
		2\gamma 
		\widehat{\mathcal{X}}_k^{U,(1)} 
		E(Y_k^{(1)}) 
		(\widehat{\mathcal{X}}_k^{U,(1)})^{\T}
		\\
		&\equiv
		2\gamma 
		\mathcal{X}_k^{U,(1)}
		(I_{2^{k-1}}\otimes M_0^G)
		E(Y_k^{(1)}) 
		(I_{2^{k-1}}\otimes M_0^G)^{\T}
		(\mathcal{X}_k^{U,(1)})^{\T},  
		\\
		H_k^{(1)} 
		&= 
		2\gamma \widehat{\mathcal{X}}_k^{V,(1)} 
		F(Y_k^{(1)}) 
		(\widehat{\mathcal{X}}_k^{V,(1)})^{\T}
		\\
		&\equiv
		2\gamma
		\mathcal{X}_k^{V,(1)}
		(I_{2^{k-1}}\otimes M_0^H)
		F(Y_k^{(1)}) 
		(I_{2^{k-1}}\otimes M_0^H)^{\T}
		(\mathcal{X}_k^{V,(1)})^{\T}, 
	\end{aligned}	
\end{equation}
where 
\begin{align*}
	&E(Y_k^{(1)}) := [I_{2^km} + Y_k^{(1)}(Y_k^{(1)})^{\T}]^{-1}, 
	\qquad F(Y_k^{(1)}) := [I_{2^kl} + (Y_k^{(1)})^{\T}Y_k^{(1)}]^{-1},
\end{align*}
with  $Y_j^{(1)}=\begin{bmatrix}
	0& Y_{j-1}^{(1)}\\ Y_{j-1}^{(1)}  & 2\gamma T_{j-1}^{(1)}
\end{bmatrix}\in \mathbb{R}^{2^jm\times 2^jl}$, $Y_1^{(1)}\equiv Y_1$ and 
\[
	T_j^{(1)} = (I_{2^{j-1}}\otimes M_0^G)^{\T}  
	(\mathcal{X}_j^{U,(1)})^{\T}\mathcal{X}_j^{V,(1)}(I_{2^{j-1}}\otimes M_0^H).
\]

\subsubsection{Truncating \texorpdfstring{$G_2^{(1)}$}{G2(1)} and \texorpdfstring{$H_2^{(1)}$}{H2(1)}}\label{sssec:truncating-g_2-and-h_2-}

From~\eqref{eq:Hk1},  we know that  
\begin{align*}
	G_2^{(1)} &= 2\gamma \mathcal{X}_2^{U,(1)} 
	(I_2 \otimes M_0^G) 
	E(Y_2^{(1)}) 
	(I_2 \otimes M_0^G)^{\T} (\mathcal{X}_2^{U,(1)})^{\T}, 
	\\
	H_2^{(1)} &= 2\gamma \mathcal{X}_2^{V,(1)} 
	(I_2 \otimes M_0^H) 
	F(Y_2^{(1)})
	(I_2 \otimes M_0^H)^{\T} (\mathcal{X}_2^{V,(1)})^{\T}. 
\end{align*}
Write $\Gamma:=(I_{2m}+Y_1Y_1^{\T})^{-1}Y_1 (T_1^{(1)})^{\T}$  and  
\[
	\Psi_1:= I_{2m}+Y_1Y_1^{\T} + 4\gamma^2 T_1^{(1)}(I_{2l} + Y_1^{\T} Y_1)^{-1}(T_1^{(1)})^{\T},
\]
then from the definition of $Y_2^{(1)}$, we have  
\begin{align*}
	E(Y_2^{(1)}) 
	=&
	\begin{bmatrix}
		I_{2m}&-2\gamma \Gamma \\ 0&I_{2m}
	\end{bmatrix}
	\left[ (I_{2m} + Y_1 Y_1^{\T})^{-1} \oplus  \Psi_1^{-1} \right]
	\begin{bmatrix}
		I_{2m}&0\\-2\gamma \Gamma^{\T} & I_{2m}
	\end{bmatrix}.
\end{align*}
With $\Omega_1=(Q_1^U\Theta_{1,1}^G)^{\T} Q_1^V \Theta_{1,1}^H$ and  
\[
L_1^{G}:=2\gamma (\Theta_{1,1}^G)^{\T} R_1^U P_1^U 
(I_{2m}+Y_1Y_1^{\T})^{-1}Y_1 (P_1^V)^{\T} (R_1^V)^{\T} \Theta_{1,1}^H \Omega_1^{\T},
\]
then subsequently by the definition of $T_1^{(1)}$, it holds that 
\begin{align*}
	&
	(I_2 \otimes M_0^G)
	E(Y_2^{(1)})
	(I_2 \otimes M_0^G)^{\T}
	\\
	=&
	\begin{multlined}[t]
		\begin{bmatrix}
			I_{r_1^G}&-L_1^{G}	\\ 0&I_{r_1^G}   
		\end{bmatrix}
		(I_2 \otimes M_0^G)
		\left[ (I_{2m} + Y_1 Y_1^{\T})^{-1} \oplus  \Psi_1^{-1}\right]
		(I_2 \otimes M_0^G)^{\T}
		\begin{bmatrix}
			I_{r_1^G}&	-L_1^{G} \\ 0 &I_{r_1^G}
		\end{bmatrix}^{\T}
	\end{multlined}
	\\
	\equiv&
	\begin{multlined}[t]
		\begin{bmatrix}
			I_{r_1^G}&-L_1^{G}		\\ 0&I_{r_1^G}   
		\end{bmatrix}
		\left\{ 
			(\Sigma_{1,1}^G)^2 \oplus 
			\left[ \Sigma_{1,1}^G \left(I_{r_1^G}+4\gamma^2\Sigma_{1,1}^G \Omega_1 
			(\Sigma_{1,1}^H)^2 \Omega_1^{\T}\Sigma_{1,1}^G \right)^{-1}\Sigma_{1,1}^G \right]
		\right\} 
		\\
		\cdot
		\begin{bmatrix}
			I_{r_1^G} & -L_1^{G} \\ 0	& I_{r_1^G}
		\end{bmatrix}^{\T}.
	\end{multlined}
\end{align*}
Similarly, with $L_1^{H}:=2\gamma (\Theta_{1,1}^H)^{\T} R_1^V P_1^V 
(I_{2l}+ Y_1^{\T} Y_1)^{-1} Y_1^{\T} (P_1^U)^{\T} (R_1^U)^{\T} \Theta_{1,1}^G \Omega_1$,   
we have 
	\begin{align*}
		&
		(I_2\otimes M_0^H)
		F(Y_2^{(1)}) 
		(I_2\otimes M_0^H)^{\T}
		\\
		=&
		\begin{multlined}[t]
			\begin{bmatrix}
				I_{r_1^H}&-L_1^{H}	\\ 0&I_{r_1^H}   
			\end{bmatrix}
			\left\{
				(\Sigma_{1,1}^H)^2 \oplus
				\left[\Sigma_{1,1}^H
					\left(I_{r_1^H}+4\gamma^2 \Sigma_{1,1}^H 
					\Omega_1^{\T}(\Sigma_{1,1}^G)^2 \Omega_1 \Sigma_{1,1}^H \right)^{-1} 
				\Sigma_{1,1}^H \right]
			\right\} 
			 \\
			 \cdot
			 \begin{bmatrix}
				I_{r_1^H}&-L_1^{H}\\ 0 & I_{r_1^H}
			\end{bmatrix}^{\T}.
		\end{multlined}
	\end{align*}

Compute the QR factorizations   using the  modified Gram-Schmidt process: 
\begin{equation}\label{eq:qr2}
	\begin{aligned}
		\mathcal{X}_2^{U,(1)}
		= 
		\left[Q_1^U\Theta_{1,1}^G, Q_2^U\right]
		\begin{bmatrix}
			I_{r_1^G}& R_{12}^U \\0&R_2^U
		\end{bmatrix}, 
		\qquad
		\mathcal{X}_2^{V,(1)}
		= 
		\left[Q_1^V \Theta_{1,1}^H, Q_2^V\right]
		\begin{bmatrix}
			I_{r_1^H}& R_{12}^V \\0&R_2^V
		\end{bmatrix}, 
	\end{aligned}
\end{equation}
where $Q_2^U\in \mathbb{R}^{n\times (p_2-r_1^G)}$, 
$Q_2^V\in \mathbb{R}^{n\times (q_2-r_1^H)}$, 
$R_{12}^U\in \mathbb{R}^{r_1^G \times r_1^G}$, 
$R_2^U\in \mathbb{R}^{(p_2-r_1^G)\times r_1^G}$,  
$R_{12}^V\in \mathbb{R}^{r_1^H \times r_1^H}$ and  
$R_2^V\in \mathbb{R}^{(q_2-r_1^H)\times r_1^H}$.   Consider the SVD:  
$\Sigma_{1,1}^G \Omega_1 \Sigma_{1,1}^H = U_2^Y \Sigma_2^Y (V_2^Y)^{\T}$, 
where $U_2^Y\in \mathbb{R}^{r_1^G \times r_1^G}$, 
$\Sigma_2^Y\in \mathbb{R}^{r_1^G \times r_1^H}$ and 
$V_2^Y \in \mathbb{R}^{r_1^H \times r_1^H}$, 
we then obtain 
\begin{align*}
	\Sigma_{1,1}^G \Omega_1 (\Sigma_{1,1}^H)^2 \Omega_1^{\T}\Sigma_{1,1}^G
	&= U_2^Y \Sigma_2^Y (\Sigma_2^Y)^{\T} (U_2^Y)^{\T}, 
	\\ 
	\Sigma_{1,1}^H \Omega_1^{\T}(\Sigma_{1,1}^G)^2 \Omega_1 \Sigma_{1,1}^H 
	&= V_2^Y (\Sigma_2^Y)^{\T} \Sigma_2^Y (V_2^Y)^{\T}.
\end{align*}
Now let $\Upsilon_2^G:= I_{r_1^G}+4\gamma^2 \Sigma_2^Y(\Sigma_2^Y)^{\T}$ 
and  $\Upsilon_2^H:= I_{r_1^H}+4\gamma^2 (\Sigma_2^Y)^{\T}\Sigma_2^Y$. 
Consider further   the SVDs:
\begin{equation}\label{eq:svd2_H}
	\begin{aligned}
	&\begin{bmatrix}
		I_{r_1^G}&R_{12}^U\\0&R_2^U
	\end{bmatrix}
	\begin{bmatrix}
		I_{r_1^G}&-L_1^{G}		\\ 0&I_{r_1^G}   
	\end{bmatrix}
	\left\{
		\Sigma_{1,1}^G \oplus 
		\left[\Sigma_{1,1}^G U_2^Y (\Upsilon_2^G)^{-1/2}\right]
	\right\}
	=
	\Theta_2^G \Sigma_2^G (\Phi_2^G)^{\T},%\nonumber 
	\\
	&\begin{bmatrix}
		I_{r_1^H}&R_{12}^V\\0&R_2^V
	\end{bmatrix}
	\begin{bmatrix}
		I_{r_1^H}&-L_1^{H}		\\ 0&I_{r_1^H}   
	\end{bmatrix}
	\left\{
		\Sigma_{1,1}^H \oplus 
		\left[ \Sigma_{1,1}^H V_2^Y (\Upsilon_2^H)^{-1/2}\right] 
	\right\} 
	=
	\Theta_2^H \Sigma_2^H (\Phi_2^H)^{\T}, %\label{eq:svd2_H}
\end{aligned}
\end{equation}
where $\Theta_2^G, \Sigma_2^G\in \mathbb{R}^{p_2 \times p_2}$;  
$\Theta_2^H, \Sigma_2^H\in \mathbb{R}^{q_2 \times q_2}$;   
$\Phi_2^G\in \mathbb{R}^{2r_1^G \times p_2}$ and 
$\Phi_2^H\in \mathbb{R}^{2r_1^H \times q_2}$. 
We obtain  
\begin{align*}
	G_2^{(1)} = 
	2\gamma \widehat{\mathcal{Q}}_2^U
	\Theta_2^G (\Sigma_2^G)^2 (\Theta_2^G)^{\T}
	(\widehat{\mathcal{Q}}_2^U)^{\T},
	\qquad 
	H_2^{(1)} = 
	2\gamma \widehat{\mathcal{Q}}_2^V
	\Theta_2^H (\Sigma_2^H)^2 (\Theta_2^H)^{\T} 
	(\widehat{\mathcal{Q}}_2^V)^{\T},
\end{align*}
where $\widehat{\mathcal{Q}}_2^U:=[Q_1^U \Theta_{1,1}^G,Q_2^U]$ 
and $\widehat{\mathcal{Q}}_2^V:= [Q_1^V\Theta_{1,1}^H, Q_2^V]$.   
With $\varepsilon_2$ being some small tolerance,  write 
$\Sigma_2^G = \Sigma_{1,2}^G \oplus \Sigma_{2,2}^G$, 
$\Sigma_2^H = \Sigma_{1,2}^H \oplus \Sigma_{2,2}^H$  
with $\Sigma_{1,2}^G \in \mathbb{R}^{r_2^G\times r_2^G}$, 
$\Sigma_{1,2}^H \in \mathbb{R}^{r_2^H\times r_2^H}$, satisfying 
$\|\Sigma_{2,2}^G\|\leq \varepsilon_2 \|\Sigma_{1,2}^G\|$ and 
$\|\Sigma_{2,2}^H\|\leq \varepsilon_2 \|\Sigma_{1,2}^H\|$.   Write  
$\Theta_2^G=[\Theta_{1,2}^G, \Theta_{2,2}^G ]$, 
$\Theta_2^H=[\Theta_{1,2}^H, \Theta_{2,2}^H]$, 
$\Phi_2^G=[\Phi_{1,2}^G, \Phi_{2,2}^G ]$ and  
$\Phi_2^H=[\Phi_{1,2}^H, \Phi_{2,2}^H ]$, 
whose partitions respectively are compatible with those of $\Sigma_2^G$ and $\Sigma_2^H$, 
i.e., $\Theta_{1,2}^G \in \mathbb{R}^{p_2\times r_2^G}$, 
$\Theta_{1,2}^H \in \mathbb{R}^{q_2\times r_2^H}$, 
$\Phi_{1,2}^G \in \mathbb{R}^{2r_1^G \times r_2^G}$ and 
$\Phi_{1,2}^H \in \mathbb{R}^{2r_1^H \times r_2^H}$. 
Then the truncated matrices of $G_2^{(1)}$ and $H_2^{(1)}$, 
with respect to the  tolerance $\varepsilon_2$, respectively  are 
\begin{equation}\label{eq:wtdGH2}
	\widetilde G_2 \equiv G_2^{(2)} = 
	2\gamma \mathcal{Q}_2^U
	(\Sigma_{1,2}^G)^2 	(\mathcal{Q}_2^U)^{\T}, 
	\qquad  
	\widetilde H_2 \equiv H_2^{(2)} = 
	2\gamma  \mathcal{Q}_2^V
	(\Sigma_{1,2}^H)^2  
	(\mathcal{Q}_2^V)^{\T}, 
\end{equation}
where $\mathcal{Q}_2^U:= \widehat{\mathcal{Q}}_2^U \Theta_{1,2}^G$ and 
$\mathcal{Q}_2^V:= \widehat{\mathcal{Q}}_2^V \Theta_{1,2}^H$. 

Again, after truncation we apply the dSDA starting from $G_2^{(2)}$ and $H_2^{(2)}$.  

Substituting \eqref{eq:qr2} into $G_k^{(1)}$ and $H_k^{(1)}$ 
 in \eqref{eq:Hk1}, with 
\begin{align*}
	\widehat{\mathcal{X}}_k^{U,(2)} := 
	\left[ \widehat{\mathcal{Q}}_2^U, \widetilde A_{\gamma}^4 \widehat{\mathcal{Q}}_2^U, 
	\cdots, \widetilde A_{\gamma}^{2^k-4} \widehat{\mathcal{Q}}_2^U \right],
	\,    
	\widehat{\mathcal{X}}_k^{V,(2)} := 
	\left[\widehat{\mathcal{Q}}_2^V, (\widetilde A_{\gamma}^{\T})^4 \widehat{\mathcal{Q}}_2^V, 
	\cdots, (\widetilde A_{\gamma}^{\T})^{2^k-4} \widehat{\mathcal{Q}}_2^V \right], 	
\end{align*}
$M_1^G= \begin{bmatrix}
		I_{r_1^G}& R_{12}^U\\0&R_2^U
	\end{bmatrix} (I_2\otimes M_0^G)$ and 
$M_1^H= \begin{bmatrix}
		I_{r_1^H}& R_{12}^V\\0&R_2^V
	\end{bmatrix} (I_2 \otimes M_0^H)$,  
we reformulate  $G_k^{(1)}$ and $H_k^{(1)}$:  
\begin{equation}\label{eq:Hk1_form2}
	\begin{aligned}
	&G_k^{(1)} 
	=   
	\begin{multlined}[t]
		2\gamma \widehat{\mathcal{X}}_k^{U,(2)}
		\left(I_{2^{k-2}}\otimes M_1^G\right)	
		E(Y_k^{(1)}) 
		\left(I_{2^{k-2}}\otimes M_1^G\right)^{\T} 
		(\widehat{\mathcal{X}}_k^{U,(2)})^{\T}, 
	\end{multlined} %\nonumber 
	\\
	&H_k^{(1)} 
	= 
	\begin{multlined}[t]
		2\gamma \widehat{\mathcal{X}}_k^{V,(2)}
		\left(I_{2^{k-2}}\otimes M_1^H \right)	
		F(Y_k^{(1)}) 
		(I_{2^{k-2}}\otimes M_1^H)^{\T}
		(\widehat{\mathcal{X}}_k^{V,(2)})^{\T}. 
	\end{multlined} %\label{eq:Hk1_form2}
\end{aligned}
\end{equation}
It is clear that  
\begin{align*} 
	&\Theta_{1,2}^G (\Sigma_{1,2}^G)^2 (\Theta_{1,2}^G)^{\T}  
	=
	\Theta_2^G  (I_{r_2^G} \oplus 0) (\Sigma_2^G)^2 (I_{r_2^G} \oplus 0)(\Theta_2^G)^{\T}
	\\
	\equiv&
	\begin{multlined}[t]
		\Theta_2^G (I_{r_2^G} \oplus 0)(\Theta_2^G)^{\T} 
		\begin{bmatrix}
			I_{r_1^G}& R_{12}^U\\0 & R_2^U
		\end{bmatrix}
		\begin{bmatrix}
			I_{r_1^G}&-L_1^{G}		\\ 0&I_{r_1^G}   
		\end{bmatrix}
		\\
		\cdot
		\left\{ 	
			(\Sigma_{1,1}^G)^2 \oplus 
			\left[ \Sigma_{1,1}^G U_2^Y (\Upsilon_2^G)^{-1}(U_2^Y)^{\T}\Sigma_{1,1}^G \right] 
		\right\}
		\begin{bmatrix}
			I_{r_1^G}&-L_1^{G}		\\ 0&I_{r_1^G}   
		\end{bmatrix}^{\T}
		\begin{bmatrix}
			I_{r_1^G}&R_{12}^U\\0&R_2^U
		\end{bmatrix}^{\T}
		\\
		\cdot
		\Theta_2^G (I_{r_2^G} \oplus 0)(\Theta_2^G)^{\T}
	\end{multlined}
	\\
	\equiv &
	\widehat{M}_1^G	 
	E(Y_2^{(1)}) 
	(\widehat{M}_1^G)^{\T},
\end{align*}
where $\widehat{M}_1^G:= \Theta_2^G (I_{r_2^G} \oplus 0)(\Theta_2^G)^{\T}M_1^G$.   
Similarly, with $\widehat{M}_1^H:= \Theta_2^H (I_{r_2^H} \oplus 0)(\Theta_2^H)^{\T} M_1^H$,   
we have 
\begin{align*}
	\Theta_{1,2}^H(\Sigma_{1,2}^H)^2(\Theta_{1,2}^H)^{\T} =
	\widehat{M}_1^H
	F(Y_2^{(1)}) 
	(\widehat{M}_1^H)^{\T}.
\end{align*}
Hence we obtain 
\[ 
	G_2^{(2)}
	\equiv 
	2\gamma \widehat{\mathcal{Q}}_2^U 
	\widehat{M}_1^G 
	E(Y_2^{(1)}) 
	(\widehat{M}_1^G)^{\T}
	(\widehat{\mathcal{Q}}_2^U)^{\T}, 
	\ \ \ 
		H_2^{(2)}
	\equiv 
	2\gamma \widehat{\mathcal{Q}}_2^V 
	\widehat{M}_1^H	
	F(Y_2^{(1)}) 
	(\widehat{M}_1^H)^{\T}
	(\widehat{\mathcal{Q}}_2^V)^{\T}.
\] 

Define  $A_2^{(2)}:= \widetilde{A}_{\gamma}^4 - 2\gamma 
\widehat{\mathcal{Q}}_2^U \widehat{M}_1^G
[I_{4m}+Y_2^{(1)}(Y_2^{(1)})^{\T}]^{-1} Y_2^{(1)}
(\widehat{M}_1^H)^{\T}(\widehat{\mathcal{Q}}_2^V)^{\T}$. 
Analogously to 
\eqref{eq:Hk1_form2}, 
with 
\[
E(Y_k^{(2)}) := [I_{2^k m} + Y_k^{(2)}(Y_k^{(2)})^{\T}]^{-1}, \qquad 
F(Y_k^{(2)}) := [I_{2^k l} +(Y_k^{(2)})^{\T} Y_k^{(2)} ]^{-1},  
\]
applying the dSDA~\eqref{eq:care-iteration} starting from  
$A_2^{(2)}$, $G_2^{(2)}$ and $H_2^{(2)}$    produces 
\begin{align*}
	G_k^{(2)} 
	& = 
	2\gamma \widehat{\mathcal{X}}_k^{U,(2)}  
	(I_{2^{k-2}} \otimes \widehat{M}_1^G)
	E(Y_k^{(2)}) 
	(I_{2^{k-2}} \otimes \widehat{M}_1^G)^{\T}
	(\widehat{\mathcal{X}}_k^{U,(2)})^{\T}
	\\
	&\equiv
	2\gamma \mathcal{X}_k^{U,(2)} 
	\left[I_{2^{k-2}}\otimes  (\Theta_{1,2}^G)^{\T} M_1^G  \right] 
	E(Y_k^{(2)}) 
	\left[I_{2^{k-2}}\otimes  (M_1^G)^{\T} \Theta_{1,2}^G  \right]
	(\mathcal{X}_k^{U,(2)})^{\T}, 
	\\ 
	H_k^{(2)} 
	&= 
	2\gamma \widehat{\mathcal{X}}_k^{V,(2)} 
	(I_{2^{k-2}}\otimes \widehat{M}_1^H)
	F(Y_k^{(2)}) 
	(I_{2^{k-2}}\otimes \widehat{M}_1^H)^{\T}
	(\widehat{\mathcal{X}}_k^{V,(2)})^{\T} 
	\\
	&\equiv
	2\gamma 
	\mathcal{X}_k^{V,(2)} 
	\left[I_{2^{k-2}}\otimes  (\Theta_{1,2}^H)^{\T} M_1^H \right]
	F(Y_k^{(2)}) 
	\left[I_{2^{k-2}}\otimes   ( M_1^H)^{\T} \Theta_{1,2}^H   \right]
	(\mathcal{X}_k^{V,(2)})^{\T}, 
\end{align*}
where  
\begin{align*}
	\mathcal{X}_k^{U,(2)}:=\left[\mathcal{Q}_2^U, \widetilde{A}_{\gamma}^4 \mathcal{Q}_2^U, 
	\cdots, \widetilde{A}_{\gamma}^{2^k-4}\mathcal{Q}_2^U \right],
	\, 
\mathcal{X}_k^{V,(2)}:=\left[\mathcal{Q}_2^V, (\widetilde{A}_{\gamma}^{\T})^4 \mathcal{Q}_2^V, 
	\cdots, (\widetilde{A}_{\gamma}^{\T})^{2^k-4}\mathcal{Q}_2^V \right],
\end{align*}
$Y_j^{(2)}=\begin{bmatrix}
	0& Y_{j-1}^{(2)}\\ Y_{j-1}^{(2)}  & 2\gamma T_{j-1}^{(2)}
\end{bmatrix}$ with  $Y_2^{(2)}\equiv Y_2^{(1)}$ and 
\[
	T_j^{(2)} = 
	[I_{2^{j-2}}\otimes (M_1^G)^{\T} \Theta_{1,2}^G ]
	(\mathcal{X}_j^{U,(2)})^{\T} 
	\mathcal{X}_j^{V,(2)} 
	[I_{2^{j-2}}\otimes  (\Theta_{1,2}^H)^{\T} M_1^H ].
\] 

Obviously, with $E(Y_3^{(2)}) := [I+Y_3^{(2)} (Y_3^{(2)})^{\T}]^{-1}$,  
$F(Y_3^{(2)}) := [I+ (Y_3^{(2)})^{\T}Y_3^{(2)}]^{-1}$, we have 
\begin{align*}
	G_3^{(2)}& = 2\gamma \mathcal{X}_3^{U,(2)} 
	\left[I_2 \otimes (\Theta_{1,2}^G )^{\T} M_1^G\right] 
	E(Y_3^{(2)}) 
	\left[I_2 \otimes (M_1^G)^{\T} \Theta_{1,2}^G\right] (\mathcal{X}_3^{U,(2)})^{\T},
	\\
	H_3^{(2)}&= 2\gamma \mathcal{X}_3^{V,(2)} 
	\left[I_2 \otimes (\Theta_{1,2}^H )^{\T} M_1^H\right] 
	F(Y_3^{(2)}) 
	\left[I_2 \otimes (M_1^H)^{\T} \Theta_{1,2}^H\right] (\mathcal{X}_3^{V,(2)})^{\T}.
\end{align*}
To get $\widetilde{G}_3 \equiv G_3^{(3)}$ and $\widetilde{H}_3\equiv H_3^{(3)}$, 
we need to reformulate the kernels 
\begin{align*}
	&\left[I_2 \otimes (\Theta_{1,2}^G )^{\T} M_1^G\right] 
E(Y_3^{(2)}) 
\left[I_2 \otimes (M_1^G)^{\T} \Theta_{1,2}^G\right], 
\\ 
&\left[I_2 \otimes (\Theta_{1,2}^H )^{\T} M_1^H\right] 
F(Y_3^{(2)}) 
\left[I_2 \otimes (M_1^H)^{\T} \Theta_{1,2}^H\right],
\end{align*}
and  compute the QR factorizations of the column spaces  
$\mathcal{X}_3^{U,(2)}$  and $\mathcal{X}_3^{V,(2)}$.  
The  details for the general cases can be found in the next section. 

\subsubsection{Truncating \texorpdfstring{$G_{j+1}^{(j)}$}{Gj+1(j)} and \texorpdfstring{$H_{j+1}^{(j)}$}{Hj+1(j)}}\label{sssec:truncating-g_j-and-h_j-}

Generalizing the results in the previous section, 
with respect to some small tolerance $\varepsilon_j$,  
we  truncate $G_{j}^{(j-1)}$ and $H_j^{(j-1)}$  respectively to  
\begin{equation}\label{eq:wtdH_j}
\begin{aligned}
	\widetilde G_j \equiv G_j^{(j)}
	&=2\gamma \mathcal{Q}_{j}^U (\Sigma_{1,j}^G)^2  	
	(\mathcal{Q}_j^U)^{\T},\ \ \  \mathcal{Q}_j^U\in \mathbb{R}^{n\times r_j^G}, \ \  
	\Sigma_{1,j}^G\in \mathbb{R}^{r_j^G \times r_j^G}, %\nonumber 
	\\
	\widetilde H_j \equiv H_j^{(j)}
	&=2\gamma \mathcal{Q}_{j}^V (\Sigma_{1,j}^H)^2  (\mathcal{Q}_j^V)^{\T}, \ \ \  
	\mathcal{Q}_j^V\in \mathbb{R}^{n\times r_j^H}, \ \ 
	\Sigma_{1,j}^H\in \mathbb{R}^{r_j^H \times r_j^H}. %\label{eq:wtdH_j}
\end{aligned}
\end{equation}
By the dSDA~\eqref{eq:care-iteration}, with 
\begin{align*}
	& 
	E(Y_k^{(j)}) := [I_{2^k m} +  Y_k^{(j)}(Y_k^{(j)})^{\T}]^{-1}, 
	\qquad \quad 
	F(Y_k^{(j)}) := [I_{2^k l} + (Y_k^{(j)})^{\T} Y_k^{(j)}]^{-1}, 
	\\
	&\mathcal{X}_k^{U,(j)} := 
	\left[\mathcal{Q}_j^U, \widetilde A_{\gamma}^{2^j} \mathcal{Q}_j^U, 
	\cdots, \widetilde A_{\gamma}^{2^k-2^j} \mathcal{Q}_j^U\right], 
	\\
	&\mathcal{X}_k^{V,(j)} := 
	\left[\mathcal{Q}_j^V, (\widetilde A_{\gamma}^{\T})^{2^j} \mathcal{Q}_j^V, 
	\cdots, (\widetilde A_{\gamma}^{\T})^{2^k-2^j} \mathcal{Q}_j^V\right],
\end{align*}
it  produces the  following iterates: (for $k>j$) 
\begin{align}
	G_k^{(j)} 
	&=
	2\gamma 
	\mathcal{X}_k^{U,(j)}
	\left[I_{2^{k-j}}\otimes (\Theta_{1,j}^G)^{\T} M_{j-1}^G \right] 
	E(Y_k^{(j)})
	\left[I_{2^{k-j}}\otimes  (M_{j-1}^G)^{\T} \Theta_{1,j}^G \right]
	(\mathcal{X}_k^{U,(j)})^{\T}, \nonumber 
	\\
	H_k^{(j)} 
	&= 
	2\gamma 
	\mathcal{X}_k^{V,(j)}
	\left[I_{2^{k-j}}\otimes (\Theta_{1,j}^H)^{\T} M_{j-1}^H \right]
	F(Y_k^{(j)}) 
	\left[I_{2^{k-j}}\otimes  (M_{j-1}^H)^{\T} \Theta_{1,j}^H \right]
	(\mathcal{X}_k^{V,(j)})^{\T}, 
	\label{eq:H_k_j}
\end{align}
with 
$ 
	T_k^{(j)} = 
	[I_{2^{k-j}}\otimes (M_{j-1}^G)^{\T} \Theta_{1,j}^G]
	(\mathcal{X}_k^{U,(j)})^{\T} 
	\mathcal{X}_k^{V,(j)}
	[I_{2^{k-j}}\otimes (\Theta_{1,j}^H)^{\T} M_{j-1}^H], 
$ 
$Y_j^{(j)}\equiv Y_j^{(j-1)}$ and  
$Y_{k}^{(j)}=\begin{bmatrix}
	0& Y_{k-1}^{(j)}\\ Y_{k-1}^{(j)} & 2\gamma T_{k-1}^{(j)}
\end{bmatrix}$,   
satisfying       
\begin{equation}\label{eq:Sigma_j_H}
\begin{aligned}
	(\Theta_{1,j}^G)^{\T} M_{j-1}^G
	E(Y_j^{(j)}) 
	(M_{j-1}^G)^{\T} \Theta_{1,j}^G & \equiv (\Sigma_{1,j}^G)^2, 
	\\
	(\Theta_{1,j}^H)^{\T} M_{j-1}^H 
	F(Y_j^{(j)}) 
	(M_{j-1}^H)^{\T} \Theta_{1,j}^H  &\equiv (\Sigma_{1,j}^H)^2. 
\end{aligned}	
\end{equation}
%\begin{align}
	%&(\Theta_{1,j}^G)^{\T} M_{j-1}^G
	%E(Y_j^{(j)}) 
	%(M_{j-1}^G)^{\T} \Theta_{1,j}^G  \equiv (\Sigma_{1,j}^G)^2, 
	%\ \
	%(\Theta_{1,j}^H)^{\T} M_{j-1}^H 
	%F(Y_j^{(j)}) 
	%(M_{j-1}^H)^{\T} \Theta_{1,j}^H  \equiv (\Sigma_{1,j}^H)^2.   \label{eq:Sigma_j_H}
%\end{align}
As shown in Figure~\ref{fig:truncation-routine:}, we now  truncate $G_{j+1}^{(j)}$ and 
$H_{j+1}^{(j)}$ respectively to    $\widetilde G_{j+1}\equiv G_{j+1}^{(j+1)}$ 
and $\widetilde H_{j+1}\equiv H_{j+1}^{(j+1)}$,  then apply the dSDA in \eqref{eq:care-iteration} 
to produce  the iterations  $G_k^{(j+1)}$ 
and $H_{k}^{(j+1)}$ ($k>j+1$), where $G_{j+1}^{(j+1)}$ and $H_{j+1}^{(j+1)}$ are 
the initial iterates. 
From 
\eqref{eq:H_k_j} we have 
\begin{small}
	\begin{equation}\label{eq:H_j+1-j}
\begin{aligned}
	G_{j+1}^{(j)} = 2\gamma 
	\mathcal{X}_{j+1}^{U,(j)}
	\left[I_{2}\otimes (\Theta_{1,j}^G)^{\T} M_{j-1}^G \right] 
	E(Y_{j+1}^{(j)})
	\left[I_{2}\otimes  (M_{j-1}^G)^{\T} \Theta_{1,j}^G \right]
	(\mathcal{X}_{j+1}^{U,(j)})^{\T}, %\nonumber 
	\\
	H_{j+1}^{(j)} = 2\gamma 
	\mathcal{X}_{j+1}^{V,(j)}
	\left[I_{2}\otimes (\Theta_{1,j}^H)^{\T} M_{j-1}^H \right] 
	F(Y_{j+1}^{(j)}) 
	\left[I_{2}\otimes  (M_{j-1}^H)^{\T} \Theta_{1,j}^H \right]
	(\mathcal{X}_{j+1}^{V,(j)})^{\T}. %\label{eq:H_j+1-j} 
\end{aligned}
\end{equation}
\end{small}
Define $\Psi_j:= I_{2^j m}+Y_j^{(j)}(Y_j^{(j)})^{\T} + 
4\gamma^2 T_j^{(j)} [I_{2^j l}+(Y_j^{(j)})^{\T}Y_j^{(j)}]^{-1} (T_j^{(j)})^{\T}$ 
and  $\Omega_j:= (\mathcal{Q}_j^U)^{\T} \mathcal{Q}_j^V$.  
Since  
\eqref{eq:Sigma_j_H} and the Sherman-Morrison-Woodbury formula (SMWF) indicate  
\begin{align*}
	&
	(\Theta_{1,j}^G)^{\T} M_{j-1}^G \Psi_j^{-1}
	(M_{j-1}^G)^{\T} \Theta_{1,j}^G
	\\
	=& 
	(\Sigma_{1,j}^G)^2 -4\gamma^2 (\Sigma_{1,j}^G)^2 
	\left[I_{r_j^G} + 4\gamma^2 \Omega_j 
	(\Sigma_{1,j}^H)^2 (\Omega_j)^{\T}(\Sigma_{1,j}^G)^2\right]^{-1}
	\Omega_j (\Sigma_{1,j}^H)^2 \Omega_j^{\T} (\Sigma_{1,j}^G)^2
	\\
	\equiv&
	\Sigma_{1,j}^G\left[I_{r_j^G}+4\gamma^2 
		\Sigma_{1,j}^G \Omega_j (\Sigma_{1,j}^H)^2 \Omega_j^{\T} 
	\Sigma_{1,j}^G\right]^{-1} \Sigma_{1,j}^G, 
\end{align*} 
then with 
\begin{align}\label{eq:L-j-UV}
	L_j^{G}&:=2\gamma(\Theta_{1,j}^G)^{\T} M_{j-1}^G 
	\left[I_{2^jm}+Y_j^{(j)}(Y_j^{(j)})^{\T}\right]^{-1}Y_j^{(j)}
	(M_{j-1}^H)^{\T} \Theta_{1,j}^H \Omega_j^{\T}, 
\end{align}
we deduce that 
	\begin{equation}\label{eq:inverse_YZ_j}
	\begin{aligned}
		&
		\left[I_2\otimes (\Theta_{1,j}^G)^{\T} M_{j-1}^G\right] 
		E(Y_{j+1}^{(j)}) 
		\left[I_2\otimes  (M_{j-1}^G)^{\T} \Theta_{1,j}^G\right]%\nonumber
		\\
		=&
		\begin{multlined}[t]
			\begin{bmatrix}
				I_{r_j^G}&
				-L_j^{G}
				\\ 0&I_{r_j^G}   
			\end{bmatrix}
			\left[I_2\otimes (\Theta_{1,j}^G)^{\T} M_{j-1}^G\right]
			\left[ 
				E(Y_{j}^{(j)}) 
				\oplus \Psi_j^{-1} 
			\right]
			\left[I_2\otimes  (M_{j-1}^G)^{\T} \Theta_{1,j}^G\right]
			\\
			\cdot
			\begin{bmatrix}
				I_{r_j^G}& -L_j^{G} \\ 
				0 & I_{r_j^G}
			\end{bmatrix}^{\T}
		\end{multlined}%\nonumber
		\\
		\equiv&
		\begin{multlined}[t]
			\begin{bmatrix}
				I_{r_j^G}&
				-L_j^{G}
				\\ 0&I_{r_j^G}   
			\end{bmatrix}
			\left\{
				(\Sigma_{1,j}^G)^2 \oplus 
				\left[
					\Sigma_{1,j}^G
					(I_{r_j^G}+4\gamma^2 \Sigma_{1,j}^G \Omega_j 
				(\Sigma_{1,j}^H)^2 \Omega_j^{\T} \Sigma_{1,j}^G )^{-1} \Sigma_{1,j}^G \right] 
			\right\}
			\\
			\cdot
			\begin{bmatrix}
				I_{r_j^G} &	-L_j^{G} \\
				0 & I_{r_j^G}
			\end{bmatrix}^{\T}. 
		\end{multlined} %\label{eq:inverse_YZ_j}  
	\end{aligned} 
\end{equation}
Similarly, with $L_j^{H}:=2\gamma(\Theta_{1,j}^H)^{\T} M_{j-1}^H
\left[I_{2^jl}+(Y_j^{(j)})^{\T} Y_j^{(j)}\right]^{-1}
(Y_j^{(j)})^{\T} (M_{j-1}^G)^{\T} \Theta_{1,j}^G \Omega_j$, we obtain 
\begin{equation}\label{eq:inverse_ZY_j}
	\begin{aligned} 
		&\left[I_2\otimes (\Theta_{1,j}^H)^{\T} M_{j-1}^H\right]
		F(Y_{j+1}^{(j)}) 
		\left[I_2\otimes  (M_{j-1}^H)^{\T} \Theta_{1,j}^H\right]% \nonumber
		\\
		\equiv&
		\begin{multlined}[t]
			\begin{bmatrix}
				I_{r_j^H}&-L_j^{H}	\\ 0&I_{r_j^H}      
			\end{bmatrix}
			\left\{	
				(\Sigma_{1,j}^H)^2 \oplus  
				\left[ \Sigma_{1,j}^H
					(I_{r_j^H}+4\gamma^2  \Sigma_{1,j}^H \Omega_j^{\T}
				(\Sigma_{1,j}^G)^2 \Omega_j \Sigma_{1,j}^H )^{-1} \Sigma_{1,j}^H \right]
			\right\} 		
			\\
			\cdot
			\begin{bmatrix}
				I_{r_j^H}& -L_j^{H} \\
				0 & I_{r_j^H}
			\end{bmatrix}^{\T}.
		\end{multlined} %\label{eq:inverse_ZY_j}  
	\end{aligned}
\end{equation}
By the  modified Gram-Schmidt process, compute the QR factorizations:
\begin{equation}\label{eq:qrj_U}
	\begin{aligned}
	\left[ \mathcal{Q}_j^U, \widetilde A_{\gamma}^{2^j} \mathcal{Q}_j^U \right]
	&=\left[ \mathcal{Q}_j^U, Q_{j+1}^U \right]
	\begin{bmatrix}
		I_{r_j^G} & R_{12}^{j,U} \\ 0 & R_2^{j,U}
	\end{bmatrix}, 
	\\
	\left[ \mathcal{Q}_j^V, (\widetilde A_{\gamma}^{\T})^{2^j} \mathcal{Q}_j^V \right] 
	&=\left[ \mathcal{Q}_j^V, Q_{j+1}^V \right]
	\begin{bmatrix}
		I_{r_j^H} & R_{12}^{j,V} \\ 0 & R_2^{j,V}
	\end{bmatrix}, 
	\end{aligned}
\end{equation}
where $Q_{j+1}^U\in \mathbb{R}^{n\times (p_{j+1}-r_j^G)}$,   
$Q_{j+1}^V\in \mathbb{R}^{n\times (q_{j+1}-r_j^H)}$, and 
$R_{12}^{j,U}\in \mathbb{R}^{r_j^G\times r_j^G}$, 
$R_2^{j,U}\in \mathbb{R}^{(p_{j+1}-r_j^G)\times r_j^G}$, 
$R_{12}^{j,V}\in \mathbb{R}^{r_j^H\times r_j^H}$,  
$R_2^{j,V}\in \mathbb{R}^{(q_{j+1}-r_j^H)\times r_j^H}$. 
With the SVD: 
$\Sigma_{1,j}^G \Omega_j \Sigma_{1,j}^H = U_{j+1}^Y \Sigma_{j+1}^Y (V_{j+1}^Y)^{\T}$, 
where $U_{j+1}^Y\in \mathbb{R}^{r_j^G \times r_j^G}$, 
$\Sigma_{j+1}^Y\in \mathbb{R}^{r_j^G \times r_j^H}$,  
$V_{j+1}^Y\in \mathbb{R}^{r_j^H \times r_j^H}$, 
and $\Upsilon_{j+1}^G := I_{r_j^G} + 4\gamma^2 \Sigma_{j+1}^Y(\Sigma_{j+1}^Y)^{\T}$,  
$\Upsilon_{j+1}^H := I_{r_j^H} + 4\gamma^2 (\Sigma_{j+1}^Y)^{\T} \Sigma_{j+1}^Y$,  
we calculate  the SVDs: 
\begin{equation}\label{eq:svdH_j+1}
	\begin{aligned}
		&
		\begin{bmatrix}
			I_{r_j^G}&R_{12}^{j,U}\\0&R_2^{j,U}
		\end{bmatrix}
		\begin{bmatrix}
			I_{r_j^G} & -L_j^G\\ 0 & I_{r_j^G}
		\end{bmatrix}
		\left\{
			\Sigma_{1,j}^G \oplus 
			\left[\Sigma_{1,j}^G U_{j+1}^Y (\Upsilon_{j+1}^G)^{-1/2}\right]
		\right\} 
		\\
		=&
		\Theta_{j+1}^G \Sigma_{j+1}^G (\Phi_{j+1}^G)^{\T}, %\nonumber 
		\\
		& 
		\begin{bmatrix}
			I_{r_j^H}&R_{12}^{j,V}\\0&R_2^{j,V}
		\end{bmatrix}
		\begin{bmatrix}
			I_{r_j^H} & -L_j^H\\ 0 & I_{r_j^H}
		\end{bmatrix}
		\left\{ 
			\Sigma_{1,j}^H \oplus 
			\left[ \Sigma_{1,j}^H V_{j+1}^Y (\Upsilon_{j+1}^{H})^{-1/2} \right]
		\right\} 
		\\
		=&
		\Theta_{j+1}^H \Sigma_{j+1}^H (\Phi_{j+1}^H)^{\T}, %\label{eq:svdH_j+1} 
	\end{aligned}
\end{equation}
where $\Theta_{j+1}^G\in \mathbb{R}^{p_{j+1}\times p_{j+1}}$, 
$\Sigma_{j+1}^G\in \mathbb{R}^{p_{j+1}\times p_{j+1}}$, 
$\Phi_{j+1}^G\in \mathbb{R}^{2r_j^G\times p_{j+1}}$ and 
$\Theta_{j+1}^H\in \mathbb{R}^{q_{j+1}\times q_{j+1}}$, 
$\Sigma_{j+1}^H\in \mathbb{R}^{q_{j+1}\times q_{j+1}}$, 
$\Phi_{j+1}^H\in \mathbb{R}^{2r_j^H \times q_{j+1}}$. 
Write $\widehat{\mathcal{Q}}_{j+1}^U:=[ \mathcal{Q}_j^U, Q_{j+1}^U ]$ 
and $\widehat{\mathcal{Q}}_{j+1}^V:=[ \mathcal{Q}_j^V, Q_{j+1}^V ]$, 
we subsequently obtain  
\begin{align*}
	G_{j+1}^{(j)}&=2\gamma \widehat{\mathcal{Q}}_{j+1}^U
	\Theta_{j+1}^G (\Sigma_{j+1}^G)^2 (\Theta_{j+1}^G)^{\T}
	(\widehat{\mathcal{Q}}_{j+1}^U)^{\T},
	\\   
	H_{j+1}^{(j)}&=2\gamma \widehat{\mathcal{Q}}_{j+1}^V
	\Theta_{j+1}^H (\Sigma_{j+1}^H)^2 (\Theta_{j+1}^H)^{\T}
	(\widehat{\mathcal{Q}}_{j+1}^V)^{\T}. 
\end{align*}

Let  $\varepsilon_{j+1}$ be  a small tolerance  and write 
$\Sigma_{j+1}^G =\Sigma_{1,j+1}^G \oplus \Sigma_{2,j+1}^G$,  
$\Sigma_{j+1}^H =\Sigma_{1,j+1}^H \oplus \Sigma_{2,j+1}^H$,  
where $\Sigma_{1,j+1}^G \in \mathbb{R}^{r_{j+1}^G\times r_{j+1}^G}$, 
$\Sigma_{1,j+1}^H \in \mathbb{R}^{r_{j+1}^H\times r_{j+1}^H}$, satisfying 
\[
\|\Sigma_{2,j+1}^G\|\leq \varepsilon_{j+1} \|\Sigma_{1,j+1}^G\|, \quad  
\|\Sigma_{2,j+1}^H\|\leq \varepsilon_{j+1} \|\Sigma_{1,j+1}^H\|.
\]  
Partition $\Theta_{j+1}^G=[\Theta_{1,j+1}^G, \Theta_{2,j+1}^G ]$,  
$\Theta_{j+1}^H=[\Theta_{1,j+1}^H, \Theta_{2,j+1}^H ]$,   
$\Phi_{j+1}^G=[\Phi_{1,j+1}^G, \Phi_{2,j+1}^G ]$ and  
$\Phi_{j+1}^H=[\Phi_{1,j+1}^H, \Phi_{2,j+1}^H ]$,    
compatibly with those in $\Sigma_{j+1}^G$ and $\Sigma_{j+1}^H$, 
with $\Theta_{1,j+1}^G \in \mathbb{R}^{p_{j+1}\times r_{j+1}^G}$, 
$\Theta_{1,j+1}^H \in \mathbb{R}^{q_{j+1}\times r_{j+1}^H}$, 
$\Phi_{1,j+1}^G \in \mathbb{R}^{2r_j^G \times r_{j+1}^G}$ and 
$\Phi_{1,j+1}^H \in \mathbb{R}^{2r_j^H \times r_{j+1}^H}$. 
With respect   to $\varepsilon_{j+1}$, 
$G_{j+1}^{(j)}$ and $H_{j+1}^{(j)}$ are truncated respectively to 
\begin{equation}\label{eq:wtdH_j+1}
\begin{aligned}
	\widetilde G_{j+1}\equiv G_{j+1}^{(j+1)}
	=2\gamma \widehat{\mathcal{Q}}_{j+1}^{U} 
	\Theta_{1,j+1}^G (\Sigma_{1,j+1}^G)^2 (\Theta_{1,j+1}^G)^{\T}
	(\widehat{\mathcal{Q}}_{j+1}^{U})^{\T}, %\nonumber 
	\\ 
	\widetilde H_{j+1}\equiv H_{j+1}^{(j+1)}
	= 2\gamma  \widehat{\mathcal{Q}}_{j+1}^{V} 
	\Theta_{1,j+1}^H (\Sigma_{1,j+1}^H)^2 (\Theta_{1,j+1}^H)^{\T}
	(\widehat{\mathcal{Q}}_{j+1}^{V})^{\T}.  %\label{eq:wtdH_j+1}
\end{aligned}
\end{equation}

Next we reformulate $G_{j+1}^{(j+1)}$ and $H_{j+1}^{(j+1)}$ 
and then  generate $G_k^{(j+1)}$ and $H_k^{(j+1)}$ by the dSDA~\eqref{eq:care-iteration}, 
starting from $G_{j+1}^{(j+1)}$ and $H_{j+1}^{(j+1)}$. 
Define 
\begin{equation}\label{eq:M_j_GH}
	\begin{aligned}%\label{eq:M_j_GH}
		M_j^G&:=
		\begin{bmatrix}
			I_{r_j^G} & R_{12}^{j,U}\\ 0& R_2^{j,U}
		\end{bmatrix}
		\left[I_2\otimes (\Theta_{1,j}^G)^{\T} M_{j-1}^G\right], 
		\\ 
		M_j^H&:=
		\begin{bmatrix}
			I_{r_j^H} & R_{12}^{j,V}\\ 0& R_2^{j,V}
		\end{bmatrix}
		\left[I_2\otimes (\Theta_{1,j}^H)^{\T} M_{j-1}^H\right], 
	\end{aligned}
\end{equation}
and   denote 
\begin{align*}
	\widehat{\mathcal{X}}_k^{U,(j+1)} &:= 
	\left[ \widehat{\mathcal{Q}}_{j+1}^U, 
		\widetilde A_{\gamma}^{2^{j+1}} \widehat{\mathcal{Q}}_{j+1}^U, 
	\cdots, \widetilde A_{\gamma}^{2^k-2^{j+1}} \widehat{\mathcal{Q}}_{j+1}^U \right], 
	\\
	\widehat{\mathcal{X}}_k^{V,(j+1)} &:= 
	\left[ \widehat{\mathcal{Q}}_{j+1}^V, 
		(\widetilde A_{\gamma}^{\T})^{2^{j+1}} \widehat{\mathcal{Q}}_{j+1}^V, 
	\cdots, (\widetilde A_{\gamma}^{\T})^{2^k-2^{j+1}} \widehat{\mathcal{Q}}_{j+1}^V \right].
\end{align*}
Then $G_k^{(j)}$ and $H_k^{(j)}$ in \eqref{eq:H_k_j} 
can be rewritten as: 
\begin{equation}\label{eq:H_k_j_form2}
\begin{aligned}  
	G_k^{(j)} 
	&= 
	\begin{multlined}[t]
		2\gamma \widehat{\mathcal{X}}_k^{U,(j+1)}
		(I_{2^{k-j-1}}\otimes M_{j}^G) 
		E(Y_k^{(j)}) 
		\left(I_{2^{k-j-1}}\otimes  M_{j}^G \right)^{\T}
		(\widehat{\mathcal{X}}_k^{U,(j+1)})^{\T},  
	\end{multlined}
	\\
	H_k^{(j)} 
	&= 
	\begin{multlined}[t]
		2\gamma \widehat{\mathcal{X}}_k^{V,(j+1)}  
		(I_{2^{k-j-1}}\otimes  M_{j}^H)
		F(Y_k^{(j)}) 
		\left(I_{2^{k-j-1}}\otimes  M_{j}^H \right)^{\T}
		(\widehat{\mathcal{X}}_k^{V,(j+1)})^{\T}. 
	\end{multlined} %\label{eq:H_k_j_form2}
\end{aligned}
\end{equation}
It follows from \eqref{eq:inverse_YZ_j}, \eqref{eq:inverse_ZY_j},  
\eqref{eq:svdH_j+1} and the definitions 
of $M_j^G$ and $M_j^H$ in~\eqref{eq:M_j_GH}  that  
\begin{align*}
	\Theta_{1,j+1}^G (\Sigma_{1,j+1}^G)^2 (\Theta_{1,j+1}^G)^{\T}
	&=
	\widehat{M}_j^G 
	E(Y_{j+1}^{(j)}) 
	(\widehat{M}_j^G)^{\T}, 
	\\ 
	\Theta_{1,j+1}^H (\Sigma_{1,j+1}^H)^2 (\Theta_{1,j+1}^H)^{\T}
	&=
	\widehat{M}_j^{H}
	F(Y_{j+1}^{(j)}) 
	(\widehat{M}_j^H)^{\T},
\end{align*} 
where $\widehat{M}_j^G:= \Theta_{j+1}^G (I_{r_{j+1}^G}\oplus 0)
(\Theta_{j+1}^G)^{\T} M_j^G$ and $\widehat{M}_j^H:=	\Theta_{j+1}^H (I_{r_{j+1}^H}\oplus 0)
(\Theta_{j+1}^H)^{\T} M_j^H$.  
As a result,  we  can   reformulate 
\begin{equation}\label{eq:Hj+1}
	\begin{aligned}
		G_{j+1}^{(j+1)} 
		&=
		2\gamma \widehat{\mathcal{Q}}_{j+1}^U
		\widehat{M}_j^G 
		E(Y_{j+1}^{(j)})
		(\widehat{M}_j^G)^{\T}
		\left(\widehat{\mathcal{Q}}_{j+1}^U\right)^{\T},  %\nonumber 
		\\ 
		H_{j+1}^{(j+1)} 
		&=
		2\gamma \widehat{\mathcal{Q}}_{j+1}^V
		\widehat{M}_j^H
		F(Y_{j+1}^{(j)}) 
		(\widehat{M}_j^H)^{\T} 
		\left(\widehat{\mathcal{Q}}_{j+1}^V\right)^{\T}.  %\label{eq:Hj+1}
	\end{aligned}
\end{equation}
Now let  
\begin{align}\label{eq:Aj+1}
	A_{j+1}^{(j+1)} = \widetilde{A}_{\gamma}^{2^{j+1}} - 2\gamma  \widehat{\mathcal{Q}}_{j+1}^U
	\widehat{M}_j^G 
	[I_{2^{j+1}m} +  Y_{j+1}^{(j)}(Y_{j+1}^{(j)})^{\T} ]^{-1}
	Y_{j+1}^{(j)} (\widehat{M}_j^H)^{\T}  
	\left(\widehat{\mathcal{Q}}_{j+1}^V\right)^{\T}.	
\end{align}
Starting from $A_{j+1}^{(j+1)}$,  $G_{j+1}^{(j+1)}$ and $H_{j+1}^{(j+1)}$, 
similar to \eqref{eq:H_k_j_form2},   
the dSDA~\eqref{eq:care-iteration}  produces the iterations: (for $k>j+1$)
	\begin{align*}
		G_k^{(j+1)} 
		&= 
		2\gamma \widehat{\mathcal{X}}_{k}^{U,(j+1)}
		(I_{2^{k-j-1}}\otimes \widehat{M}_j^G)
		E(Y_k^{(j+1)}) 
		(I_{2^{k-j-1}}\otimes \widehat{M}_j^G)^{\T}
		(\widehat{\mathcal{X}}_{k}^{U,(j+1)})^{\T} 
		\\
		&\equiv
		\begin{multlined}[t]
			2\gamma \mathcal{X}_k^{U,(j+1)}
			\left[I_{2^{k-j-1}}\otimes	\left( (\Theta_{1,j+1}^G)^{\T} M_{j}^G \right)\right]
			E(Y_k^{(j+1)}) 
			\\
			\cdot
			\left[I_{2^{k-j-1}}\otimes \left(( M_{j}^G)^{\T}\Theta_{1,j+1}^G\right)\right]
			(\mathcal{X}_k^{U,(j+1)})^{\T}, 
		\end{multlined}
	\end{align*}
\begin{align*}
	H_k^{(j+1)} 
	&= 
	2\gamma 
	\widehat{\mathcal{X}}_{k}^{V,(j+1)}
	(I_{2^{k-j-1}}\otimes \widehat{M}_j^H)
	F(Y_k^{(j+1)})
	(I_{2^{k-j-1}}\otimes \widehat{M}_j^H)^{\T}
	(\widehat{\mathcal{X}}_{k}^{V,(j+1)})^{\T} 
	\\
	&\equiv
	\begin{multlined}[t]
		2\gamma 
		\mathcal{X}_k^{V,(j+1)}
		\left[I_{2^{k-j-1}}\otimes \left((\Theta_{1,j+1}^H)^{\T}  M_{j}^H \right)\right]
		F(Y_k^{(j+1)}) 
		\\
		\cdot
		\left[I_{2^{k-j-1}}\otimes \left((M_{j}^H)^{\T} \Theta_{1,j+1}^H\right) \right]
		(\mathcal{X}_{k}^{V,(j+1)})^{\T}, 
	\end{multlined}
\end{align*}
where
\[
E(Y_k^{(j+1)}):= [I_{2^km} +  Y_k^{(j+1)} (Y_k^{(j+1)})^{\T}]^{-1}, 
\quad 
F(Y_k^{{j+1}}):=[I_{2^k l} + (Y_k^{(j+1)})^{\T}Y_k^{(j+1)}]^{-1} 
\]
with $Y_{j+1}^{(j+1)}\equiv Y_{j+1}^{(j)}$,   
$Y_{k}^{(j+1)}=\begin{bmatrix}
	0& Y_{k-1}^{(j+1)}\\ Y_{k-1}^{(j+1)} & 2\gamma T_{k-1}^{(j+1)}
\end{bmatrix}$,  
\begin{align*}
	T_k^{(j+1)} = 
	\begin{multlined}[t]
		\left[I_{2^{k-j-1}}\otimes \left((M_j^G)^{\T} \Theta_{1,j+1}^G\right)\right]
		(\mathcal{X}_k^{U,(j+1)})^{\T}
		\mathcal{X}_k^{V,(j+1)}
		\\
		\cdot
		\left[I_{2^{k-j-1}}\otimes \left( (\Theta_{1,j+1}^H)^{\T} M_{j}^H\right)\right], 
	\end{multlined}
\end{align*}
and 
\begin{align*}
	&
	\mathcal{Q}_{j+1}^U:=\widehat{\mathcal{Q}}_{j+1}^U \Theta_{1,j+1}^G, 
	\qquad 
	\mathcal{X}_{k}^{U,(j+1)} := \left[\mathcal{Q}_{j+1}^U,   
		\widetilde{A}_{\gamma}^{2^{j+1}}\mathcal{Q}_{j+1}^U,   
	\cdots, \widetilde{A}_{\gamma}^{2^k-2^{j+1}}\mathcal{Q}_{j+1}^U \right], 
	\\
	& 
	\mathcal{Q}_{j+1}^V:= \widehat{\mathcal{Q}}_{j+1}^V \Theta_{1,j+1}^H, 
	\qquad 
	\mathcal{X}_{k}^{V,(j+1)} := \left[	\mathcal{Q}_{j+1}^V,   
		(\widetilde{A}_{\gamma}^{\T})^{2^{j+1}}\mathcal{Q}_{j+1}^V,   
	\cdots,  (\widetilde{A}_{\gamma}^{\T})^{2^k-2^{j+1}}\mathcal{Q}_{j+1}^V 	\right].
\end{align*} 

Evidently, the above iterate recursions for  $G_k^{(j+1)}$ and $H_{k}^{(j+1)}$ 
are quite similar to those for  $G_k^{(j)}$ and $H_k^{(j)}$ in~\eqref{eq:H_k_j}. 
One thing  left  is the identities analogous to
\eqref{eq:Sigma_j_H} for the index $j+1$. 

By~\eqref{eq:inverse_YZ_j} and \eqref{eq:svdH_j+1}, 
it is simple to check  that 
\begin{align*}
	&
	M_{j}^G 
	E(Y_{j+1}^{(j+1)}) 
	(M_{j}^G)^{\T} 
	\\
	=&
	\begin{multlined}[t]
		\begin{bmatrix}
			I_{r_j^G} & R_{12}^{j,U}\\ 0& R_2^{j,U}
		\end{bmatrix}
		\left[I_2\otimes \left( (\Theta_{1,j}^G)^{\T} M_{j-1}^G\right)\right]
		E(Y_{j+1}^{(j+1)}) 
		\left[I_2\otimes \left( (M_{j-1}^G)^{\T} \Theta_{1,j}^G\right)\right]
		\begin{bmatrix}
			I_{r_j^G} & R_{12}^{j,U}\\ 0& R_2^{j,U}
		\end{bmatrix}^{\T} 
	\end{multlined} 
	\\
	=&
	\begin{multlined}[t]
		\begin{bmatrix}
			I_{r_j^G} & R_{12}^{j,U}\\ 0& R_2^{j,U}
		\end{bmatrix}
		\begin{bmatrix}
			I_{r_j^G}&
			-L_j^{G}
			\\ 0&I_{r_j^G}   
		\end{bmatrix} 
		\\
		\cdot
		\left\{(\Sigma_{1,j}^G)^2 \oplus  
			\left[ \Sigma_{1,j}^G
				\left( I_{r_j^G}+4\gamma^2 \Sigma_{1,j}^G \Omega_j (\Sigma_{1,j}^H)^2 \Omega_j^{\T} 
		\Sigma_{1,j}^G \right)^{-1} \Sigma_{1,j}^G \right] \right\}
		\\
		\cdot
		\begin{bmatrix}
			I_{r_j^G}& -L_j^G
			\\
			0& I_{r_j^G}
		\end{bmatrix}^{\T}
		\begin{bmatrix}
			I_{r_j^G} & R_{12}^{j,U}\\ 0& R_2^{j,U}
		\end{bmatrix}^{\T} 
	\end{multlined} \nonumber
	\\
	=&
	\Theta_{j+1}^G (\Sigma_{j+1}^G)^2 
	(\Theta_{j+1}^G)^{\T}. \nonumber  
\end{align*}
This and similar techniques imply that
\begin{equation}\label{eq:Simga_j+1_H}
	\begin{aligned}
		(\Theta_{1,j+1}^G)^{\T}
		M_{j}^G 
		E(Y_{j+1}^{(j+1)}) 
		(M_{j}^G)^{\T}
		\Theta_{1,j+1}^G 
		\equiv 
		(\Sigma_{1,j+1}^G)^2.  
		%\nonumber 
		\\
		(\Theta_{1,j+1}^H)^{\T} M_{j}^H 
		F(Y_{j+1}^{(j+1)}) 
		(M_{j}^H)^{\T} \Theta_{1,j+1}^H \equiv (\Sigma_{1,j+1}^H)^2. %\label{eq:Simga_j+1_H}
	\end{aligned}
\end{equation}

\begin{remark}\label{rk:trun-recusion}
	The truncation forms $\widetilde{G}_{j+1}$  
	and $\widetilde{H}_{j+1}$ in~\eqref{eq:wtdH_j+1} are respectively 
	similarly to $\widetilde{G}_j$ 
	and $\widetilde{H}_j$ in \eqref{eq:wtdH_j}. 
	The decoupled doubling recursions  on $G_k^{(j+1)}$ and $H_k^{(j+1)}$ are 
	in the same form as  $G_k^{(j)}$ 	and $H_k^{(j)}$ 
	given in~\eqref{eq:H_k_j}. Also, the equalities in 
	\eqref{eq:Simga_j+1_H} follow the relationships specified in  
	\eqref{eq:Sigma_j_H}. The general formulae of the truncation  
	are displayed in \eqref{eq:wtdH_j}, \eqref{eq:H_k_j} and \eqref{eq:Sigma_j_H}.     
\end{remark}

\subsubsection{Computing \texorpdfstring{$L_j^{G}$}{LjG} and \texorpdfstring{$L_j^{H}$}{LjH}}
\label{sssec:computing-l_j_H_l_j_g}

From Section~\ref{sssec:truncating-g_2-and-h_2-}, to truncate 
$G_2^{(1)}$ and $H_2^{(1)}$  to $\widetilde{G}_2$ and $\widetilde{H}_2$, 
we need to compute  $L_1^G$ and $L_1^H$. 
For the general case in each truncation step, we are required  to calculate  $L_j^G$ and $L_j^H$. 
Specifically, by \eqref{eq:H_j+1-j}, \eqref{eq:inverse_YZ_j} 
and \eqref{eq:inverse_ZY_j}, we have 
	\begin{align*}
		G_{j+1}^{(j)} &=
		\begin{multlined}[t]
			2\gamma 
			\mathcal{X}_{j+1}^{U,(j)}	
			\begin{bmatrix}
				I_{r_j^G}&
				-L_j^{G}
				\\ 0&I_{r_j^G}   
			\end{bmatrix}
			\left[ 
				(\Sigma_{1,j}^G)^2  \oplus 
				\left( \Sigma_{1,j}^G U_{j+1}^Y (\Upsilon_{j+1}^G)^{-1} (U_{j+1}^Y)^{\T}  
				\Sigma_{1,j}^G \right) 
			\right] 
			\\ \cdot
			\begin{bmatrix}
				I_{r_j^G}& -L_j^{G}
				\\ 0& I_{r_j^G}
			\end{bmatrix}^{\T} (\mathcal{X}_{j+1}^{U,(j)})^{\T},
		\end{multlined}	
		\\
		H_{j+1}^{(j)} &=
		\begin{multlined}[t]
			2\gamma 
			\mathcal{X}_{j+1}^{V,(j)}
			\begin{bmatrix}
				I_{r_j^H}&
				-L_j^{H}
				\\ 0&I_{r_j^H}   
			\end{bmatrix}  
			\left[ (\Sigma_{1,j}^H)^2   \oplus 
				\left( \Sigma_{1,j}^H V_{j+1}^Y (\Upsilon_{j+1}^H)^{-1} (V_{j+1}^Y)^{\T} 
				\Sigma_{1,j}^H \right) 
			\right] 
			\\ \cdot 
			\begin{bmatrix}
				I_{r_j^H}& -L_j^{H}\\ 
				0& I_{r_j^H}
			\end{bmatrix}^{\T} (\mathcal{X}_{j+1}^{V,(j)})^{\T},
		\end{multlined}	
	\end{align*} 
where $\mathcal{X}_{j+1}^{U,(j)}=\left[ \mathcal{Q}_j^U, \widetilde A_{\gamma}^{2^j}\mathcal{Q}_j^U \right]$, 
$\mathcal{X}_{j+1}^{V,(j)}=\left[ \mathcal{Q}_j^V, (\widetilde{A}_{\gamma}^{\T})^{2^{j}}
\mathcal{Q}_j^V \right]$, 
\[
\Upsilon_{j+1}^G = I_{r_j^G} + 4\gamma^2 \Sigma_{j+1}^Y (\Sigma_{j+1}^Y)^{\T}, 
\quad
\Upsilon_{j+1}^H = I_{r_j^H} + 4\gamma^2 (\Sigma_{j+1}^Y)^{\T} \Sigma_{j+1}^Y
\]
with $\Sigma_{1,j}^G \Omega_j \Sigma_{1,j}^H = U_{j+1}^Y \Sigma_{j+1}^Y (V_{j+1}^Y)^{\T}$.  
Consequently, to obtain the truncated iterates $\widetilde{G}_{j+1}\equiv G_{j+1}^{(j+1)}$ 
and $\widetilde{H}_{j+1}\equiv H_{j+1}^{(j+1)}$, 
we need the recursion formulae for  $L_j^G$ and $L_j^H$, which we deduce below.   

As mentioned before, we aim to compute $\widetilde{G}_{j+1}=G_{j+1}^{(j+1)} $ 
directly from $\widetilde{G}_j=G_{j}^{(j)}$ without performing the intermediate step for $G_{j+1}^{(j)}$ explicitly. 
This follows from the fact that we can compute $L_{j+1}^G$ (or $L_{j+1}^H$)  
from $L_j^G$ (or $L_j^H$) directly.  We  display   
 the relationship  between  $L_j^G$ and $L_{j+1}^G$ ($L_j^H$ and $L_{j+1}^H$)     
in the following lemmas. 

\begin{lemma}\label{lm:L_1}
	Define $K_1^G:=(\Phi_{1,1}^G)^{\T} \Sigma_1^Y \Phi_{1,1}^H$ and $K_1^H:= (K_1^G)^{\T}$, 
	it  holds that 
	\begin{align*}
		L_1^G
		= 2\gamma \Sigma_{1,1}^G K_1^G \Sigma_{1,1}^H \Omega_1^{\T}, 
		\qquad
		L_1^H 
		= 2\gamma \Sigma_{1,1}^H K_1^H \Sigma_{1,1}^G \Omega_1. 
	\end{align*}
\end{lemma}
\begin{proof}
	Since $(\Sigma_1^H)^{\T} = \Sigma_1^H$ and  
	\begin{align}
		&
		M_0^G
		\left[I_{2m} + Y_1^{(1)} (Y_1^{(1)})^{\T}\right]^{-1} 
		Y_1^{(1)} (M_0^H)^{\T}  \nonumber
		\\
		=&
		(\Theta_{1,1}^G)^{\T} R_1^U P_1^U 
		\left[I_{2m} + Y_1^{(1)} (Y_1^{(1)})^{\T}\right]^{-1} 
		Y_1^{(1)} (P_1^V)^{\T} (R_1^V)^{\T} \Theta_{1,1}^H \nonumber
		\\
		=&
		(\Theta_{1,1}^G)^{\T} R_1^U P_1^U U_1^Y 
		\left[I_{2m}+\Sigma_1^Y(\Sigma_1^Y)^{\T}\right]^{-1}
		\Sigma_1^Y (V_1^Y)^{\T}
		(P_1^V)^{\T}	(R_1^V)^{\T} \Theta_{1,1}^H \nonumber
		\\
		=&	
		(\Theta_{1,1}^G)^{\T} \Theta_1^G \Sigma_1^G (\Phi_1^G)^{\T} 
		\Sigma_1^Y \Phi_1^H (\Sigma_1^H)^{\T} (\Theta_1^H)^{\T} \Theta_{1,1}^H  
		\equiv 
		\Sigma_{1,1}^G K_1^G \Sigma_{1,1}^H, \label{eq:L-2-12}
	\end{align}
	we  have  
	\begin{align*}
		L_1^G
		=
		2\gamma (\Theta_{1,1}^G)^{\T} R_1^U P_1^U 
		\left[I_{2m}+ Y_1Y_1^{\T}\right]^{-1} 
		Y_1 
		(P_1^V)^{\T} (R_1^V)^{\T} \Theta_{1,1}^H \Omega_1^{\T} 
		&\equiv	 2\gamma \Sigma_{1,1}^G K_1^G \Sigma_{1,1}^H \Omega_1^{\T}.
	\end{align*}
	Similarly, we have the result for $L_1^H$. 
\end{proof}
\begin{lemma}\label{lm:L_j}
	It holds that 
	\begin{equation}\label{eq:Lj_GH}
		\begin{aligned}
			L_j^G = 2\gamma  \Sigma_{1,j}^G  K_j^G 
			\Sigma_{1,j}^H \Omega_j^{\T}, 
			\qquad
			L_j^H = 2\gamma \Sigma_{1,j}^H  K_j^H 
			\Sigma_{1,j}^G \Omega_j, 
		\end{aligned}
	\end{equation}
	where $\Omega_j = (\mathcal{Q}_j^U)^{\T} \mathcal{Q}_j^V$, $K_j^H \equiv (K_j^G)^{\T}$ 
	with $\Upsilon_j^G := I_{r_{j-1}^G}+4\gamma^2 \Sigma_j^Y(\Sigma_j^Y)^{\T}$, 
    $\Upsilon_j^H := I_{r_{j-1}^H} + 4\gamma^2 (\Sigma_j^Y)^{\T}\Sigma_j^Y$ and 
	\begin{align}
		K_j^G 
		&= 
		(\Phi_{1,j}^G)^{\T}
		\begin{bmatrix}
			2\gamma K_{j-1}^G \Sigma_{1,j-1}^H \Omega_{j-1}^{\T} \Sigma_{1,j-1}^G K_{j-1}^G 
			& K_{j-1}^G V_j^Y (\Upsilon_j^H)^{1/2}
			\\
			(\Upsilon_j^G)^{1/2}
			(U_j^Y)^{\T} K_{j-1}^G 
			& 2\gamma \Sigma_j^Y
		\end{bmatrix} 
		\Phi_{1,j}^H \nonumber
		\\
		&\equiv
		\begin{multlined}[t]
			(\Phi_{1,j}^G)^{\T} 
		(K_{j-1}^G \oplus I_{r_{j-1}^G} )
		\begin{bmatrix}
			2\gamma V_j^Y (\Sigma_j^Y)^{\T}(U_j^Y)^{\T}   
			& V_{j}^Y  (\Upsilon_j^H)^{1/2}
			\\
			(\Upsilon_j^G)^{1/2}
			(U_j^Y)^{\T} & 2\gamma \Sigma_j^Y
		\end{bmatrix} 
		\\
		\cdot
		( 
		K_{j-1}^G \oplus I_{r_{j-1}^H}
		)
		\Phi_{1,j}^H. 
		\end{multlined}\label{eq:Kj_G} 
	\end{align}
\end{lemma}
\begin{proof}
The tedious proof can be found in Appendix~\ref{sec:proof-of-lemma-lm}.
\end{proof}

Note that $L_j^G$ and $L_j^H$ are required when we truncate $G_{j+1}^{(j)}$ 
and $H_{j+1}^{(j)}$ respectively to $\widetilde{G}_{j+1}\equiv G_{j+1}^{(j+1)}$ and 
$\widetilde{H}_{j+1}\equiv H_{j+1}^{(j+1)}$, 
and all integrants for $K_{j}^G$ (also $L_{j}^G$ and $L_j^H$) are known from the 
previous step when computing $\widetilde{G}_j= G_j^{(j)}$ and $\widetilde{H}_j=H_j^{(j)}$.    

\begin{remark}\label{rk:compute_svd_G}
	Based on \eqref{eq:Lj_GH}  for $L_j^G$ and $L_j^H$, we can write the 
	SVDs in 
	\eqref{eq:svdH_j+1}  as below:  
		\begin{align}
			&\begin{bmatrix}
				I_{r_j^G}&R_{12}^{j,U}\\0&R_2^{j,U}
			\end{bmatrix}
			\begin{bmatrix}
				I_{r_j^G} & -L_j^G\\ 0 & I_{r_j^G}
			\end{bmatrix}
			\left\{ 
				\Sigma_{1,j}^G \oplus 
				\left[\Sigma_{1,j}^G U_{j+1}^Y 
				(\Upsilon_{j+1}^G)^{-1/2}\right]
			\right\} \nonumber 
			\\
			=&
			\begin{bmatrix}
				I_{r_j^G}&R_{12}^{j,U}\\0&R_2^{j,U}
			\end{bmatrix}
			\begin{bmatrix}
				\Sigma_{1,j}^G & 
				-2\gamma \Sigma_{1,j}^G K_j^G V_{j+1}^Y (\Sigma_{j+1}^Y)^{\T}  
				(\Upsilon_{j+1}^G)^{-1/2}
				\\
				0&  \Sigma_{1,j}^G U_{j+1}^Y 
				(\Upsilon_{j+1}^G)^{-1/2}
			\end{bmatrix} \nonumber
			\\
			=&
			\Theta_{j+1}^G \Sigma_{j+1}^G (\Phi_{j+1}^G)^{\T}, \label{eq:svdG_j+1new}
			\\
			&\begin{bmatrix}
				I_{r_j^H}&R_{12}^{j,V}\\0&R_2^{j,V}
			\end{bmatrix}
			\begin{bmatrix}
				I_{r_j^H} & -L_j^H\\ 0 & I_{r_j^H}
			\end{bmatrix}
			\left\{\Sigma_{1,j}^H \oplus 
				\left[\Sigma_{1,j}^H V_{j+1}^Y 
				(\Upsilon_{j+1}^H)^{-1/2}\right]
			\right\} \nonumber 
			\\
			=&
			\begin{bmatrix}
				I_{r_j^H}&R_{12}^{j,V}\\0&R_2^{j,V}
			\end{bmatrix}
			\begin{bmatrix}
				\Sigma_{1,j}^H & -2\gamma \Sigma_{1,j}^H K_j^H U_{j+1}^Y \Sigma_{j+1}^Y 
				(\Upsilon_{j+1}^H)^{-1/2}
				\\
				0 & \Sigma_{1,j}^H V_{j+1}^Y 
				(\Upsilon_{j+1}^H)^{-1/2}
			\end{bmatrix}\nonumber
			\\
			=&
			\Theta_{j+1}^H \Sigma_{j+1}^H (\Phi_{j+1}^H)^{\T}.  \label{eq:svdH_j+1new} 
		\end{align}
\end{remark}

To clarify how  we   skip the doubling step and compute 
$\widetilde{G}_{j+1} \equiv G_{j+1}^{(j+1)}$ directly from 
$\widetilde{G}_j \equiv G_j^{(j)}$ by $K_j^G$ and $K_j^H$ 
(or analogously $L_j^G$ and $L_j^H$), we illustrate with the calculation 
of $\widetilde{G}_{3} \equiv G_{3}^{(3)}$ and $\widetilde{H}_{3} \equiv H_{3}^{(3)}$. 
For this we have $U_2^Y, \Sigma_2^Y, V_2^{Y}$ from   
$\Sigma_{1,1}^G \Omega_1 \Sigma_{1,1}^H = U_2^Y \Sigma_2^Y (V_2^Y)^{\T}$, 
and  $K_1^G, K_1^H =(K_1^G)^{\T}$ when computing 
$\widetilde{G}_{2} \equiv G_{2}^{(2)} = 2\gamma \mathcal{Q}_2^U (\Sigma_{1,2}^G)^2 	(\mathcal{Q}_2^U)^{\T}$ 
and $\widetilde{H}_{2} \equiv H_{2}^{(2)}=2\gamma  \mathcal{Q}_2^V (\Sigma_{1,2}^H)^2 (\mathcal{Q}_2^V)^{\T}$.   

By   the modified Gram-Schmidt process, we produce
\begin{small}
\begin{align*}
	\left[ \mathcal{Q}_2^U, \widetilde A_{\gamma}^{2^2} \mathcal{Q}_2^U \right]
	= 		\left[ \mathcal{Q}_2^U, Q_{3}^U \right]
	\begin{bmatrix}
		I_{r_2^G} & R_{12}^{2,U} \\ 0 & R_2^{2,U}
	\end{bmatrix}, \ \ \ 
	\left[ \mathcal{Q}_2^V, (\widetilde A_{\gamma}^{\T})^{2^2} \mathcal{Q}_2^V \right] 
	=	\left[ \mathcal{Q}_2^V, Q_{3}^V \right]
	\begin{bmatrix}
		I_{r_2^H} & R_{12}^{2,V} \\ 0 & R_2^{2,V}
	\end{bmatrix}, 
\end{align*}
\end{small}
and compute the SVD of   
$\Sigma_{1,2}^G \Omega_2 \Sigma_{1,2}^H = U_3^Y \Sigma_3^Y (V_3^Y)^{\T}$ with 
$\Omega_2=(\mathcal{Q}_2^U)^{\T} \mathcal{Q}_2^V$.  
By \eqref{eq:Kj_G}, we construct
\begin{align*}
	K_2^G 
	&= 
	(\Phi_{1,2}^G)^{\T} 
	(K_{1}^G \oplus I_{r_{1}^G} )
	\begin{bmatrix}
		2\gamma V_2^Y (\Sigma_2^Y)^{\T}(U_2^Y)^{\T}& V_{2}^Y  (\Upsilon_2^H)^{1/2}
		\\
		(\Upsilon_2^G)^{1/2}
		(U_2^Y)^{\T} & 2\gamma \Sigma_2^Y
	\end{bmatrix} 
	( 
		K_{1}^G \oplus I_{r_{1}^H}
	)
	\Phi_{1,2}^H
\end{align*}
with $\Upsilon_2^G:= I_{r_1^G}+4\gamma^2 \Sigma_2^Y(\Sigma_2^Y)^{\T}$ and  
$\Upsilon_2^H:= I_{r_1^H}+4\gamma^2 (\Sigma_2^Y)^{\T}\Sigma_2^Y$.  
Then by computing the SVDs as in  \eqref{eq:svdG_j+1new} and \eqref{eq:svdH_j+1new}, 
with $j=2$, we  obtain $\widetilde{G}_3$ and $\widetilde{H}_3$ by truncation: 
\begin{align*}
	\widetilde{G}_3
	&=
	\left[ \mathcal{Q}_2^U, Q_{3}^U \right]
	\Theta_{1,3}^G (\Sigma_{1,3}^G)^{\T} (\Theta_{1,3}^G)^{\T} 
	\left[ \mathcal{Q}_2^U, Q_{3}^U \right]^{\T},
	\\
	\widetilde{H}_3
	&=
	\left[ \mathcal{Q}_2^V, Q_{3}^V \right]
	\Theta_{1,3}^H (\Sigma_{1,3}^H)^{\T} (\Theta_{1,3}^H)^{\T} 
	\left[ \mathcal{Q}_2^V, Q_{3}^V \right]^{\T},
\end{align*}
where  $\Sigma_{3}^G =\Sigma_{1,3}^G \oplus \Sigma_{2,3}^G$,  
$\Sigma_{3}^H =\Sigma_{1,3}^H \oplus \Sigma_{2,3}^H$ with 
$\|\Sigma_{2,3}^G\|\leq \varepsilon_{3} \|\Sigma_{1,3}^G\|$, 
$\|\Sigma_{2,3}^H\|\leq \varepsilon_{3} \|\Sigma_{1,3}^H\|$, and 
$\Theta_{3}^G=[\Theta_{1,3}^G, \Theta_{2,3}^G ]$,  
$\Theta_{3}^H=[\Theta_{1,3}^H, \Theta_{2,3}^H ]$,   
$\Phi_{3}^G=[\Phi_{1,3}^G, \Phi_{2,3}^G ]$,   
$\Phi_{3}^H=[\Phi_{1,3}^H, \Phi_{2,3}^H ]$.     

Clearly, to get $\widetilde{G}_3\equiv G_3^{(3)}$ and $\widetilde{H}_3\equiv H_3^{(3)}$,   
we require $\mathcal{Q}_2^U$, $\mathcal{Q}_2^V$, $\Sigma_{1,2}^G$, $\Sigma_{1,2}^H$, 
$K_1^G$, $K_1^H=(K_1^G)^{\T}$, $\Phi_{1,2}^G$, $\Phi_{1,2}^H, U_2^Y$, 
$\Sigma_2^Y, V_2^Y$ from the previous step, followed by the truncation. 
A similar procedure can be carried out for the general case, for  
$\widetilde{G}_j\equiv G_j^{(j)}$ and $\widetilde{H}_j\equiv H_j^{(j)}$  ($j\ge 3$). 

\subsection{Algorithm dSDA\texorpdfstring{$\rm{_t}$}{t}}\label{ssec:algorithm-dsda--t-}

In this section, we list the computational steps  
for the dSDA$\rm{_t}$. 

\begin{enumerate}
	\item Initial ($j=1$): given $A, \gamma, B, C$,  \texttt{compute} $U_0$, $U_1$, $V_0$, $V_1$, $Y_0$ and $T_0$. 
	\item \texttt{Compute} $\widetilde{G}_1$ and $\widetilde{H}_1$ with  
		$U_0$, $U_1$, $V_0$, $V_1$, $Y_0$ and $T_0$. 
		\begin{enumerate}
			\item \texttt{Compute} the QR factorizations with  column pivoting  of 
				$[U_0, U_1]$ and $[	V_0, V_1]$:
				\begin{align*}
						[U_0, U_1]
						= Q_1^U R_1^U P_1^U, 
						\qquad 
						[V_0, V_1
						] = Q_1^V R_1^V P_1^V. 
				\end{align*}   
			\item \texttt{Compute} the SVD of $Y_1 = \begin{bmatrix}
					0 & Y_0 \\ Y_0 & 2\gamma T_0
				\end{bmatrix}$:
				\begin{align*}
					Y_1=U_1^Y\Sigma_1^Y (V_1^Y)^{\T},
					\ \ \ 
					U_1^Y\in \mathbb{R}^{2m \times 2m},
					\ \ 
					\Sigma_1^Y \in \mathbb{R}^{2m \times 2l},
					\ \ 
					V_1^Y\in \mathbb{R}^{2l \times 2l}.
				\end{align*}
			\item \texttt{Compute} the SVDs of $R_1^U P_1^U U_1^Y
				(\Upsilon_1^G)^{-1/2}$,  $R_1^V P_1^V  V_1^Y 
				(\Upsilon_1^H)^{-1/2}$ 
				by~\eqref{eq:svd1} .
			\item \texttt{Compute} the truncated $\widetilde G_1$ and $\widetilde H_1$ 
				by \eqref{eq:widetildeG1-and-widetilde-H1}. 
			\item \texttt{Save} $\widetilde{G}_1$, $\widetilde{H}_1$,  
				$Q_1^U\Theta_{1,1}^G$, $Q_1^V\Theta_{1,1}^H$, $\Sigma_{1,1}^G$, $\Sigma_{1,1}^H$, 
				$\Phi_{1,1}^G$, $\Phi_{1,1}^H$ and $\Sigma_1^Y$.
		\end{enumerate}
	\item \texttt{Compute} $\widetilde{G}_2$ and $\widetilde{H}_2$ with inputs $A, \gamma$,  
		$Q_1^U\Theta_{1,1}^G, Q_1^V\Theta_{1,1}^H$, $\Sigma_{1,1}^G, \Sigma_{1,1}^H$, 
		$\Phi_{1,1}^G, \Phi_{1,1}^H$, $\Sigma_1^Y$.  
		\begin{enumerate}
			\item \texttt{Compute} the QR factorizations of 
				\[
						\left[
						Q_1^U \Theta_{1,1}^G, \, \widetilde A_{\gamma}^2 Q_1^U \Theta_{1,1}^G\right] 
						\quad \text{and} \quad   
						\left[Q_1^V \Theta_{1,1}^H, \, (\widetilde A_{\gamma}^{\T})^2Q_1^V \Theta_{1,1}^H\right]
				\]
				by the modified Gram-Schmidt process, as in \eqref{eq:qr2}.
			\item \texttt{Compute} the SVD of $\Sigma_{1,1}^G \Omega_1 \Sigma_{1,1}^H$ with 
				$\Omega_1=(Q_1^U \Theta_{1,1}^G)^{\T} Q_1^V \Theta_{1,1}^H$: 
				\begin{small}
					\begin{align*}
					\Sigma_{1,1}^G \Omega_1 \Sigma_{1,1}^H = U_2^Y \Sigma_2^Y (V_2^Y)^{\T}, 
					\ 
					U_2^Y\in \mathbb{R}^{r_1^G \times r_1^G},  
					\Sigma_2^Y\in \mathbb{R}^{r_1^G \times r_1^H},   
					V_2^Y\in \mathbb{R}^{r_1^H \times r_1^H}.
				\end{align*}				
			\end{small}
			\item \texttt{Construct} $K_1^G = (\Phi_{1,1}^G)^{\T} \Sigma_1^Y \Phi_{1,1}^H$.
			\item \texttt{Compute} $L_1^G=2\gamma \Sigma_{1,1}^G K_1^G \Sigma_{1,1}^H \Omega_1^{\T}$ 
				and $L_1^H=2\gamma \Sigma_{1,1}^H (K_1^G)^{\T} \Sigma_{1,1}^G \Omega_1$.  
			\item \texttt{Compute} by \eqref{eq:svd2_H}   the SVDs of 
				\begin{small}
					\begin{align*}
						\begin{bmatrix}
							I_{r_1^G}&R_{12}^U\\0&R_2^U
						\end{bmatrix}
						\begin{bmatrix}
							I_{r_1^G}&-L_1^{G}		\\ 0&I_{r_1^G}   
						\end{bmatrix}
						\left\{ 
							\Sigma_{1,1}^G \oplus  
							\left[\Sigma_{1,1}^G U_2^Y
							\left(I_{r_1^G}+4\gamma^2 \Sigma_2^Y(\Sigma_2^Y)^{\T}\right)^{-1/2} \right]
						\right\}, 
					\end{align*}
				\end{small}
			\begin{small}
				\begin{align*}
					\begin{bmatrix}
						I_{r_1^H}&R_{12}^V\\0&R_2^V
					\end{bmatrix}
					\begin{bmatrix}
						I_{r_1^H}&-L_1^{H}		\\ 0&I_{r_1^H}   
					\end{bmatrix}
					\left\{ 
						\Sigma_{1,1}^H \oplus 
						\left[\Sigma_{1,1}^H V_2^Y 
						\left( I_{r_1^H}+ 4\gamma^2 (\Sigma_2^Y)^{\T} \Sigma_2^Y\right)^{-1/2}\right]
					\right\}.
				\end{align*}
			\end{small}
			\item \texttt{Compute} the truncated  $\widetilde  G_2$ and $\widetilde  H_2$  by \eqref{eq:wtdGH2}. 
			\item \texttt{Save} $\widetilde{G}_2$, $\widetilde{H}_2$,    
				$\mathcal{Q}_2^U$, $\mathcal{Q}_2^V$, $\Sigma_{1,2}^G$, $\Sigma_{1,2}^H$, 
				$K_1^G$,  $\Phi_{1,2}^G$, $\Phi_{1,2}^H, U_2^Y$ and  $\Sigma_2^Y, V_2^Y$; \texttt{set} $j=2$.  
		\end{enumerate}

	\item \texttt{Compute} the truncated  $\widetilde{G}_{j+1}$ and $\widetilde{H}_{j+1}$, with  
		 $A$, $\gamma$,  $\mathcal{Q}_j^U$, $\mathcal{Q}_j^V$, 
		$\Sigma_{1,j}^G$, $\Sigma_{1,j}^H$, $K_{j-1}^G$, 
		$\Phi_{1,j}^G$, $\Phi_{1,j}^H$, $U_j^Y$, $\Sigma_j^Y$ and $V_j^Y$.  
		\begin{enumerate}
			\item By the modified Gram-Schmidt process, \texttt{compute} the QR factorizations of 
				$\left[\mathcal{Q}_{j}^U, \, \widetilde A_{\gamma}^{2^{j}} \mathcal{Q}_{j}^U\right]$
				and $\left[\mathcal{Q}_{j}^V, \, (\widetilde A_{\gamma}^{\T})^{2^{j}} \mathcal{Q}_{j}^V\right]$  by \eqref{eq:qrj_U}.
			\item \texttt{Compute} the SVD of $\Sigma_{1,j}^G \Omega_{j} \Sigma_{1,j}^H$ with 
				$\Omega_j=(\mathcal{Q}_j^U)^{\T} \mathcal{Q}_j^V$:
				\begin{align*}
					&\Sigma_{1,j}^G \Omega_{j} \Sigma_{1,j}^H=
					U_{j+1}^Y \Sigma_{j+1}^Y (V_{j+1}^Y)^{\T},  
					\\
					&U_{j+1}^Y\in \mathbb{R}^{r_{j}^G \times r_{j}^G},
					\quad \Sigma_{j+1}^Y \in \mathbb{R}^{r_{j}^G\times r_{j}^H}, 
					\quad V_{j+1}^Y\in \mathbb{R}^{r_{j}^H \times r_{j}^H}.
				\end{align*}  
			\item With  $\Phi_{1,j}^G, \Phi_{1,j}^H, K_{j-1}^G, U_j^Y, \Sigma_j^Y$ and $V_j^Y$, 
				\texttt{construct} $K_j^G$ by~\eqref{eq:Kj_G}.  
			\item \texttt{Compute} $L_j^G= 2\gamma  \Sigma_{1,j}^G  K_j^G \Sigma_{1,j}^H \Omega_j^{\T}$ 
				and $L_j^H = 2\gamma \Sigma_{1,j}^H  K_j^H \Sigma_{1,j}^G \Omega_j$. 
			\item \texttt{Compute} by \eqref{eq:svdH_j+1}  the SVDs of
				\begin{small}
					\begin{align*}
						&\begin{bmatrix}
							I_{r_j^G}&R_{12}^{j,U}\\0&R_2^{j,U}
						\end{bmatrix}
						\begin{bmatrix}
							I_{r_j^G} & -L_j^G\\ 0 & I_{r_j^G}
						\end{bmatrix}
						\{ 
							\Sigma_{1,j}^G \oplus 
							[\Sigma_{1,j}^G U_{j+1}^Y 
							(I_{r_j^G} + 4\gamma^2 \Sigma_{j+1}^Y(\Sigma_{j+1}^Y)^{\T})^{-1/2}
							]
						\}, 
						\\
						&\begin{bmatrix}
							I_{r_j^H}&R_{12}^{j,V}\\0&R_2^{j,V}
						\end{bmatrix}
						\begin{bmatrix}
							I_{r_j^H} & -L_j^H\\ 0 & I_{r_j^H}
						\end{bmatrix}
						\{ 
							\Sigma_{1,j}^H \oplus  
							[\Sigma_{1,j}^H V_{j+1}^Y 
							(I_{r_j^H} + 4\gamma^2 (\Sigma_{j+1}^Y)^{\T} \Sigma_{j+1}^Y)^{-1/2}
							]
						\}. 
					\end{align*}
				\end{small}
			\item \texttt{Compute} the truncated $\widetilde  G_{j+1}$ and $\widetilde  H_{j+1}$ by \eqref{eq:wtdH_j+1}. 
			\item \texttt{Save} $\widetilde{G}_{j+1}$, $\widetilde{H}_{j+1}$,   
				$\mathcal{Q}_{j+1}^U$, $\mathcal{Q}_{j+1}^V$, $\Sigma_{1,j+1}^G$, 
				$\Sigma_{1,j+1}^H$, 	$K_j^G$, 
				$\Phi_{1,j+1}^G$, $\Phi_{1,j+1}^H$,$ U_{j+1}^Y$, $\Sigma_{j+1}^Y$ and $V_{j+1}^Y$.   
			\item \texttt{Set} $j:=j+1$; \texttt{repeat} Step $4$  until convergence.  
		\end{enumerate}
\end{enumerate}

From the above algorithm, the dominant flop counts occurs in the generation of the bases 
$\mathcal{Q}_j^U$ and $\mathcal{Q}_j^V$ for the associated Krylov subspaces. 
With truncation controlling their ranks and benefiting from the structures of $A$ like sparsity, 
the dominant flop counts will be those for the multiplication or the solution of 
linear systems associated with $A_{\gamma}$ or its transpose.

\section{Error Analysis for dSDA\texorpdfstring{$\rm{_t}$}{t}} \label{sec:forward-error-analysis-for-dsda}

The dSDA$\rm{_t}$ obviously produces  totally different matrix sequences 
$\{G_0, \widetilde{G}_1, \widetilde{G}_2, \widetilde{G}_3, \cdots\}$ and 
$\{H_0, \widetilde{H}_1, \widetilde{H}_2, \widetilde{H}_3, \cdots\}$  
from those by the  dSDA or  the SDA. 
Then the obvious question on the convergence of the dSDA$\rm{_t}$ has to be asked. 
Dose it hold that $\lim_{k \to \infty}\widetilde{G}_k = Y$ and   
$\lim_{k \to \infty}\widetilde{H}_k = X$, 
where $X$ is the solution to \eqref{eq:care} and $Y$ is the solution to the dual problem? 
To answer this  fully, we first show the relationship 
between the CARE~\eqref{eq:care} and some  DAREs. 
Then we    construct some perturbed DAREs  which the truncated  iterates  
$G_j^{(j)}\equiv \widetilde{G}_j$ and $H_j^{(j)}\equiv \widetilde{H}_j$ satisfy. 
We then   analyze   the errors of the symmetric positive semi-definite solutions 
for these perturbed  DAREs.  The detailed  analysis will eventually  
prove the convergence of  the dSDA$\rm{_t}$.  

\begin{lemma}\label{lm:care_to_dare}
	For the CARE problem \eqref{eq:care} and the iterates in \eqref{eq:care-iteration}, 
	it holds that 
	\begin{align*}
		A_k^{\T} X (I+G_k X)^{-1} A_k + H_k = X, \ \ \ 
		A_k Y (I+H_k Y)^{-1} A_k^{\T} + G_k = Y,
	\end{align*}
	where $Y$ is the unique symmetric positive semi-definite solution to the dual problem of \eqref{eq:care}.
\end{lemma}
\begin{proof}
	The results follow from the theory of the SDA 
	\cite{chuFL2005structurepreserving,linX2006convergence}, and the facts that 
	\begin{align*}
		\begin{bmatrix}
			A_k & 0 \\ -H_k & I
		\end{bmatrix}
		\begin{bmatrix}
			I \\ X
		\end{bmatrix} 
		=
		\begin{bmatrix}
			I& G_k \\ 0& A_k^{\T}
		\end{bmatrix}
		\begin{bmatrix}
			I \\ X
		\end{bmatrix} R_{\gamma}^{2^k},
		\qquad
		\begin{bmatrix}
			A_k  & 0 \\ -H_k & I
		\end{bmatrix}
		\begin{bmatrix}
			-Y \\ I
		\end{bmatrix} S_{\gamma}^{2^k}
		=
		\begin{bmatrix}
			I& G_k \\0& A_k^{\T}
		\end{bmatrix}
		\begin{bmatrix}
			-Y \\ I
		\end{bmatrix}, 
	\end{align*}
	where $R_{\gamma}:= (A-GX-\gamma I)^{-1}(A-GX+\gamma I)$ and 
	$S_{\gamma}:=(A^{\T} -HY -\gamma I)^{-1}(A^{\T} -HY + \gamma I)$.
\end{proof}

With $G_{j}^{(j)}$, $H_{j}^{(j)}$ and $A_{j}^{(j)}$ ($j\ge 1$) given 
explicitly in 
\eqref{eq:Hj+1} and \eqref{eq:Aj+1} respectively, 
the doubling iteration~\eqref{eq:care-iteration} produces, 
for $k > j$, $G_k^{(j)}$ and 
$H_k^{(j)}$ in~\eqref{eq:H_k_j} and  
\begin{equation}\label{eq:A-k-j}
	\begin{aligned}
		A_k^{(j)} 
		&= 
		\begin{multlined}[t]
			\widetilde A_{\gamma}^{2^k} - 2\gamma 
		\widehat{\mathcal{X}}_k^{U,(j)} 
		(I_{2^{k-j}}\otimes \widehat{M}_{j-1}^G )
		\left[I_{2^km} +  Y_k^{(j)} (Y_k^{(j)})^{\T} \right]^{-1}
		Y_k^{(j)} 
		\\
		\cdot
		(I_{2^{k-j}} \otimes \widehat{M}_{j-1}^H)^{\T} 
		(\widehat{\mathcal{X}}_k^{V, (j)})^{\T}
		\end{multlined}%\nonumber
		\\
		&\equiv 
		\begin{multlined}[t]
			\widetilde A_{\gamma}^{2^k} - 2\gamma 
			\mathcal{X}_k^{U,(j)} 
			\left[I_{2^{k-j}}\otimes \left((\Theta_{1,j}^G)^{\T} {M}_{j-1}^G\right) \right]
			\left[I_{2^km} +  Y_k^{(j)} (Y_k^{(j)})^{\T} \right]^{-1}
			\\
			\cdot
			Y_k^{(j)} 
			\left[I_{2^{k-j}} \otimes \left( ({M}_{j-1}^H)^{\T} \Theta_{1,j}^H\right) \right] 
			({\mathcal{X}}_k^{V, (j)})^{\T}.
		\end{multlined}	
		%\label{eq:A-k-j}
	\end{aligned}
\end{equation}
Now  consider respectively the DARE and its dual
\begin{equation}\label{eq:dare_j_dual_trun}
	\begin{aligned} 
		(A_j^{(j)})^{\T} X^{(j)} (I+G_j^{(j)}X^{(j)})^{-1} A_j^{(j)} + H_j^{(j)}&= X^{(j)}, 
		\\ 
		A_j^{(j)}  Y^{(j)} (I+H_j^{(j)} Y^{(j)})^{-1} (A_j^{(j)})^{\T} + G_j^{(j)}&=Y^{(j)}. 
	\end{aligned}
\end{equation}
Assuming that the unique symmetric positive semi-definite  
solutions $X^{(j)}$ and $Y^{(j)}$ exist, 
then the matrix sequences $\{A_k^{(j)}\}$, $\{G_k^{(j)}\}$, and $\{H_k^{(j)}\}$ 
satisfy \cite{chuFL2005structurepreserving,linX2006convergence}
\begin{enumerate}
	\item [(a)] $A_k^{(j)} = (I + G_k^{(j)} X^{(j)})
		\left[(I+G_j^{(j)}X^{(j)})^{-1} A_j^{(j)}\right]^{2^{k-j}}$;	
	\item [(b)] $\{H_k^{(j)}\}$ is monotonically increasing with upper bound $X^{(j)}$ and 
		\begin{small}
			\begin{align*}
				&
				X^{(j)} - H_k^{(j)} 
				\\
				=&
				\begin{multlined}[t]
					\left[ (A_j^{(j)})^{\T} (I+ X^{(j)}G_j^{(j)})^{-1} \right]^{2^{k-j}}
					X^{(j)} ( I + G_k^{(j)}  X^{(j)}) 
					\left[(I+G_j^{(j)} X^{(j)})^{-1} A_j^{(j)} \right]^{2^{k-j}} 
				\end{multlined}
				\\
				\leq & 
				\begin{multlined}[t]
					\left[(A_j^{(j)})^{\T} (I+ X^{(j)}G_j^{(j)})^{-1}\right]^{2^{k-j}}
					X^{(j)} ( I +  Y^{(j)}  X^{(j)}) 
					\left[(I+G_j^{(j)} X^{(j)})^{-1} A_j^{(j)} \right]^{2^{k-j}}; 
				\end{multlined}		
			\end{align*}
		\end{small}
	\item [(c)] $\{G_k^{(j)}\}$ is monotonically increasing with upper bound $Y^{(j)}$ and 
		\begin{small}
			\begin{align*}
				&
				Y^{(j)} - G_k^{(j)} 
				\\
				=&
				\begin{multlined}[t]
					\left[ A_j^{(j)}(I+ Y^{(j)}H_j^{(j)})^{-1}\right]^{2^{k-j}}
					Y^{(j)} (I + H_k^{(j)} Y^{(j)})
					\left[(I+H_j^{(j)} Y^{(j)})^{-1}(A_j^{(j)})^{\T} \right]^{2^{k-j}}
				\end{multlined}			
				\\
				\leq & 
				\begin{multlined}[t]
					\left[A_j^{(j)}(I+Y^{(j)}H_j^{(j)})^{-1}\right]^{2^{k-j}}
					Y^{(j)} (I + X^{(j)} Y^{(j)}) 
					\left[(I+H_j^{(j)}Y^{(j)})^{-1}(A_j^{(j)})^{\T} \right]^{2^{k-j}}.
				\end{multlined}	
			\end{align*}
		\end{small}
\end{enumerate}
We thus  deduced  that $A_k^{(j)} \to  0$,   
$G_k^{(j)} \to X^{(j)}$  and $H_k^{(j)} \to Y^{(j)}$  as $k\to \infty$. 

Note that by Lemma~\ref{lm:care_to_dare} and the doubling transformation for $j\ge 0$,  
we have  
\begin{equation}\label{eq:dare_j_dual}
	\begin{aligned}
		(A_{j+1}^{(j)})^{\T}  X^{(j)} 	(I+G_{j+1}^{(j)}  X^{(j)})^{-1}
		A_{j+1}^{(j)} + H_{j+1}^{(j)} &=  X^{(j)}, 
		\\
		A_{j+1}^{(j)}  Y^{(j)} ( I+H_{j+1}^{(j)}Y^{(j)} )^{-1} 
		(A_{j+1}^{(j)})^{\T} + G_{j+1}^{(j)} &=  Y^{(j)}, %\label{eq:dare_j_dual}
	\end{aligned}
\end{equation}
where  $A_1^{(0)}:= A_1$, $G_1^{(0)}:= G_1$, $H_1^{(0)}:=H_1$,  
$X^{(0)}:= X$  and $Y^{(0)}:= Y$. Now take $j=0, 1, 2, \cdots$ for 
\eqref{eq:dare_j_dual}, and  at the same time set $j=1, 2, 3, \cdots$ 
for \eqref{eq:dare_j_dual_trun}. 
Obviously, the coefficients in  the DAREs in
\eqref{eq:dare_j_dual_trun}  
are respectively  the truncated results from those in the DAREs in 
\eqref{eq:dare_j_dual}. 
This implies that we can work out the difference between $X^{(j)}$ and $X^{(j+1)}$ 
(also $Y^{(j)}$ and  $Y^{(j+1)}$) by perturbation theory in Lemma~\ref{lm:forward_error_dare}. 
We first  need to estimate  the errors in the coefficient matrices; i.e.,  
the differences between $A_{j+1}^{(j)}$ and $A_{j+1}^{(j+1)}$, 
$G_{j+1}^{(j)}$ and $G_{j+1}^{(j+1)}$, $H_{j+1}^{(j+1)}$ and $H_{j+1}^{(j)}$ 
for $j\ge 0$. When these  differences  are sufficiently small,   
we  can then apply   Lemma~\ref{lm:forward_error_dare}   to the DAREs in 
\eqref{eq:dare_j_dual}, subsequently verify   
the existence of the symmetric positive semi-definite solutions 
$X^{(j+1)}$ and $Y^{(j+1)}$ in~\eqref{eq:dare_j_dual_trun}. 
The analysis also yields the errors  
$\|X^{(j)} - X^{(j+1)}\|$ and  $\|Y^{(j)} - Y^{(j+1)}\|$. 

Assume that we have obtained the  differences 
$A_{j+1}^{(j)} - A_{j+1}^{(j+1)}$, $G_{j+1}^{(j)}-G_{j+1}^{(j+1)}$ and 
$H_{j+1}^{(j)}-H_{j+1}^{(j+1)}$. Then by Lemma~\ref{lm:forward_error_dare},  
Remark~\ref{rk:forward-error-dare} and item~(b) above,  we  conclude that 
\begin{align}
	&
	\|X-H_k^{(j)}\| 
	= 
	\left\|X - X^{(1)} + \sum_{s=1}^{j-1} (X^{(s)} - X^{(s+1)}) + X^{(j)} - H_k^{(j)} \right\| 
	\nonumber
	\\
	\leq& 
	\|X-X^{(1)}\| + \sum_{s=1}^{j-1} \|X^{(s)} - X^{(s+1)}\| + \|X^{(j)}-H_k^{(j)}\|
	\nonumber
	\\
	\leq&
	\begin{multlined}[t]
		\frac{1}{\ell^{(0)}}\|H_1^{(1)}-H_1\| + \xi^{(0)}\|A_1^{(1)}-A_1\| 
		+ \eta^{(0)} \|G_1^{(1)}-G_1\| 
		\\
		+ \bigO(\|(H_1^{(1)}-H_1, A_1^{(1)}-A_1,  G_1^{(1)}-G_1)\|^2)
		\\
		+ \sum_{s=1}^{j-1}\Big\{\frac{1}{\ell^{(s)}}\|H_{s+1}^{(s)}-H_{s+1}^{(s+1)}\| + 
			\xi^{(s)}\|A_{s+1}^{(s)}-A_{s+1}^{(s+1)}\| + \eta^{(s)} \|G_{s+1}^{(s)}-G_{s+1}^{(s+1)}\| 
			\\
			+ \bigO(\|(H_{s+1}^{(s)}-H_{s+1}^{(s+1)}, A_{s+1}^{(s)}-A_{s+1}^{(s+1)},  
			G_{s+1}^{(s)}-G_{s+1}^{(s+1)})\|^2)
		\Big\}
		\\
		+\N*{[ (A_{j}^{(j)})^{\T} (I+ X^{(j)}G_{j}^{(j)})^{-1}]^{2^{k-j}}
			X^{(j)} ( I + Y^{(j)}  X^{(j)}) 
		[(I+G_{j}^{(j)} X^{(j)})^{-1} A_{j}^{(j)} ]^{2^{k-j}}}, 
	\end{multlined}	\label{eq:solution_X_diff}
\end{align}
where $\ell^{(s)}$, $\xi^{(s)}$ and $\eta^{(s)}$ (for $s\geq 0$) 
 are defined similarly as $\ell$, $\xi$ and $\eta$  respectively 
in \eqref{eq:lxieta}, but with $A_c$, $A_0$, $G_0$, and $H_0$ being    
replaced by $(I+G_{s+1}^{(s)} X^{(s)})^{-1} A_{s+1}^{(s)}$, 
$A_{s+1}^{(s)}$, $G_{s+1}^{(s)}$ and $H_{s+1}^{(s)}$, respectively. 

The truncation errors satisfy  
$\|G_{s+1}^{(s+1)}-G_{s+1}^{(s)}\|\le \varepsilon_{s+1} \|G_{s+1}^{(s)}\|$ and 
$\|H_{s+1}^{(s+1)}-H_{s+1}^{(s)}\|\le \varepsilon_{s+1} \|H_{s+1}^{(s)}\|$, 
where $\varepsilon_{s+1}$ is some  small tolerance.  Hence,  
for the difference $\N{X-H_k^{(j)}}$, it follows from \eqref{eq:solution_X_diff} that   
we just need to estimate  $\N{A_1^{(1)}-A_1}$ and  $\N{A_{s+1}^{(s+1)} - A_{s+1}^{(s)}}$, 
as   in the following lemma.  

\begin{lemma}\label{lm:doubling_k1}
	With $\kappa_{s} := \max\{1, \N{K_s^G}^2\} 
	\big(2\gamma\N{\Sigma_{s+1}^Y} 
	+\sqrt{1+4\gamma^2\N{\Sigma_{s+1}^Y}^2 }\big)$ for $s\geq 1$, 
	we have 
	\begin{enumerate}
		\item [(\romannumeral1)] $\N{A_1^{(1)}-A_1} \leq 
			4\gamma \varepsilon_1 \N{\Sigma_{1,1}^G} \N{\Sigma_1^Y} \N{\Sigma_{1,1}^H}$; and  
		\item [(\romannumeral2)] $
			\N{A_{s+1}^{(s+1)} - A_{s+1}^{(s)}} 
			\leq 
			4\gamma  \kappa_s \varepsilon_{s+1}\N{\Sigma_{1,s+1}^G} \N{\Sigma_{1,s+1}^H}$. 
	\end{enumerate}
\end{lemma}
\begin{proof}
	The proof, especially for (\romannumeral2), is tedious and 	
	can be found in  Appendix~\ref{sec:proof-of-lemma-ref-lm-doubling_k1}. 	
\end{proof}

Although $\{H_k^{(j)}\}_{k=j}^{\infty}$ may not converge to $X$ for $j\ge 1$, 
however, by  \eqref{eq:solution_X_diff} and Lemma~\ref{lm:doubling_k1} 
we know that the error $H_k^{(j)}-X$ equals  the sum   of a finite number of   
truncated errors, which is bounded by the truncated errors. 
Hence we have the following convergence result. 
\begin{theorem}\label{thm:}
Provided that  the truncated errors are  small enough, 
$\{H_k^{(j)}\}_{k=j}^{\infty}$ and  $\{G_k^{(j)}\}_{k=j}^{\infty}$
converges quadratically  to $X$ and $Y$ respectively.  
\end{theorem}

\section{Numerical Examples}\label{sec:numerical-examples}

In this section,  we illustrate the performance of the  dSDA$\rm{_t}$  
by applying it to  three steel profile cooling models, 
all of which are  from  the benchmarks collected at morWiki~\cite{community},  
and several randomly generated examples. For comparison, we also apply the rational 
Krylov subspace projection (RKSM)~\cite{simoncini2016analysis}, 
the RADI~\cite{bennerBKS2018radi} and the low-rank Newton-Kleinman ADI (NKADI)~\cite{saakKB2016mmess}  
methods
\footnote{The codes for RKSM and NKADI are available  
respectively from the homepage of Prof.~V.~Simoncini and the M-M.E.S.S.\ package.}.  
Note that the rational Krylov subspace in RKSM is 
$$
\Span\Big\{(A-\alpha_1 I)^{-\T}C^{\T}, \cdots, \prod_{i=1}^j (A -\alpha_i I)^{-\T} C^{\T}\Big\}.
$$
With $\alpha_1 = \cdots = \alpha_i=\gamma$, 
it is the subspace where the dSDA seeks the solution.
In \eqref{eq:care-iteration}, we illustrate that choosing those different 
shift parameters $\alpha_i$ seems unnecessary, 
although an appropriate selection may improve convergence.  
All algorithms are implemented in MATLAB 2017a on
a 64-bit PC with an Intel Core i7 processor at 3.20 GHz and 64G RAM.

\begin{example}\label{eg:example-1}
	The dimensions of the three models respectively are  $1357, 5177$ and  $20209$. In all test examples, 
	$A$ is symmetric and negative definite (thus stable)  
	and $B\in \mathbb{R}^{n\times 7}$ and $C\in \mathbb{R}^{6\times n}$.  
	For  all displayed numerical results corresponding to the dSDA$\rm_t$,   
	we set the tolerance for the normalized  residual, which is used for the stop criteria, 
	as $10^{-13}$ and the maximal number of iterations to $20$. 

	With    $\gamma=10^{-6}$   and   setting  the truncation 
	tolerance in each step  as $10^{-15}$,    we apply our dSDA$\rm_t$  to all  three test examples. 
	Figures~\ref{fig1357}--\ref{fig20209} trace the 
	normalized    residuals  of the CAREs   
	and the corresponding  dual equations: 
	\begin{align*}
		\rho_X&:=
		\frac{	\|A^{\T} \widetilde{H}_j + \widetilde{H}_j A - \widetilde{H}_j B B^{\T} \widetilde{H}_j + C^{\T} C \|_F}
		{2\|A^{\T} \widetilde{H}_j\|_F + \|\widetilde{H}_j B B^{\T} \widetilde{H}_j\|_F + \|C^{\T} C\|_F }, 
		\\
		\rho_Y&:=
		\frac{	\|A \widetilde{G}_j + \widetilde{G}_j A^{\T} - \widetilde{G}_j C^{\T} C \widetilde{G}_j + B B^{\T} \|_F}
		{2\|A \widetilde{G}_j\|_F + \|\widetilde{G}_j C^{\T} C \widetilde{G}_j\|_F + \|B B^{\T}\|_F },   
	\end{align*}
and the  numerical ranks of  $\widetilde{H}_j\equiv H_j^{(j)}$  
	and $\widetilde{G}_j\equiv G_j^{(j)}$ through the iteration. 

	\begin{figure*}[ht]
		\centering
		\subfloat[residuals]{
			\includegraphics[width=2.4in]{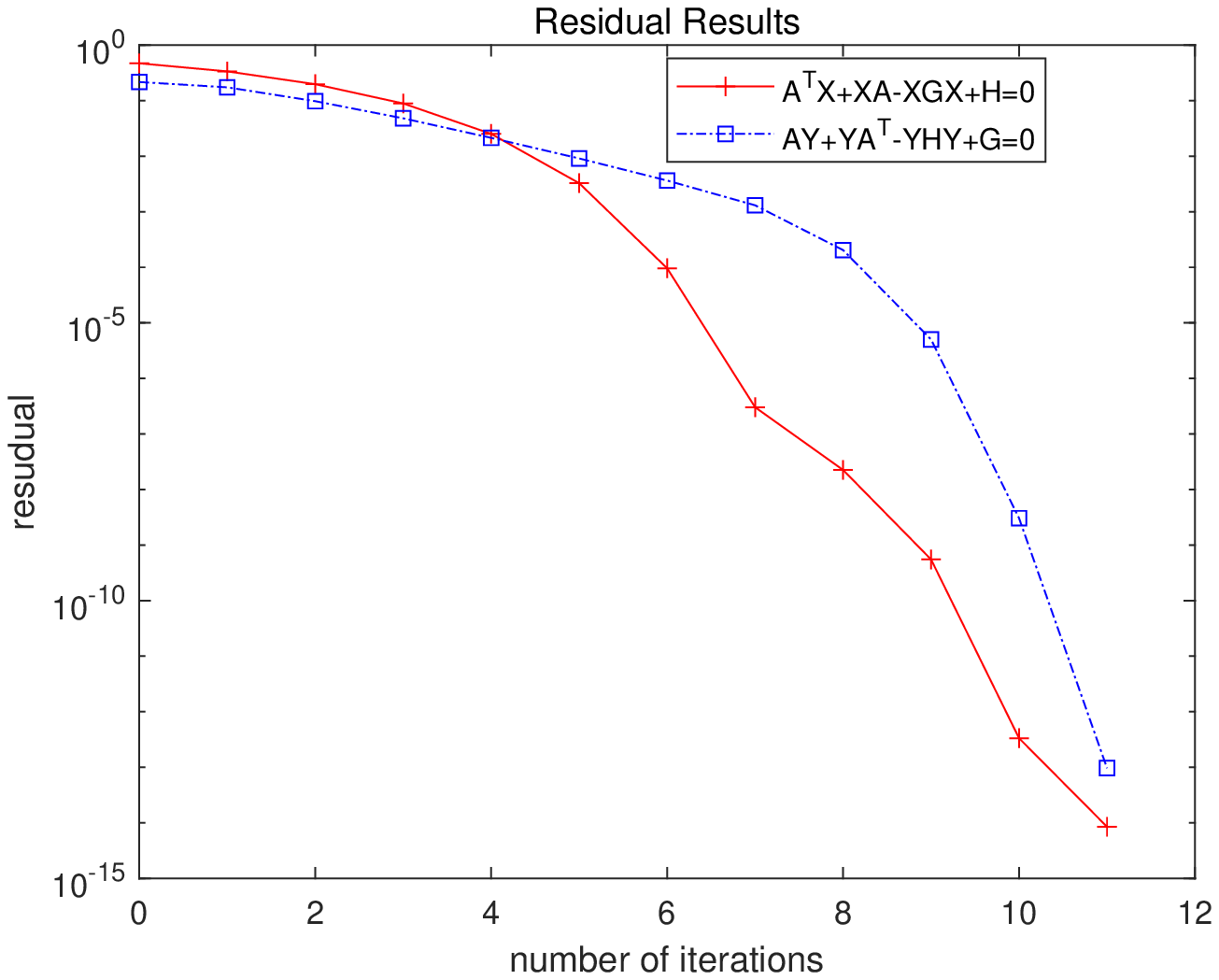}
		}
		\subfloat[ranks]{
			\includegraphics[width=2.4in]{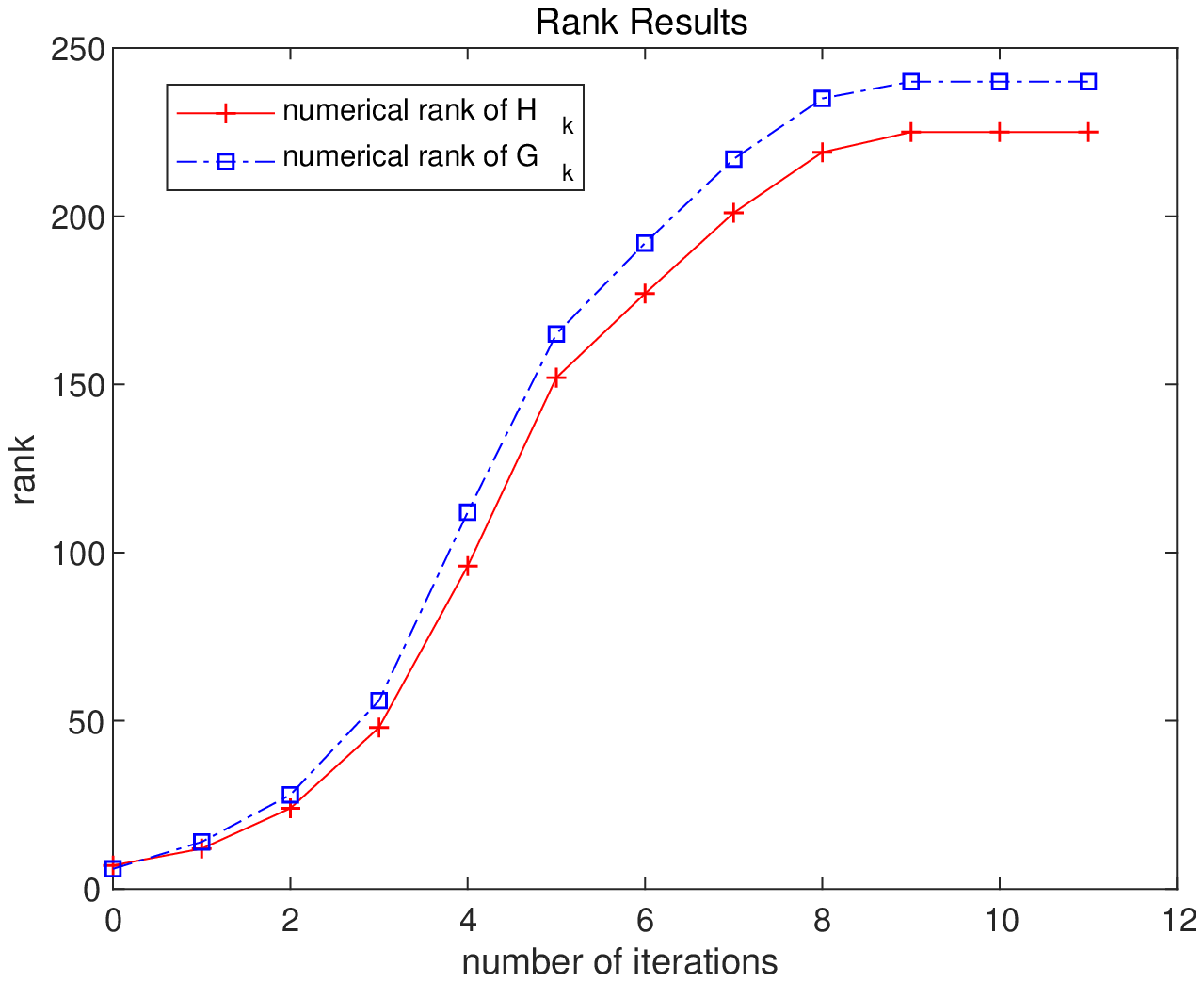}
		}
		\caption{Normalized  residuals and numerical  ranks for $n=1357$}
		\label{fig1357}
	\end{figure*}

	\begin{figure*}[ht]
		\centering
		\subfloat[residuals]{
			\includegraphics[width=2.4in]{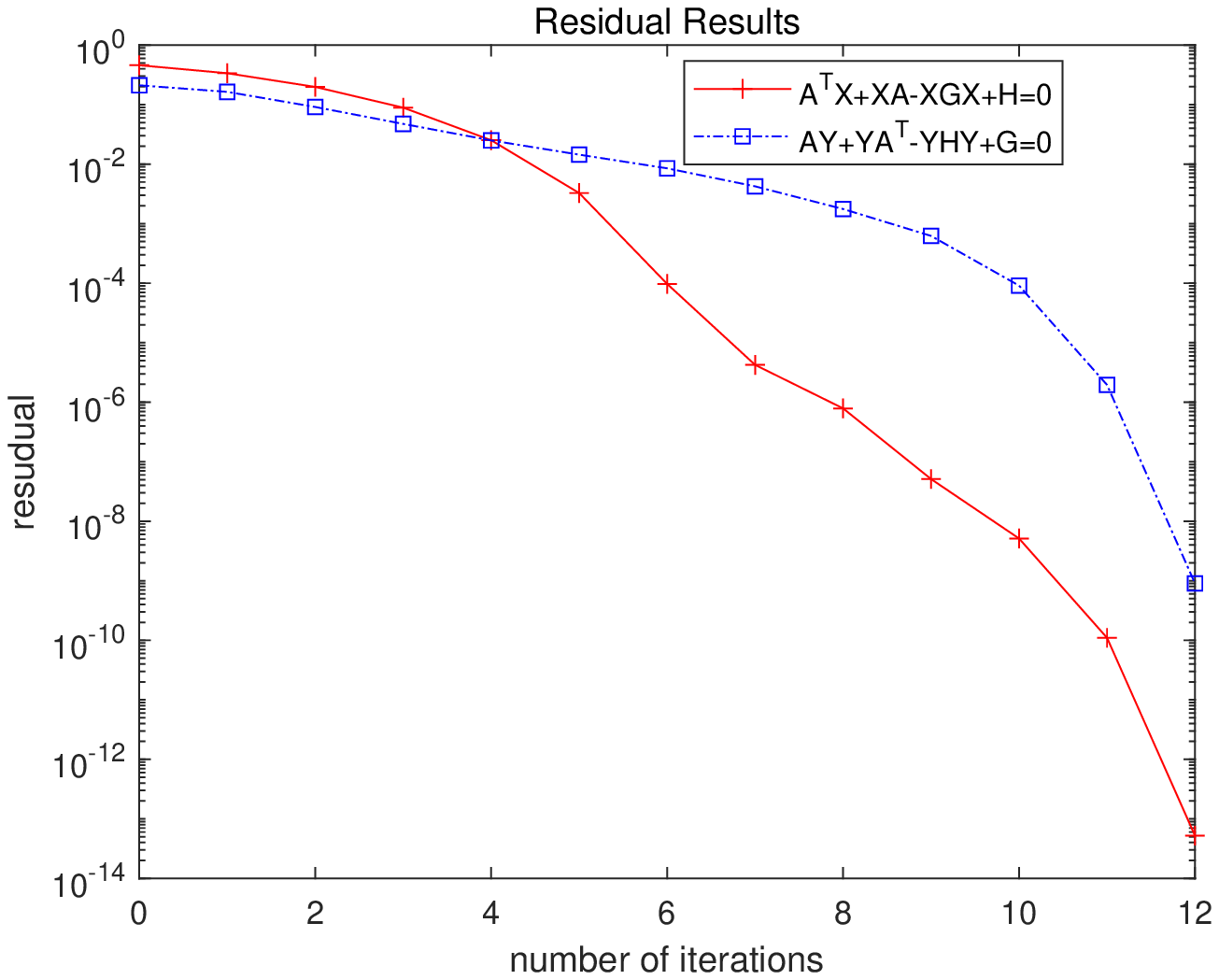}
		}
		\subfloat[ranks]{
			\includegraphics[width=2.4in]{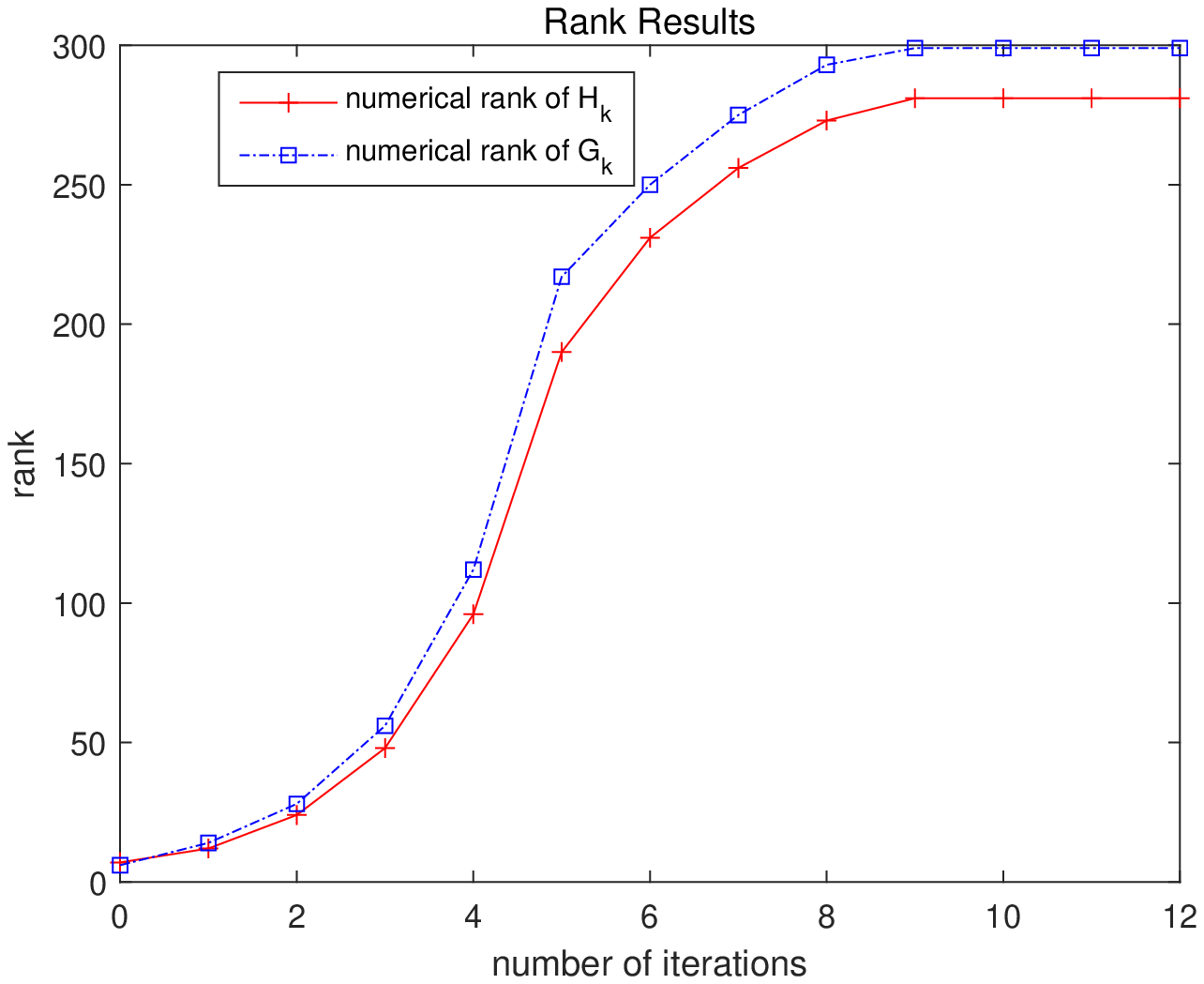}
		}
		\caption{Normalized residuals and numerical  ranks for $n=5177$}
		\label{fig5177}
	\end{figure*}

	\begin{figure*}[ht]
		\centering
		\subfloat[residuals]{
			\includegraphics[width=2.4in]{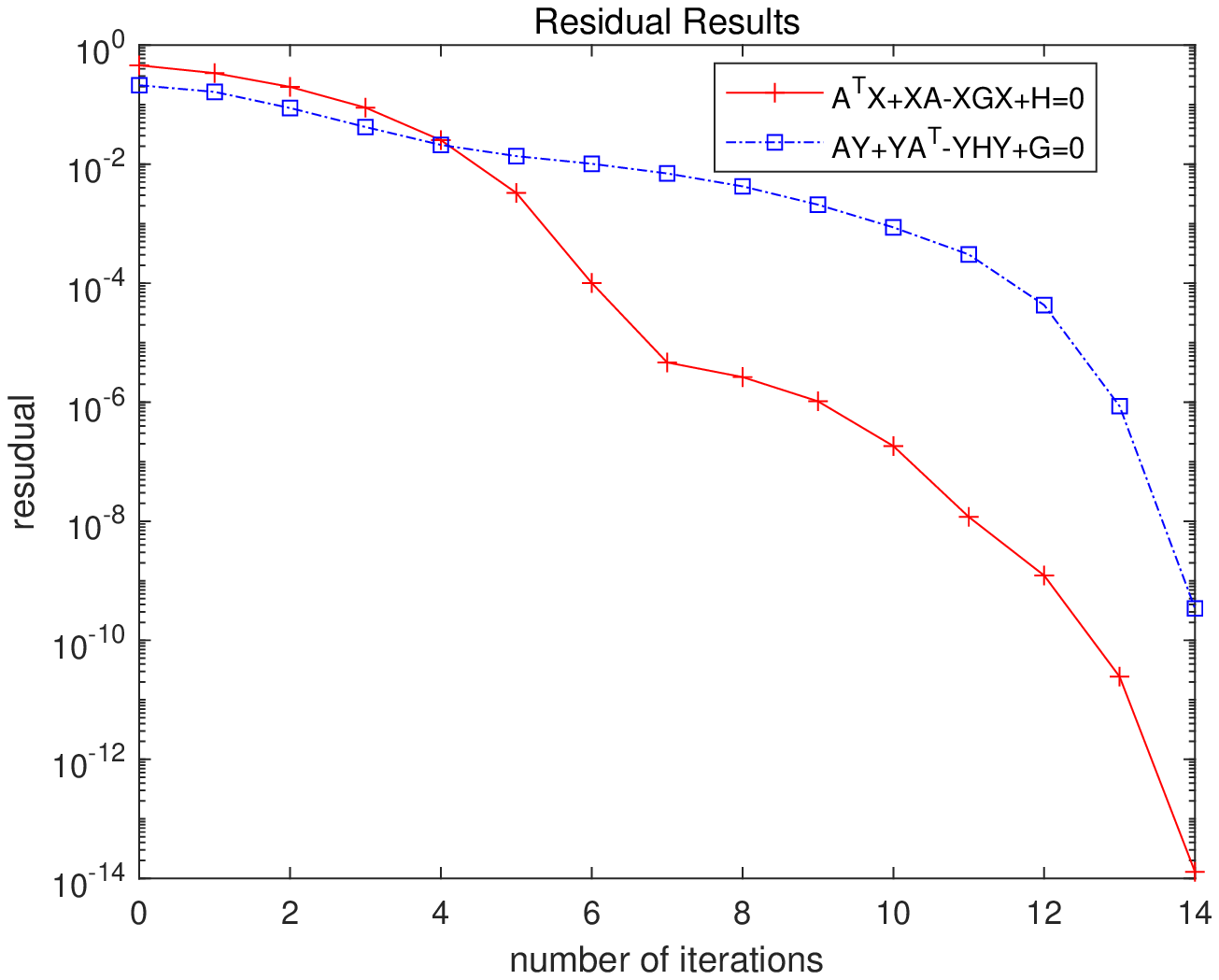}
		}
		\subfloat[ranks]{
			\includegraphics[width=2.4in]{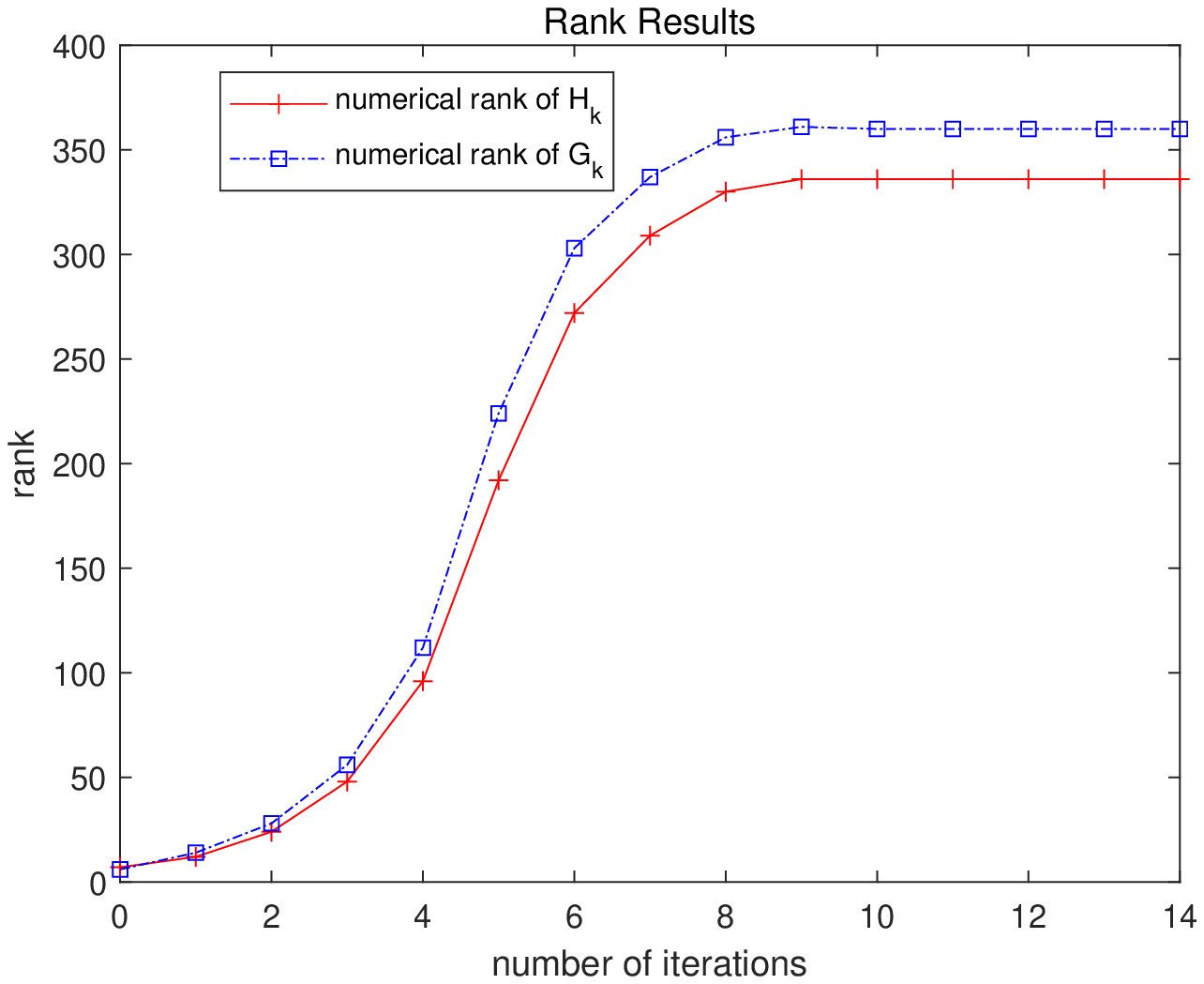}
		}
		\caption{Normalized residuals and numerical ranks for $n=20209$}
		\label{fig20209}
	\end{figure*}

	We compare the efficiency of the dSDA$\rm_t$, RKSM, RADI and NKADI for the three test examples. 
	Table~\ref{table-compare} displays the numerical results produced by the 
	four algorithms, where $r_X$ and ``eTime'' are respectively the rank of 
	the numerical solution and the associated execution time.    

	\begin{table}[t]
		\footnotesize
		\centering
		\begin{tabular}{c|c|c|c|c}
			\hline 
			\multicolumn{5}{c}{dimension $n=1357$}  \\
			\hline
			&dSDA$\rm_t$         & RKSM                    & RADI                    & NKADI                   \\
			\hline 
			$\rho_X$                                & $7.19995\times10^{-15}$ & $1.67961\times10^{-13}$ & $2.43172\times10^{-15}$ & $1.18202\times10^{-15}$ \\
			$r_X$                               & $225$                   & $1147$                  & $210$                   & $173$                   \\
			eTime                                    & $2.09850\times10^{1}$   & $6.98003\times10^{2}$   & $1.53205\times10^{0}$   & $9.21641\times10^{-1}$  \\
			\hline 
			\hline
			\multicolumn{5}{c}{dimension $n=5177$}  \\
			\hline 
			&dSDA$\rm_t$        & RKSM                    & RADI                    & NKADI                   \\
			\hline
			$\rho_X$                                & $5.20068\times10^{-14}$ & $7.79517\times10^{-12}$ & $7.57244\times10^{-14}$ & $1.41009\times10^{-15}$ \\
			$r_X$                               & $281$                   & $777$                   & $216$                   & $206$                   \\
			eTime                                    & $2.39945\times10^{2}$   & $1.99366\times10^{4}$   & $4.06343\times10^{1}$   & $9.93228\times10^{0}$   \\
			\hline 
			\hline
			\multicolumn{5}{c}{dimension $n=20209$} \\
			\hline 
			&dSDA$\rm_t$       & RKSM                    & RADI                    & NKADI                   \\
			\hline 
			$\rho_X$                                & $1.25375\times10^{-14}$ & $1.91847\times10^{-11}$ & $9.25792\times10^{-15}$ & $1.69087\times10^{-15}$ \\
			$r_X$                               & $337$                   & $2630$                  & $276$                   & $222$                   \\
			eTime                                    & $6.73489\times10^{3}$   & $2.60143\times10^{4}$   & $1.46030\times10^{3}$   & $3.28714\times10^{2}$   \\
			\hline 
		\end{tabular}
		\caption{Numerical results from different four methods}
		\label{table-compare}
	\end{table}

		In these three steel  profile cooling  examples, the NKADI performs the best, and our dSDA$\rm_t$   
	    is a little worse than the RADI. However, the ratio of  
		the execution time  for the  dSDA$\rm_t$  and the RADI shows a downtrend as 
		$n$ increases: when $n=1357$, the ratio is $13.6973$; 
	    and for $n=5177$, it is $5.9050$, while for $n=20209$, it declines to $4.6120$.   

	Table~\ref{table-difftol} shows the numerical results produced by  the  dSDA$\rm_t$    
	with five different truncation tolerances, 
	where $tol_j =10^{-(2j+4)}\times  tol$ ($j=1, \cdots, 5$)  
	with ``$tol$" being a  vector  and its entries  
	$tol(i)=\max\{10^{-i}, 10^{-15}\}$ for $i=1, 2, \cdots, 20$.   
	In Table~\ref{table-difftol}, ``iterations" stands for  the required  number of the iterations.    
	It follows from Table~\ref{table-difftol} that with different tolerances in truncation 
	the dSDA$\rm{_t}$ yield similar satisfactory results, meaning that 
	for the three models our dSDA$\rm{_t}$ is insensitive to the truncation tolerance.  
	\begin{table}[t]
		\tiny
		\centering
		\begin{tabular}{c|c|c|c|c|c}
			\hline 
			\multicolumn{6}{c}{$n=1357$}\\
			\hline 
			& $tol_1$ & $tol_2$    &  $tol_3$ & $tol_4$ & $tol_5$\\
			\hline 
			$\rho_X$ & $7.15828\times10^{-15}$ & $7.09397\times10^{-15}$ & $7.55596\times10^{-15}$  & $7.55596\times 10^{-15}$ & $7.55596\times10^{-15}$ \\
			$\rho_Y$ & $1.55751\times10^{-12}$ & $9.65711\times10^{-14}$ & $9.66872\times10^{-14}$  & $9.66872\times 10^{-14}$ & $9.66872\times 10^{-14}$ \\
			$r_X$ &  $218$ & $225$   &$225$ & $225$& $225$ \\
			$r_Y$ &  $230$ & $240$   &$241$ &$241$ & $241$  \\
			iterations &$11$ & $11$ &$11$ & $11$& $11$ \\
			eTime &$2.10654\times 10^{1}$ & $2.22646\times 10^{1}$ & $2.22835\times 10^{1}$ & $2.24233\times 10^{1}$&$2.22558\times 10^{1}$\\
			\hline
			\hline 
			\multicolumn{6}{c}{$n=5177$}\\
			\hline 
			& $tol_1$ & $tol_2$    &  $tol_3$& $tol_4$ & $tol_5$\\
			\hline 
			$\rho_X$ & $5.17750\times10^{-14}$ & $5.18665\times10^{-14}$ & $  5.17608 \times10^{-14}$ & $5.17608 \times10^{-14}$ & $5.17608 \times10^{-14}$  \\
			$\rho_Y$ & $9.03585\times10^{-10}$ & $9.03585\times10^{-10}$ & $   9.03585\times10^{-10}$  &$9.03585\times10^{-10}$ &$9.03585\times10^{-10}$  \\
			$r_X$ &  $265$ & $276$   &$281$ &$281$ & $281$  \\
			$r_Y$ &  $281$ & $298$   &$299$ & $299$ & $299$ \\
			iterations &$12$ &$12$ & $12$&$12$ &$12$\\
			eTime &$2.35127\times 10^{2}$ & $2.45846\times 10^{2}$ & $2.52503\times 10^{2}$ & $2.53253\times 10^{2}$& $ 2.49535\times 10^{2}$\\
			\hline 
			\hline
			\multicolumn{6}{c}{$n=20209$}\\
			\hline 
			& $tol_1$ & $tol_2$    &  $tol_3$& $tol_4$ & $tol_5$\\
			\hline 
			$\rho_X$ & $1.21879\times10^{-14}$ & $1.34771\times10^{-14}$ & $1.37765\times10^{-14}$ & $1.37765\times10^{-14}$ & $1.37765\times10^{-14}$ \\
			$\rho_Y$ & $3.43638\times10^{-10}$ & $3.43638\times10^{-10}$ & $3.43638\times10^{-10}$  & $3.43638\times10^{-10}$ & $3.43638\times10^{-10}$  \\
			$r_X$ &  $321$ & $335$   &$336$ & $336$ & $336$ \\
			$r_Y$ &  $341$ & $358$   &$360$ & $360$ & $360$ \\
			iterations &$14$ &$14$ & $14$& $14$ & $14$\\
			eTime & $7.07405\times10^{3}$ & $7.41030\times10^{3}$ & $7.39546\times10^{3}$ & $7.41256\times10^{3}$ & $7.39824\times10^{3}$\\
			\hline 
		\end{tabular}
		\caption{Numerical results with different truncation tolerances}
		\label{table-difftol}
	\end{table}
\end{example}

\begin{example}\label{eg:random}
	We compare further the dSDA$\rm{_t}$ with the NKADI and RADI. 
	This test  set  includes $1000$  examples,  all of which  
	are randomly generated as follows: firstly we obtain  a nonsingular  
	$X$ by the command \verb|randn| in MATLAB  and two diagonal matrices  
	$\Lambda_1>0, \Lambda_2<0$,  whose sizes respectively are $100$ and $3$.	
	The absolute values of all entries of $\Lambda_1, \Lambda_2$ follow 
	the uniform distribution in the interval $(0,1)$.	
	Then we set $A = \frac{1}{100}  X \diag(\Lambda_1, \Lambda_2) X^{-1}$ 
	and randomly generate $B \in \mathbb{R}^{103 \times 3}, C \in \mathbb{R}^{3\times 103}$ 
	with \verb|randn|, with   $(A, B)$ being stabilizable and $(A, C)$   detectable. 

	For those $1000$ random examples, our dSDA$\rm_t$  and the NKADI, 
	which does not perform the Galerkin projection process, 
	converge and produce low rank solutions. 
	On average, the dSDA$\rm_t$  requires $8.4780$ doubling steps 
	for achieving a normalized  residual smaller than $10^{-13}$, 
	while the NKADI needs $18.5080$ Newton-Kleinman steps. 
	The NKADI with the Galerkin acceleration produces no result, 
	because it fails to solve some projected CAREs. 
	The RADI fails for all these $1000$ random examples, 
	possibly attributable the unstable $A$ or the choices of shifts. 
	In fact, \cite{bennerBKS2018radi} claims that with the same shifts, 
	the RADI and the Incremental Low-Rank Subspace 
	Iteration~\cite{massoudiOT2016analysis} are equivalent. 
	The latter 	achieves convergence 
	when $A$ is stable and satisfies the non-Blaschke condition 
	$\sum_{k=1}^{\infty} \frac{\Re(\alpha_k)}{1+|\alpha_k|^2} = -\infty$, 
	where $\alpha_k$ are the shifts in each iteration. 
	However, in our test set, $A$ are not stable for all randomly generated examples. 

    Next with generated $A, B, C$ as above,  we scale $B$ and $C$ to 
	one tenth of their sizes, and then apply the NKADI and the dSDA$\rm_t$  
	to the randomly generated examples. 
	The NKADI with the Galerkin projection still fails, 
	while the NKADI without the Galerkin step achieves convergence only for $26$ examples,  
	even though the maximum iteration number for the Newton-Kleinman and  the ADI steps 
	are both set as $1000$. 
	In fact, the NKADI is quadratically convergent provided the initial guess $X_0$ is stabilizing. 
	However, for such large random examples, it is difficult to find good initial 
	stabilizing  values of $X_0$. 	In the same $1000$ tests, the dSDA$\rm_t$  is effective for $32$\%
	examples within $9.9718$ iterations, and all convergent examples produce low-rank solutions. 
	For those failed examples, the dSDA$\rm_t$  seems to converge 
	within several iterations, then spin out  of the convergence. 
	We observe that imbalance in entries in some matrices,
	possibly leading to ill-conditioning. A balancing technique may cure the problem 
	but we shall leave this research for the future. 
\end{example}

In summary, Examples~\ref{eg:example-1} and \ref{eg:random} illustrate the efficiency and 
convergence of the dSDA$\rm{_t}$ for large well-conditioned CAREs, with the method occasionally 
outperformed by the NKADI and the RADI for problems with stable $A$. 
However, for  examples with unstable $A$, the dSDA$\rm{_t}$ demonstrates its superiority, 
without any need  for any initial stabilizing $X_0$.   

\section{Conclusions}\label{sec:conclusions}

The classical structure-preserving doubling algorithm (SDA) is an efficient and 
elegant method for computing the unique symmetric 
positive semi-definite solution to CAREs of small and medium sizes. However, for large-scale CAREs,  
it suffers from  high computational costs,  
in terms of execution time and memory  requirement. Fortunately, the decoupled structure-preserving 
doubling algorithm (dSDA) decouples the three iteration  
recursions, thus improving the efficiency of the SDA for CAREs.  
Based on the elegant form of the dSDA,  we propose a novel  
truncation technique, which control the ill-conditioning of the kernels of the approximate solutions and their ranks. 
The resulting algorithm, the truncated dSDA or dSDA$\rm_t$, computes low-rank approximate solutions efficiently. 
Furthermore, we analyze the proposed algorithm and 
prove its  convergence. Numerical experiments illustrate the efficiency of the dSDA$\rm_t$.

\appendix
\section{Proof of Lemma~\ref{lm:L_j}}\label{sec:proof-of-lemma-lm}

We just show the computing  details for $L_2^{G}$ and $L_2^{H}$, from the known $L_1^G$ and $L_1^H$.  
For $L_j^{G}$ and $L_j^{H}$ with $j\ge 3$, the process is similar.  
Since 
\begin{align*}
	E(Y_2^{(1)}) 
	= &\left[I_{4m} + Y_2^{(1)} (Y_2^{(1)})^{\T}\right]^{-1}
	\\
	=&
	\begin{bmatrix}
		I_{2m}&-2\gamma \Gamma \\ 0&I_{2m}
	\end{bmatrix}
	\left[ (I_{2m} + Y_1 Y_1^{\T})^{-1} \oplus  \Psi_1^{-1} \right]
	\begin{bmatrix}
		I_{2m}&0\\-2\gamma \Gamma^{\T} & I_{2m}
	\end{bmatrix},
\end{align*}
where $\Gamma:=(I_{2m}+Y_1Y_1^{\T})^{-1}Y_1 (T_1^{(1)})^{\T}$, 
$T_1^{(1)} = (M_0^G)^{\T} \Omega_1 M_0^H$ and  
$\Psi_1:= I_{2m}+Y_1Y_1^{\T} + 4\gamma^2 T_1^{(1)}(I_{2l} + Y_1^{\T} Y_1)^{-1}(T_1^{(1)})^{\T}$,
then by the definition of $L_2^G$ in \eqref{eq:L-j-UV} and 
\eqref{eq:L-2-12}, we  have 
\begin{align}
	L_2^{G}
	&\equiv 
	2\gamma (\Theta_{1,2}^G)^{\T} M_1^G 
	\left[I_{4m} + Y_2^{(2)}(Y_2^{(2)})^{\T} \right]^{-1}
	Y_2^{(2)}(M_1^H)^{\T} \Theta_{1,2}^H \Omega_2^{\T} \nonumber
	\\  
	&=
	\begin{multlined}[t]
		2\gamma (\Theta_{1,2}^G)^{\T} M_1^G 
		\begin{bmatrix}
			I_{2m}&-2\gamma \Gamma 
			\\ 0&I_{2m}
		\end{bmatrix}
		[ 	
		E(Y_1^{(1)}) 
		\oplus \Psi_1^{-1} 	
		]
		\begin{bmatrix}
			I_{2m}&0\\
			-2\gamma \Gamma^{\T}& I_{2m}
		\end{bmatrix}
		\\
		\cdot
		\begin{bmatrix}
			0 & Y_1^{(1)} \\ Y_1^{(1)} & 2\gamma T_1^{(1)}
		\end{bmatrix}
		(M_1^H)^{\T} \Theta_{1,2}^H \Omega_2^{\T}
	\end{multlined} \nonumber
	\\
	&=
	\begin{multlined}[t]
		2\gamma (\Theta_{1,2}^G)^{\T} 
		\begin{bmatrix}
			I_{r_1^G}&R_{12}^U \\ 0& R_2^U
		\end{bmatrix}
		\begin{bmatrix}
			I_{r_1^G}& -L_1^G \\ 0 & I_{r_1^G}
		\end{bmatrix}
		(I_2 \otimes   M_0^G)
		[ 	
		E(Y_1^{(1)}) 
		\oplus \Psi_1^{-1}	
		] 
		\\
		\cdot
		\begin{bmatrix}
			0 & Y_1^{(1)} \\ Y_1^{(1)} & 
			2\gamma T_1^{(1)}
			F(Y_1^{(1)})
		\end{bmatrix}
		(I_2 \otimes M_0^H)^{\T}  
		\begin{bmatrix}
			I_{r_1^H} & R_{12}^V \\ 0 & R_2^V
		\end{bmatrix}^{\T}
		\Theta_{1,2}^H \Omega_2^{\T}
	\end{multlined}\nonumber
	\\
	&=
	\begin{multlined}[t]
		2\gamma (\Theta_{1,2}^G)^{\T} 
		\begin{bmatrix}
			I_{r_1^G}&R_{12}^U \\ 0& R_2^U
		\end{bmatrix}
		\begin{bmatrix}
			I_{r_1^G}& -L_1^G \\ 0 & I_{r_1^G}
		\end{bmatrix}
		\\
		\cdot
		\begin{bmatrix}
			0 &M_0^G
			E(Y_1^{(1)})
			Y_1^{(1)} (M_0^H)^{\T}
			\\
			M_0^G \Psi_1^{-1} Y_1^{(1)}(M_0^H)^{\T} & 
			2\gamma M_0^G \Psi_1^{-1} T_1^{(1)}
			F(Y_1^{(1)}) 
			(M_0^H)^{\T}	
		\end{bmatrix}
		\\
		\cdot
		\begin{bmatrix}
			I_{r_1^H} & R_{12}^V \\ 0 & R_2^V
		\end{bmatrix}^{\T}
		\Theta_{1,2}^H \Omega_2^{\T}
	\end{multlined}\nonumber
	\\
	&=
	\begin{multlined}[t]
		2\gamma (\Theta_{1,2}^G)^{\T} 
		\begin{bmatrix}
			I_{r_1^G}&R_{12}^U \\ 0& R_2^U
		\end{bmatrix}
		\begin{bmatrix}
			I_{r_1^G}& -L_1^G \\ 0 & I_{r_1^G}
		\end{bmatrix}
		\\
		\cdot
		\begin{bmatrix}
			0 & \Sigma_{1,1}^G K_1^G \Sigma_{1,1}^H 
			\\
			M_0^G \Psi_1^{-1} Y_1^{(1)}(M_0^H)^{\T} & 
			2\gamma M_0^G \Psi_1^{-1} T_1^{(1)}
			F(Y_1^{(1)})
			(M_0^H)^{\T}	
		\end{bmatrix}
		\\
		\cdot
		\begin{bmatrix}
			I_{r_1^H} & R_{12}^V \\ 0 & R_2^V
		\end{bmatrix}^{\T}
		\Theta_{1,2}^H \Omega_2^{\T},
	\end{multlined}
	\label{eq:L_2G}
\end{align}
where $E(Y_1^{(1)}) := [I_{2m} +Y_1^{(1)}(Y_1^{(1)})^{\T}]^{-1}$, 
$F(Y_1^{(1)}): = [ I_{2l}+(Y_1^{(1)})^{\T}Y_1^{(1)}]^{-1}$. 
We next calculate some submatrices in \eqref{eq:L_2G}.    
By \eqref{eq:svd1}, \eqref{eq:L-2-12} and  $M_0^H=(\Theta_{1,1}^H)^{\T}R_1^VP_1^V$, 
we deduce  
\begin{align*}
	&T_1^{(1)} 
	F(Y_1^{(1)})	
	(T_1^{(1)})^{\T}
	=
	(M_0^G)^{\T}
	\Omega_1 
	M_0^H
	F(Y_1^{(1)})
	(M_0^H)^{\T}
	\Omega_1^{\T}
	M_0^G
	\\
	=&
	(M_0^G)^{\T}
	\Omega_1 
	M_0^H
	V_1^Y  
	\left[I_{2l} + (\Sigma_1^Y)^{\T} \Sigma_1^Y\right]^{-1} 
	(V_1^Y)^{\T}
	(M_0^H)^{\T}
	\Omega_1^{\T} 
	M_0^G
	\\
	=&
	(M_0^G)^{\T} \Omega_1 
	(\Theta_{1,1}^H)^{\T} R_1^V P_1^V 
	V_1^Y 
	\left[I_{2l} + (\Sigma_1^Y)^{\T} \Sigma_1^Y\right]^{-1} 
	(V_1^Y)^{\T}
	(P_1^V)^{\T} (R_1^V)^{\T} \Theta_{1,1}^H
	\Omega_1^{\T} M_0^G
	\\
	=&
	(M_0^G)^{\T}
	\Omega_1 (\Theta_{1,1}^H)^{\T} 
	\Theta_1^H (\Sigma_1^H)^2  (\Theta_1^H)^{\T} 
	\Theta_{1,1}^H \Omega_1^{\T} 
	M_0^G
	\equiv  
	(M_0^G)^{\T}
	\Omega_1
	(\Sigma_{1,1}^H)^2 \Omega_1^{\T} 
	M_0^G,
\end{align*}
implying that 
	$\Psi_1 = I_{2m}+Y_1Y_1^{\T} + 4\gamma^2 	
	(M_0^G)^{\T} \Omega_1
	(\Sigma_{1,1}^H)^2 \Omega_1^{\T} 
	M_0^G$. 
Because 
\begin{align*}
	M_0^G E(Y_1^{(1)}) 
	(M_0^G)^{\T} 
	&= (\Theta_{1,1}^G)^{\T} \Theta_1^G (\Sigma_1^G)^2 (\Theta_1^G)^{\T} \Theta_{1,1}^G
	\equiv (\Sigma_{1,1}^G)^2, 
	\\ 
	M_0^H F(Y_1^{(1)}) 
	(M_0^H)^{\T} 
	&= (\Theta_{1,1}^H)^{\T} \Theta_1^H (\Sigma_1^H)^2 (\Theta_1^H)^{\T} \Theta_{1,1}^H 
	\equiv (\Sigma_{1,1}^H)^2,
\end{align*}
then by \eqref{eq:L-2-12}  and the SMWF, we have  
\begin{align}
	&
	M_0^G \Psi_1 ^{-1} Y_1^{(1)}  (M_0^H)^{\T}   \nonumber
	\\
	=&
	M_0^G 
	\left[
		I_{2m} + Y_1^{(1)}(Y_1^{(1)})^{\T} 
		+ 4\gamma^2 (M_0^G)^{\T} 
		\Omega_1(\Sigma_{1,1}^H)^2 \Omega_1^{\T} 
		M_0^G
	\right]^{-1}
	Y_1^{(1)} (M_0^H)^{\T} \nonumber 
	\\
	=&
	\begin{multlined}[t]
		M_0^G
		E(Y_1^{(1)}) 
		Y_1^{(1)}
		(M_0^H)^{\T}
		- 4\gamma^2 M_0^G
		E(Y_1^{(1)}) 
		(M_0^G)^{\T} \Omega_1 
		\\
		\cdot
		\left[ \left(\Sigma_{1,1}^H\right)^{-2} 
			+ 4\gamma^2 \Omega_1^{\T}  
			M_0^G
			E(Y_1^{(1)}) 
			(M_0^G)^{\T}\Omega_1 
		\right]^{-1}
		\Omega_1^{\T} M_0^G
		E(Y_1^{(1)}) 
		Y_1^{(1)} (M_0^H)^{\T}
	\end{multlined}\nonumber
	\\
	=& 
	\Sigma_{1,1}^G K_1^G \Sigma_{1,1}^H 
	-4\gamma^2 (\Sigma_{1,1}^G)^2
	\Omega_1
	\left[\left(\Sigma_{1,1}^H\right)^{-2} + 
	4\gamma^2 \Omega_1^{\T} (\Sigma_{1,1}^G)^2  \Omega_1\right]^{-1} 
	\Omega_1^{\T} \Sigma_{1,1}^G K_1^G \Sigma_{1,1}^H
	\nonumber
	\\
	=&
	\left\{ I_{r_1^G} - 4\gamma^2 (\Sigma_{1,1}^G)^2 \Omega_1 
		\left[ \left(\Sigma_{1,1}^H\right)^{-2} +
		4\gamma^2 \Omega_1^{\T} (\Sigma_{1,1}^G)^2  \Omega_1 \right]^{-1}
		\Omega_1^{\T}
	\right\}  \Sigma_{1,1}^G K_1^G \Sigma_{1,1}^H
	\nonumber
	\\
	=&
	\begin{multlined}[t]
		\left\{ I_{r_1^G} 	- 4\gamma^2 (\Sigma_{1,1}^G)^2 \Omega_1 \Sigma_{1,1}^H 
		[  I_{r_1^H}  +4\gamma^2  \Sigma_{1,1}^H \Omega_1^{\T} 
		(\Sigma_{1,1}^G)^2  \Omega_1 \Sigma_{1,1}^H ]^{-1} \Sigma_{1,1}^H
	\Omega_1^{\T}\right\}  
	\Sigma_{1,1}^G K_1^G \Sigma_{1,1}^H 
	\end{multlined}\nonumber  
	\\
	=&
	\begin{multlined}[t]
		\left\{
			I_{r_1^G} 	- 4\gamma^2 [  I_{r_1^G}  +4\gamma^2  
				(\Sigma_{1,1}^G)^2 \Omega_1 (\Sigma_{1,1}^H)^2
				\Omega_1^{\T} ]^{-1} (\Sigma_{1,1}^G)^2  \Omega_1 (\Sigma_{1,1}^H)^2
		\Omega_1^{\T}\right\}  	
		\Sigma_{1,1}^G K_1^G \Sigma_{1,1}^H
	\end{multlined} \nonumber  
	\\
	=&
	\left[I_{r_1^G} + 4\gamma^2 (\Sigma_{1,1}^G)^2 \Omega_1 (\Sigma_{1,1}^H)^2 \Omega_1^{\T}
	\right]^{-1}\Sigma_{1,1}^G K_1^G \Sigma_{1,1}^H \nonumber
	\\
	=&
	\Sigma_{1,1}^G \left[I_{r_1^G} + 4\gamma^2 \Sigma_{1,1}^G \Omega_1 
		(\Sigma_{1,1}^H)^2 \Omega_1^{\T} \Sigma_{1,1}^G 
	\right]^{-1} K_1^G \Sigma_{1,1}^H \nonumber
	\\
	=&
	\Sigma_{1,1}^G 
	\left[I_{r_1^G} + 4\gamma^2 U_2^Y \Sigma_2^Y 
		(\Sigma_2^Y)^{\T} (U_2^Y)^{\T} 
	\right]^{-1} K_1^G \Sigma_{1,1}^H \nonumber
	\\
	=&
	\Sigma_{1,1}^G U_2^Y 
	\left(I_{r_1^G}+4\gamma^2 \Sigma_2^Y(\Sigma_2^Y)^{\T}\right)^{-1}
	(U_2^Y)^{\T} K_1^G \Sigma_{1,1}^H \nonumber 
	\\
	=& 
	\Sigma_{1,1}^G U_2^Y 
	(\Upsilon_2^G)^{-1}
	(U_2^Y)^{\T} K_1^G \Sigma_{1,1}^H 
	:=
	Z_{21} 
	\equiv 
	\Sigma_{1,1}^G U_2^Y 
	(\Upsilon_2^G)^{-1/2}
	\widehat{Z}_{21}, \label{eq:L-2-21} 
\end{align}
where  $\Upsilon_2^G = I_{r_1^G}+4\gamma^2 \Sigma_2^Y(\Sigma_2^Y)^{\T}$ and 
$\widehat{Z}_{21}:= (\Upsilon_2^G)^{-1/2}(U_2^Y)^{\T} K_1^G \Sigma_{1,1}^H$.  
We also have 
\begin{align}
	&
	M_0^G \Psi_1^{-1} 
	T_1^{(1)} 
	F(Y_1^{(1)}) 
	(M_0^H)^{\T} \nonumber 
	\\
	=&
	M_0^G
	\left[
		I_{2m}+Y_1^{(1)} (Y_1^{(1)})^{\T} 
		+ 4\gamma^2  (M_0^G)^{\T}
		\Omega_1 (\Sigma_{1,1}^H)^2 \Omega_1^{\T} 
		M_0^G
	\right]^{-1} 
	T_1^{(1)} 
	F(Y_1^{(1)}) 
	(M_0^H)^{\T}	\nonumber 
	\\
	=& 
	\begin{multlined}[t]
		M_0^G E(Y_1^{(1)}) 
		T_1^{(1)} F(Y_1^{(1)}) 
		(M_0^H)^{\T}
		-4\gamma^2 	
		M_0^G E(Y_1^{(1)}) 
		(M_0^G)^{\T} \Omega_1 
		\\
		\cdot
		\left[ \left(\Sigma_{1,1}^H\right)^{-2} + 4\gamma^2 \Omega_1^{\T} 
			M_0^G	E(Y_1^{(1)}) 
		(M_0^G)^{\T}\Omega_1  \right]^{-1}
		\Omega_1^{\T} M_0^G
		E(Y_1^{(1)}) 
		T_1^{(1)} F(Y_1^{(1)}) 
		(M_0^H)^{\T}
	\end{multlined} \nonumber 
	\\
	=&
	\begin{multlined}[t]
		M_0^G
		E(Y_1^{(1)}) 
		(M_0^G)^{\T}
		\Omega_1 
		M_0^H F(Y_1^{(1)})
		(M_0^H)^{\T}
		- 4\gamma^2  (\Sigma_{1,1}^G)^2 \Omega_1
		\\
		\cdot
		\left[ \left(\Sigma_{1,1}^H\right)^{-2} + 4\gamma^2 \Omega_1^{\T} (\Sigma_{1,1}^G)^2 
			\Omega_1 
		\right]^{-1} \Omega_1^{\T} 
		M_0^G E(Y_1^{(1)}) 
		(M_0^G)^{\T} \Omega_1 M_0^H F(Y_1^{(1)}) 
		(M_0^H)^{\T}
	\end{multlined}\nonumber
	\\
	=&
	\begin{multlined}[t]
		(\Sigma_{1,1}^G)^2 \Omega_1 (\Sigma_{1,1}^H)^2  - 4\gamma^2 (\Sigma_{1,1}^G)^2 \Omega_1 
	\left[ (\Sigma_{1,1}^H )^{-2} + 4\gamma^2 \Omega_1^{\T}  (\Sigma_{1,1}^G)^2 \Omega_1
	\right]^{-1}
	\\
	\cdot
	\Omega_1^{\T}  (\Sigma_{1,1}^G)^2 \Omega_1
	(\Sigma_{1,1}^H)^2
	\end{multlined}  \nonumber
	\\
	=&
	\left[I_{r_1^G} + 4\gamma^2 (\Sigma_{1,1}^G)^2 \Omega_1 (\Sigma_{1,1}^H)^2 
	\Omega_1^{\T} \right]^{-1} 
	(\Sigma_{1,1}^G)^2 \Omega_1 (\Sigma_{1,1}^H)^2 \nonumber
	\\
	=&
	\Sigma_{1,1}^G \left[I_{r_1^G} + 4\gamma^2 \Sigma_{1,1}^G \Omega_1 (\Sigma_{1,1}^H)^2 
	\Omega_1^{\T} \Sigma_{1,1}^G \right]^{-1} 
	\Sigma_{1,1}^G \Omega_1 (\Sigma_{1,1}^H)^2 \nonumber
	\\
	= &
	\Sigma_{1,1}^G U_2^Y
	(\Upsilon_2^G)^{-1}
	(U_2^Y)^{\T} \Sigma_{1,1}^G \Omega_1 (\Sigma_{1,1}^H)^2 
	:=
	Z_{22}
	\equiv 
	\Sigma_{1,1}^G U_2^Y
	(\Upsilon_2^G)^{-1/2}
	\widehat{Z}_{22}, \label{eq:L-2-22}
\end{align}
where $\widehat{Z}_{22}:= (\Upsilon_2^G)^{-1/2}
	(U_2^Y)^{\T} \Sigma_{1,1}^G \Omega_1 (\Sigma_{1,1}^H)^2$. 
By~\eqref{eq:L_2G}, \eqref{eq:L-2-21} and~\eqref{eq:L-2-22}  
it consequently holds that 
\begin{align*}
	L_2^G
	&=
	2\gamma (\Theta_{1,2}^G)^{\T} 
	\begin{bmatrix}
		I_{r_1^G}&R_{12}^U \\ 0& R_2^U
	\end{bmatrix}
	\begin{bmatrix}
		I_{r_1^G}& -L_1^G \\ 0 & I_{r_1^G}
	\end{bmatrix}
	\begin{bmatrix}
		0 & \Sigma_{1,1}^G K_1^G \Sigma_{1,1}^H
		\\
		Z_{21} &  2\gamma Z_{22} 
	\end{bmatrix}
	\begin{bmatrix}
		I_{r_1^H} & R_{12}^V \\ 0 & R_2^V
	\end{bmatrix}^{\T}
	\Theta_{1,2}^H \Omega_2^{\T}
	\\
	&=
	\begin{multlined}[t]
		2\gamma (\Theta_{1,2}^G)^{\T} 
		\begin{bmatrix}
			I_{r_1^G}&R_{12}^U \\ 0& R_2^U
		\end{bmatrix}
		\begin{bmatrix}
			I_{r_1^G} & -L_1^G \\ 0 &I_{r_1^G}
		\end{bmatrix}
		\left[
			\Sigma_{1,1}^G  
			\oplus 
			\Sigma_{1,1}^G U_2^Y 
			(\Upsilon_2^G)^{-1/2}
		\right] \\ 
		\cdot \begin{bmatrix}
			0& K_1^G\Sigma_{1,1}^H \\ 
			\widehat{Z}_{21}&
		2\gamma \widehat{Z}_{22}\end{bmatrix}
		\begin{bmatrix}
			I_{r_1^H} & R_{12}^V \\ 0 & R_2^V
		\end{bmatrix}^{\T}
		\Theta_{1,2}^H \Omega_2^{\T}.
	\end{multlined}
\end{align*}
Moreover, by 
\eqref{eq:svd2_H}, we get 
\begin{align*}
	L_2
	&= 
	2\gamma (\Theta_{1,2}^G)^{\T} 
	\Theta_2^G \Sigma_2^G (\Phi_2^G)^{\T} 
	\begin{bmatrix}
		0& K_1^G\Sigma_{1,1}^H \\ 
		\widehat{Z}_{21}&
	2\gamma \widehat{Z}_{22}\end{bmatrix}
	\begin{bmatrix}
		I_{r_1^H} & R_{12}^V \\ 0 & R_2^V
	\end{bmatrix}^{\T}
	\Theta_{1,2}^H \Omega_2^{\T}
	\\
	&=
	\begin{multlined}[t]
		2\gamma \Sigma_{1,2}^G  (\Phi_{1,2}^G)^{\T}
		\begin{bmatrix}
			0& K_1^G\Sigma_{1,1}^H \\ 
			\widehat{Z}_{21}&
			2\gamma \widehat{Z}_{22}
		\end{bmatrix}
		\begin{bmatrix}
			I_{r_1^H} & L_1^H \\ 0 & I_{r_1^H} 
		\end{bmatrix}^{\T}  
		\\
		\cdot
			\left[(\Sigma_{1,1}^H)^{-1} \oplus 
			(\Sigma_{1,1}^H)^{-1}V_2^Y
			(\Upsilon_2^H)^{1/2}\right]
			\left[\Sigma_{1,1}^H \oplus 
			(\Upsilon_2^H)^{-1/2}
			(V_2^Y)^{\T} \Sigma_{1,1}^H\right]
		\\
		\cdot
		\begin{bmatrix}
			I_{r_1^H} & -L_1^H \\ 0 & I_{r_1^H} 
		\end{bmatrix}^{\T}  
		\begin{bmatrix}
			I_{r_1^H} & R_{12}^V \\ 0 & R_2^V
		\end{bmatrix}^{\T}
		\Theta_{1,2}^H \Omega_2^{\T}
	\end{multlined}
	\\
	&=
	\begin{multlined}[t]
		2\gamma \Sigma_{1,2}^G  (\Phi_{1,2}^G)^{\T}
		\begin{bmatrix}
			0& K_1^G\Sigma_{1,1}^H \\ 
			\widehat{Z}_{21}&
			2\gamma \widehat{Z}_{22}
		\end{bmatrix}
		\begin{bmatrix}
			I_{r_1^H} & L_1^H \\ 0 & I_{r_1^H} 
		\end{bmatrix}^{\T}  
		\\
		\cdot
		\left[(\Sigma_{1,1}^H)^{-1} \oplus 
			(\Sigma_{1,1}^H)^{-1}V_2^Y
			(\Upsilon_2^H)^{1/2}\right]
		\Phi_2^H \Sigma_2^H (\Theta_2^H)^{\T} 
		\Theta_{1,2}^H \Omega_2^{\T}
	\end{multlined}
	\\
	&=
	\begin{multlined}[t]
		2\gamma \Sigma_{1,2}^G  (\Phi_{1,2}^G)^{\T}
		\begin{bmatrix}
			0& K_1^G\Sigma_{1,1}^H \\ 
			\widehat{Z}_{21}&
			2\gamma \widehat{Z}_{22}
		\end{bmatrix}
		\begin{bmatrix}
			I_{r_1^H} & L_1^H \\ 0 & I_{r_1^H} 
		\end{bmatrix}^{\T}  
		\\
		\cdot
			\left[(\Sigma_{1,1}^H)^{-1} \oplus 
			(\Sigma_{1,1}^H)^{-1}V_2^Y
			(\Upsilon_2^H)^{1/2} \right]
		\Phi_{1,2}^H \Sigma_{1,2}^H \Omega_2^{\T}
	\end{multlined}
\end{align*}
where $\Upsilon_2^H= I_{r_1^H}+4\gamma^2 (\Sigma_2^Y)^{\T}\Sigma_2^Y$.

We further  deduce that
	$K_1^G \Sigma_{1,1}^H(L_1^H)^{\T}(\Sigma_{1,1}^H)^{-1}
	= 
	2\gamma K_1^G \Sigma_{1,1}^H \Omega_1^{\T} \Sigma_{1,1}^G K_1^G$,
\begin{align*}
	&
	\widehat{Z}_{21}(\Sigma_{1,1}^H)^{-1} + 
	2\gamma \widehat{Z}_{22}(L_1^H)^{\T}(\Sigma_{1,1}^H)^{-1}
	\\
	=&
	\begin{multlined}[t]
		(\Upsilon_2^G)^{-1/2}
		(U_2^Y)^{\T} K_1^G 
		+4\gamma^2 
		(\Upsilon_2^G)^{-1/2}
		(U_2^Y)^{\T} 
		\Sigma_{1,1}^G \Omega_1
		(\Sigma_{1,1}^H)^2  \Omega_1^{\T} \Sigma_{1,1}^G K_1^G 
	\end{multlined}
	\\
	=&
	(\Upsilon_2^G)^{-1/2}
	(U_2^Y)^{\T}
	\left[I_{r_1^G} + 4\gamma^2 \Sigma_{1,1}^G \Omega_1 
	(\Sigma_{1,1}^H)^2 \Omega_1^{\T} \Sigma_{1,1}^G\right] K_1^G 
	\\
	=&
	(\Upsilon_2^G)^{-1/2}
	(U_2^Y)^{\T} U_2^Y	
	\Upsilon_2^G
	(U_2^Y)^{\T} K_1^G
	\equiv 
	(\Upsilon_2^G)^{1/2}
	(U_2^Y)^{\T} K_1^G, 
\end{align*}
and 
\begin{align*}
	&\widehat{Z}_{22}(\Sigma_{1,1}^H)^{-1} V_2^Y 
	(\Upsilon_2^H)^{1/2}
	=
	(\Upsilon_2^G)^{-1/2}
	(U_2^Y)^{\T} \Sigma_{1,1}^G \Omega_1 \Sigma_{1,1}^H 
	V_2^Y 
	(\Upsilon_2^H)^{1/2}
	\\
	=& 
	(\Upsilon_2^G)^{-1/2}
	(U_2^Y)^{\T} 
	U_2^Y \Sigma_2^Y (V_2^Y)^{\T} V_2^Y 
	(\Upsilon_2^H)^{1/2}
	\equiv 
	\Sigma_2^Y. 
\end{align*}
Hence, we obtain 
\begin{align*}
	L_2^{G}
	&=
		\begin{multlined}[t]
		2\gamma \Sigma_{1,2}^G (\Phi_{1,2}^G)^{\T} 
		\begin{bmatrix}
			2\gamma K_1^G \Sigma_{1,1}^H \Omega_1^{\T} \Sigma_{1,1}^G K_1^G 
			& K_1^G V_2^Y 
			(\Upsilon_2^H)^{1/2}
			\\
			(\Upsilon_2^G)^{1/2}
			(U_2^Y)^{\T} K_1^G 
			& 2\gamma \Sigma_2^Y
		\end{bmatrix}
		\Phi_{1,2}^H \Sigma_{1,2}^H \Omega_2^{\T}
	\end{multlined}
	\\
	&\equiv
	2\gamma \Sigma_{1,2}^G 
	K_2^G  \Sigma_{1,2}^H \Omega_2^{\T}, 
\end{align*}
where 
	$K_2^G:= 
	(\Phi_{1,2}^G)^{\T}
	\begin{bmatrix}
		2\gamma K_1^G \Sigma_{1,1}^H \Omega_1^{\T} \Sigma_{1,1}^G K_1^G 
		& K_1^G V_2^Y (\Upsilon_2^H)^{1/2} 
		\\
		(\Upsilon_2^G)^{1/2} (U_2^Y)^{\T} K_1^G 
		& 2\gamma \Sigma_2^Y
	\end{bmatrix}
	\Phi_{1,2}^H$. 

By the same manipulations we also obtain 
$L_2^H = 2\gamma \Sigma_{1,2}^H  
K_2^H  \Sigma_{1,2}^G \Omega_2$ 
with $K_2^H\equiv (K_2^G)^{\T}$. 

\section{Proof of Lemma~\ref{lm:doubling_k1}}\label{sec:proof-of-lemma-ref-lm-doubling_k1}

	For (\romannumeral1), substituting the SVD of $Y_1$ and    
	\eqref{eq:svd1} into $A_1$ and $A_1^{(1)}$ gives  
	\begin{align*}
		&\N{A_1^{(1)} - A_1} 
		\\
		=&
		\begin{multlined}[t]
		2\gamma \|Q_1^U \Theta_1^G \Sigma_1^G (\Phi_1^G)^{\T} \Sigma_1^Y 
			\Phi_1^H \Sigma_1^H (\Theta_1^H)^{\T}  (Q_1^V)^{\T}  
			\\
			- Q_1^U \Theta_{1,1}^G \Sigma_{1,1}^G (\Phi_{1,1}^G)^{\T} 
		\Sigma_1^Y \Phi_{1,1}^H \Sigma_{1,1}^H (\Theta_{1,1}^H)^{\T} (Q_1^V)^{\T} \|
		\end{multlined}		\\
		=&
		2\gamma \N{\Theta_1^G \Sigma_1^G (\Phi_1^G)^{\T} \Sigma_1^Y 
			\Phi_1^H \Sigma_1^H (\Theta_1^H)^{\T} - \Theta_{1,1}^G 
			\Sigma_{1,1}^G (\Phi_{1,1}^G)^{\T} 
		\Sigma_1^Y \Phi_{1,1}^H \Sigma_{1,1}^H (\Theta_{1,1}^H)^{\T}}
		\\
		=& 
		2\gamma \N{\Theta_1^G \Sigma_1^G (\Phi_1^G)^{\T} \Sigma_1^Y 
			\Phi_{2,1}^H \Sigma_{2,1}^H (\Theta_{2,1}^H)^{\T} + 
			\Theta_{2,1}^G \Sigma_{2,1}^G (\Phi_{2,1}^G)^{\T} \Sigma_1^Y 
		\Phi_{1,1}^H \Sigma_{1,1}^H (\Theta_{1,1}^H)^{\T}}
		\\
		\leq &
		2\gamma \N{\Theta_1^G \Sigma_1^G (\Phi_1^G)^{\T} \Sigma_1^Y 
		\Phi_{2,1}^H \Sigma_{2,1}^H (\Theta_{2,1}^H)^{\T} }  
		+ 2\gamma \N{\Theta_{2,1}^G \Sigma_{2,1}^G (\Phi_{2,1}^G)^{\T} 
		\Sigma_1^Y \Phi_{1,1}^H \Sigma_{1,1}^H (\Theta_{1,1}^H)^{\T}}
		\\
		\leq &
		2\gamma \left(\N{\Sigma_{1,1}^G} \N{\Sigma_1^Y} \N{\Sigma_{2,1}^H} + 
		\N{\Sigma_{2,1}^G} \N{\Sigma_1^Y} \N{\Sigma_{1,1}^H}\right)
		\leq  
		4\gamma \varepsilon_1 \N{\Sigma_{1,1}^G} \N{\Sigma_1^Y} \N{\Sigma_{1,1}^H}.
	\end{align*}

	For (\romannumeral2),  by the definitions of $A_{s+1}^{(s)}$ (in~\eqref{eq:A-k-j}) 
	and $A_{s+1}^{(s+1)}$ (in \eqref{eq:Aj+1}),  we have 
	\begin{align*}
		A_{s+1}^{(s)}
		&=
		\begin{multlined}[t]
			\widetilde{A}_{\gamma}^{2^{s+1}} - 2\gamma 
			\left[ \mathcal{Q}_s^U, \, \widetilde{A}_{\gamma}^{2^s} \mathcal{Q}_s^U \right]
			\left[I_2 \otimes \left((\Theta_{1,s}^G)^{\T} M_{s-1}^G \right) \right]
			E(Y_{s+1}^{(s)}) 
			Y_{s+1}^{(s)}
			\\
			\cdot
			\left[I_2 \otimes\left((M_{s-1}^H)^{\T}  \Theta_{1,s}^H\right) \right]
			\left[ \mathcal{Q}_s^V,\, (\widetilde{A}_{\gamma}^{\T})^{2^s} \mathcal{Q}_s^V
			\right]^{\T}
		\end{multlined}
		\\
		&=
		\begin{multlined}[t]
			\widetilde{A}_{\gamma}^{2^{s+1}} - 2\gamma 
			\left[
				\mathcal{Q}_s^U, \widetilde{A}_{\gamma}^{2^s} \mathcal{Q}_s^U
			\right] 
			\begin{bmatrix}
				I_{r_s^G} & -L_s^G \\ 0 & I_{r_s^G}
			\end{bmatrix}
			\\
			\cdot
			\left[I_2 \otimes\left((\Theta_{1,s}^G)^{\T} M_{s-1}^G \right) \right]
			[ E(Y_{s}^{(s)}) \oplus \Psi_s^{-1} ] 
			\begin{bmatrix}
				0 &Y_s^{(s)} \\
				Y_s^{(s)} & 2\gamma T_s^{(s)}
				F(Y_{s}^{(s)})
			\end{bmatrix}
			\\
			\cdot
			\left[I_2 \otimes\left((M_{s-1}^H)^{\T}  \Theta_{1,s}^H\right) \right]
			\left[
				\mathcal{Q}_s^V, (\widetilde{A}_{\gamma}^{\T})^{2^s} \mathcal{Q}_s^V
			\right]^{\T},
		\end{multlined}
	\end{align*}
	\begin{align*} 
		A_{s+1}^{(s+1)} 
		&=
		\widetilde{A}_{\gamma}^{2^{s+1}} - 2\gamma \mathcal{Q}_{s+1}^U (\Theta_{1,s+1}^G)^{\T} M_s^G
		E(Y_{s+1}^{(s)})
		Y_{s+1}^{(s)}(M_s^H)^{\T} 
		\Theta_{1,s+1}^H (\mathcal{Q}_{s+1}^V)^{\T}
		\\
		&=\begin{multlined}[t]
			\widetilde{A}_{\gamma}^{2^{s+1}} - 2\gamma \mathcal{Q}_{s+1}^U (\Theta_{1,s+1}^G)^{\T} 
			\begin{bmatrix}
				I_{r_s^G} & R_{12}^{s,U}\\ 0& R_2^{s,U}
			\end{bmatrix}
			\left[I_2\otimes \left((\Theta_{1,s}^G)^{\T} M_{s-1}^G\right)\right]
			E(Y_{s+1}^{(s)})
			\\
			\cdot 
			Y_{s+1}^{(s)}
			\left[I_2\otimes \left((M_{s-1}^H)^{\T} \Theta_{1,s}^H \right)\right] 
			\begin{bmatrix}
				I_{r_s^H}& R_{12}^{s,V} \\ 0 & R_2^{s,V}
			\end{bmatrix}^{\T} 
			\Theta_{1,s+1}^H(\mathcal{Q}_{s+1}^V)^{\T}, 
		\end{multlined}
	\end{align*}
	where $E(Y_{s+1}^{(s)}):= [I_{2^{s+1}m} + Y_{s+1}^{(s)}(Y_{s+1}^{(s)})^{\T}]^{-1}$, 
	$E(Y_s^{(s)}):= [I_{2^s m}+Y_s^{(s)}(Y_s^{(s)})^{\T}]^{-1}$,  
	$\Psi_s:= I_{2^s m}+Y_s^{(s)}(Y_s^{(s)})^{\T} + 
	4\gamma^2 T_s^{(s)} F(Y_s^{(s)}) (T_s^{(s)})^{\T}$, 
	$F(Y_s^{(s)}):= [I_{2^s l}+(Y_s^{(s)})^{\T}Y_s^{(s)}]^{-1}$.  	

	Next we reformulate  $A_{s+1}^{(s)}$ and $A_{s+1}^{(s+1)}$. 
	By  the SMWF,  
	\eqref{eq:Sigma_j_H},  
	\eqref{eq:L-j-UV} (the definition of $L_s^G$),  the SVD of 
	$\Sigma_{1,s}^G \Omega_s \Sigma_{1,s}^H$ and  \eqref{eq:Lj_GH}, we have    
	\begin{align*}
		(\Theta_{1,s}^G)^{\T} M_{s-1}^G
		E(Y_{s}^{(s)})
		Y_s^{(s)}
		(M_{s-1}^H)^{\T} \Theta_{1,s}^H
		= \Sigma_{1,s}^G  K_s^G  \Sigma_{1,s}^H,
	\end{align*}
	\begin{align*}
		& (\Theta_{1,s}^G)^{\T} M_{s-1}^G \Psi_s^{-1}
		Y_s^{(s)} 
		(M_{s-1}^H)^{\T} \Theta_{1,s}^H \\ 
		=&
		\begin{multlined}[t] 
			(\Theta_{1,s}^G)^{\T} M_{s-1}^G 
			[
				I_{2^sm}+Y_s^{(s)}(Y_s^{(s)})^{\T} 
				+ 4\gamma^2 (M_{s-1}^G)^{\T} \Theta_{1,s}^G \Omega_s
			(\Sigma_{1,s}^H)^2 \Omega_s^{\T} (\Theta_{1,s}^G)^{\T} M_{s-1}^G ]^{-1} 
			\\
			\cdot 
			Y_s^{(s)}(M_{s-1}^H)^{\T} \Theta_{1,s}^H
		\end{multlined}
		\\
		=& 
		\begin{multlined}[t]
			(\Theta_{1,s}^G)^{\T} M_{s-1}^G 
			E(Y_{s}^{(s)})
			Y_s^{(s)}(M_{s-1}^H)^{\T} \Theta_{1,s}^H 
			-4\gamma^2(\Theta_{1,s}^G)^{\T} M_{s-1}^G 
			E(Y_{s}^{(s)})
			(M_{s-1}^G)^{\T} \Theta_{1,s}^G
			\\
			\cdot
			\left[I_{r_s^G}  + 4\gamma^2 \Omega_s (\Sigma_{1,s}^H)^2 \Omega_s^{\T} 
				(\Theta_{1,s}^G)^{\T} M_{s-1}^G 
				E(Y_{s}^{(s)})
			(M_{s-1}^G)^{\T} \Theta_{1,s}^G\right]^{-1}
			\\
			\cdot \Omega_s (\Sigma_{1,s}^H)^2 \Omega_s^{\T} (\Theta_{1,s}^G)^{\T} M_{s-1}^G 
			E(Y_{s}^{(s)})
			Y_s^{(s)}(M_{s-1}^H)^{\T} \Theta_{1,s}^H
		\end{multlined}
		\\
		=&
		\begin{multlined}[t]
			\Sigma_{1,s}^G  K_s^G  \Sigma_{1,s}^H
		- 4\gamma^2 (\Sigma_{1,s}^G)^2 \left[I_{r_s^G}+
			4\gamma^2 \Omega_s(\Sigma_{1,s}^H)^2 \Omega_s^{\T} 
		(\Sigma_{1,s}^G)^2\right]^{-1}
		\\
		\cdot
		\Omega_s(\Sigma_{1,s}^H)^2 \Omega_s^{\T} 
		\Sigma_{1,s}^G  K_s^G \Sigma_{1,s}^H
		\end{multlined}	
		\\
		=&
		\left[I_{r_s^G}+4\gamma^2(\Sigma_{1,s}^G)^2 \Omega_s
		(\Sigma_{1,s}^H)^2 \Omega_s^{\T}\right]^{-1}
		\Sigma_{1,s}^G K_s^G  \Sigma_{1,s}^H
		\\
		=&
		\Sigma_{1,s}^G U_{s+1}^Y 
		(\Upsilon_{s+1}^G)^{-1}
		(U_{s+1}^Y)^{\T}   K_s^G \Sigma_{1,s}^H, 
	\end{align*}
	\begin{align*}
		&
		(\Theta_{1,s}^G)^{\T} M_{s-1}^G \Psi_s^{-1}
		T_s^{(s)} 
		F(Y_{s}^{(s)})
		(M_{s-1}^H)^{\T} \Theta_{1,s}^H
		\\
		=&
		\begin{multlined}[t]
			(\Theta_{1,s}^G)^{\T} M_{s-1}^G
			[
				I_{2^sm}+Y_s^{(s)}(Y_s^{(s)})^{\T} 
				+ 4\gamma^2 (M_{s-1}^G)^{\T} \Theta_{1,s}^G \Omega_s
			(\Sigma_{1,s}^H)^2 \Omega_s^{\T} (\Theta_{1,s}^G)^{\T} M_{s-1}^G  ]^{-1}
			\\
			\cdot (M_{s-1}^G)^{\T} \Theta_{1,s}^G \Omega_s  (\Sigma_{1,s}^H)^2
		\end{multlined}		
		\\
		=&
		\begin{multlined}[t]
			(\Theta_{1,s}^G)^{\T} M_{s-1}^G	
			E(Y_{s}^{(s)})
			(M_{s-1}^G)^{\T} \Theta_{1,s}^G \Omega_s (\Sigma_{1,s}^H)^2 
			- 4\gamma^2 (\Theta_{1,s}^G)^{\T} M_{s-1}^G	
			E(Y_{s}^{(s)})
			\\
			\cdot
			(M_{s-1}^G)^{\T} \Theta_{1,s}^G   
			\left[I_{r_s^G}+4\gamma^2 \Omega_s(\Sigma_{1,s}^H)^2 
				\Omega_s^{\T} (\Theta_{1,s}^G)^{\T} M_{s-1}^G
				E(Y_{s}^{(s)})
				(M_{s-1}^G)^{\T} \Theta_{1,s}^G
			\right]^{-1}
			\\
			\cdot
			\Omega_s (\Sigma_{1,s}^H)^2 
			\Omega_s^{\T} (\Theta_{1,s}^G)^{\T}
		     M_{s-1}^G 	
			E(Y_{s}^{(s)})
			(M_{s-1}^G)^{\T}  \Theta_{1,s}^G \Omega_s(\Sigma_{1,s}^H)^2
		\end{multlined}	
		\\
		=&
		\begin{multlined}[t]
			(\Sigma_{1,s}^G)^2 \Omega_s (\Sigma_{1,s}^H)^2 - 4\gamma^2 (\Sigma_{1,s}^G)^2 
			\left[I_{r_s^G}+4\gamma^2 \Omega_s(\Sigma_{1,s}^H)^2 \Omega_s^{\T} 
			(\Sigma_{1,s}^G)^2\right]^{-1}
			\\
			\cdot
			\Omega_s (\Sigma_{1,s}^H)^2 \Omega_s^{\T} (\Sigma_{1,s}^G)^2 
			\Omega_s (\Sigma_{1,s}^H)^2
		\end{multlined}
		\\
		=&
		\left[I_{r_s^G}+4\gamma^2 (\Sigma_{1,s}^G)^2 \Omega_s(\Sigma_{1,s}^H)^2 \Omega_s^{\T} 
		\right]^{-1} (\Sigma_{1,s}^G)^2 \Omega_s (\Sigma_{1,s}^H)^2
		\\
		=&
		\Sigma_{1,s}^G U_{s+1}^Y  
		(\Upsilon_{s+1}^G)^{-1}
		(U_{s+1}^Y)^{\T} \Sigma_{1,s}^G \Omega_s (\Sigma_{1,s}^H)^2
		\\
		=&
		\Sigma_{1,s}^G U_{s+1}^Y  
		(\Upsilon_{s+1}^G)^{-1}
		\Sigma_{s+1}^Y (V_{s+1}^Y)^{\T} \Sigma_{1,s}^H,  
	\end{align*}
	where $\Upsilon_{s+1}^G:= I_{r_s^G} +4\gamma^2 \Sigma_{s+1}^Y (\Sigma_{s+1}^Y)^{\T}$. 

	Furthermore, with 
	\begin{align*}
		K_s^G	\Sigma_{1,s}^H (L_s^H)^{\T} (\Sigma_{1,s}^H)^{-1} 
		&=
		2\gamma  K_s^G \Sigma_{1,s}^H  
		\Omega_s^{\T} \Sigma_{1,s}^G  K_s^G 
		\equiv
		2\gamma  K_s^G  V_{s+1}^Y (\Sigma_{s+1}^Y)^{\T} 
		(U_{s+1}^Y)^{\T}  K_s^G, 
	\end{align*}
	\begin{align*}
		&
		(\Upsilon_{s+1}^G)^{-1/2}
		(U_{s+1}^Y)^{\T}   K_s^G  + 
		2\gamma 
		(\Upsilon_{s+1}^G)^{-1/2}
		\Sigma_{s+1}^Y
		(V_{s+1}^Y)^{\T} \Sigma_{1,s}^H (L_s^H)^{\T} (\Sigma_{1,s}^H)^{-1}
		\\
		=& 
		\begin{multlined}[t]
			(\Upsilon_{s+1}^G)^{-1/2}
			(U_{s+1}^Y)^{\T}   K_s^G 
			+ 
			4\gamma^2 	
			(\Upsilon_{s+1}^G)^{-1/2}
			\Sigma_{s+1}^Y
			(\Sigma_{s+1}^Y)^{\T} 
			(U_{s+1}^Y)^{\T} K_s^G  
		\end{multlined}
		\\
		=&
		(\Upsilon_{s+1}^G)^{1/2}
		(U_{s+1}^Y)^{\T}   K_s^G,
	\end{align*}
	and  the abbreviation  
	\begin{align*}
		Z:=\begin{bmatrix}
			2\gamma  K_s^G  V_{s+1}^Y (\Sigma_{s+1}^Y)^{\T} 
			(U_{s+1}^Y)^{\T}  K_s^G 
			& K_s^G V_{s+1}^Y
			(\Upsilon_{s+1}^H)^{1/2}
			\\
			(\Upsilon_{s+1}^G)^{1/2}
			(U_{s+1}^Y)^{\T}  K_s^G 
			& 2\gamma \Sigma_{s+1}^Y 
		\end{bmatrix}, 
	\end{align*}
	substituting  the  above results into the expressions 
	for $A_{s+1}^{(s)}$ and $A_{s+1}^{(s+1)}$, we obtain 
	\begin{align*}
		A_{s+1}^{(s)}
		&=
		\begin{multlined}[t]
			\widetilde{A}_{\gamma}^{2^{s+1}} - 2\gamma \left[
				\mathcal{Q}_s^U, \widetilde{A}_{\gamma}^{2^s} \mathcal{Q}_s^U
			\right]
			\begin{bmatrix}
				I_{r_s^G} & -L_s^G \\ 0 & I_{r_s^G}
			\end{bmatrix}
			\left[ 
				\Sigma_{1,s}^G \oplus 
				\Sigma_{1,s}^G U_{s+1}^Y 
				(\Upsilon_{s+1}^G)^{-1/2}
			\right] 
			\\
			\cdot
			\begin{bmatrix}
				0& K_s^G\Sigma_{1,s}^H \\ 
				(\Upsilon_{s+1}^G)^{-1/2}
				(U_{s+1}^Y)^{\T} K_s^G \Sigma_{1,s}^H &
				2\gamma 
				(\Upsilon_{s+1}^G)^{-1/2}
				\Sigma_{s+1}^Y(V_{s+1}^Y)^{\T} \Sigma_{1,s}^H
			\end{bmatrix}
			\\
			\cdot
			\left[
				\mathcal{Q}_s^V, (\widetilde{A}_{\gamma}^{\T})^{2^s} \mathcal{Q}_s^V
			\right]^{\T}
		\end{multlined}
		\\
		&=
		\begin{multlined}[t]
			\widetilde{A}_{\gamma}^{2^{s+1}} - 2\gamma \left[
				\mathcal{Q}_s^U, Q_{s+1}^U
			\right]  \Theta_{s+1}^G \Sigma_{s+1}^G (\Phi_{s+1}^G)^{\T} 
			\\
			\cdot 
			\begin{bmatrix}
				0& K_s^G\Sigma_{1,s}^H \\ 
				(\Upsilon_{s+1}^G)^{-1/2}
				(U_{s+1}^Y)^{\T} K_s^G \Sigma_{1,s}^H &
				2\gamma 
				(\Upsilon_{s+1}^G)^{-1/2}
				\Sigma_{s+1}^Y(V_{s+1}^Y)^{\T} \Sigma_{1,s}^H
			\end{bmatrix}
			\\
			\cdot
			\left[
				\mathcal{Q}_s^V, (\widetilde{A}_{\gamma}^{\T})^{2^s} \mathcal{Q}_s^V
			\right]^{\T}
		\end{multlined}
		\\
		&=
		\begin{multlined}[t]
			\widetilde{A}_{\gamma}^{2^{s+1}} - 2\gamma \left[
				\mathcal{Q}_s^U, Q_{s+1}^U
			\right]  \Theta_{s+1}^G \Sigma_{s+1}^G (\Phi_{s+1}^G)^{\T} 
			Z \Phi_{s+1}^H \Sigma_{s+1}^H (\Theta_{s+1}^H)^{\T} 
			\left[
				\mathcal{Q}_s^V, Q_{s+1}^V
			\right]^{\T},
		\end{multlined}
	\end{align*}
	\begin{align*}
		A_{s+1}^{(s+1)}
		&=
		\begin{multlined}[t]
			\widetilde{A}_{\gamma}^{2^{s+1}} - 2\gamma 
		\mathcal{Q}_{s+1}^U (\Theta_{1,s+1}^G)^{\T} 
		\Theta_{s+1}^G \Sigma_{s+1}^G (\Phi_{s+1}^G)^{\T} Z
		\Phi_{s+1}^H \Sigma_{s+1}^H (\Theta_{s+1}^H)^{\T} 
		\\
		\cdot
		\Theta_{1,s+1}^H (\mathcal{Q}_{s+1}^V)^{\T}
		\end{multlined}		
		\\
		&=
		\widetilde{A}_{\gamma}^{2^{s+1}} - 2\gamma 
		\mathcal{Q}_{s+1}^U \Sigma_{1,s+1}^G (\Phi_{1,s+1}^G)^{\T} Z
		\Phi_{1,s+1}^H \Sigma_{1,s+1}^H  (\mathcal{Q}_{s+1}^V)^{\T}. 
	\end{align*}
	Then we have the difference   
	\begin{align*}
		&A_{s+1}^{(s+1)} - A_{s+1}^{(s)} 
		\\
		=&
		\begin{multlined}[t]
			2\gamma \left\{
				\left[
					\mathcal{Q}_s^U, Q_{s+1}^U
				\right]  \Theta_{s+1}^G \Sigma_{s+1}^G (\Phi_{s+1}^G)^{\T} 
				Z \Phi_{s+1}^H \Sigma_{s+1}^H (\Theta_{s+1}^H)^{\T} 
				\left[
					\mathcal{Q}_s^V, Q_{s+1}^V
			\right]^{\T} \right.
			\\
			\left. -
				\mathcal{Q}_{s+1}^U \Sigma_{1,s+1}^G (\Phi_{1,s+1}^G)^{\T} Z
				\Phi_{1,s+1}^H \Sigma_{1,s+1}^H (\mathcal{Q}_{s+1}^V)^{\T}
			\right\}
		\end{multlined}
		\\
		\equiv&
		\begin{multlined}[t]
			2\gamma 
			\left\{
				\left[
					\mathcal{Q}_s^U, Q_{s+1}^U
				\right]
				\Theta_{s+1}^G \Sigma_{s+1}^G (\Phi_{s+1}^G)^{\T} Z
				\Phi_{2,s+1}^H \Sigma_{2,s+1}^H (\Theta_{2,s+1}^H)^{\T} 
				\left[
					\mathcal{Q}_s^V, Q_{s+1}^V
			\right]^{\T} \right. 
			\\
			\left. +   
				\left[
					\mathcal{Q}_s^U, Q_{s+1}^U
				\right]
				\Theta_{2,s+1}^G
				\Sigma_{2,s+1}^G 
				(\Phi_{2,s+1}^G)^{\T} Z 
				\Phi_{1,s+1}^H \Sigma_{1,s+1}^H (\Theta_{1,s+1}^H)^{\T}
				\left[
					\mathcal{Q}_s^V, Q_{s+1}^V
				\right]^{\T}
			\right\},
		\end{multlined}
	\end{align*}
	leading to 
	\begin{align}
		\N{A_{s+1}^{(s+1)} - A_{s+1}^{(s)}} 
		&\leq 
		2\gamma \left(
			\N{\Sigma_{s+1}^G} \N{\Sigma_{2,s+1}^H} \N{Z}
		+ \N{\Sigma_{2,s+1}^G} \N{\Sigma_{1,s+1}^H} \N{Z}\right) \nonumber
		\\
		& \leq
		4\gamma \varepsilon_{s+1} \N{\Sigma_{1,s+1}^G} \N{\Sigma_{1,s+1}^H} \N{Z}. 
		\label{eq:normA_j+1}
	\end{align}
	Since
	\[
		Z = 
			\left[K_s^G V_{s+1}^Y \oplus
			I_{r_s^G}\right]
		\begin{bmatrix}
			2\gamma(\Sigma_{s+1}^Y)^{\T} & (\Upsilon_{s+1}^H)^{1/2}
			\\
			(\Upsilon_{s+1}^G)^{1/2} & 2\gamma \Sigma_{s+1}^Y
		\end{bmatrix}
			\left[(U_{s+1}^Y)^{\T} K_s^G \oplus 
			I_{r_s^H}\right]
	\]
	and 
	\[
		\N*{\begin{bmatrix}
				2\gamma (\Sigma_{s+1}^Y)^{\T}
				& 
				(\Upsilon_{s+1}^H)^{1/2}
				\\
				(\Upsilon_{s+1}^G)^{1/2}
				& 2\gamma\Sigma_{s+1}^Y
			\end{bmatrix}
		}
		\leq 
		2\gamma\N{\Sigma_{s+1}^Y} +\sqrt{1+4\gamma^2 \N{\Sigma_{s+1}^Y}^2},
	\]
	it then holds that  
	\begin{align}\label{eq:normZ}
		\N{Z}  \leq 
		\max\left\{1, \N{K_s^G}^2 \right\}  
		\left(2\gamma\N{\Sigma_{s+1}^Y} +\sqrt{1+4\gamma^2 \N{\Sigma_{s+1}^Y}^2}\right)\equiv \kappa_s.
	\end{align}
	Substituting \eqref{eq:normZ} into \eqref{eq:normA_j+1}  yields 
	the desired result. 

\section*{Acknowledgments}
Part of the work was completed when the first three authors visited the ST Yau Research Centre 
at the National Chiao Tung University, Hsinchu, Taiwan. 
The first author is supported in part by NSFC-11901290 and Fundamental Research Funds 
for the Central Universities, and the third author is supported in part by NSFC-11901340.

\bibliographystyle{siamplain}
\bibliography{sdals}

\end{document}